\documentclass[11pt,table,x11names]{article}

\usepackage[T1]{fontenc}
\usepackage{times}
\usepackage[utf8]{inputenc}
\usepackage[english]{babel}
\usepackage{graphicx}
\usepackage[margin=1in]{geometry}
\usepackage{mathtools,amssymb,amsthm} 
\numberwithin{equation}{section}
\usepackage{url}

\usepackage[numbers,sort&compress]{natbib}
\usepackage{titlesec}
\usepackage{todonotes}
\usepackage{csquotes}
\usepackage{caption} 
\usepackage{subcaption}
\usepackage{xfrac}
\usepackage{nicefrac}
\usepackage{tikz}
\usepackage{pgfplots}

\pgfplotsset{compat=1.17}
\usetikzlibrary{intersections, patterns, decorations.pathreplacing,arrows.meta,shapes.arrows}
\PassOptionsToPackage{hyphens}{url}\usepackage[hypertexnames=false]{hyperref}
\usepackage[nameinlink, capitalise, noabbrev]{cleveref}
\crefname{assumption}{Assumption}{Assumptions}
     \newtheorem{theorem}{Theorem}[section]
\newtheorem{proposition}[theorem]{Proposition}
\newtheorem{lemma}[theorem]{Lemma}
\newtheorem{corollary}[theorem]{Corollary}
\theoremstyle{remark}
\newtheorem{remark}{Remark}[section]
\theoremstyle{definition}

\newtheorem{definition}{Definition}[section]
\newtheorem{assumption}{Assumption}
\numberwithin{equation}{section}
  
\newcommand{\N}{\mathbb{N}}

\newcommand{\R}{\mathbb{R}}

\newcommand{\E}{\mathcal{E}}
\newcommand{\F}{\mathcal{F}}

\newcommand{\norm}[1]{\left\Vert #1 \right\Vert}
\newcommand{\abs}[1]{\left\vert #1 \right\vert}
 
\DeclareMathOperator{\divtmp}{div}
\renewcommand{\div}{\divtmp}

\DeclareMathOperator{\esssup}{ess\,sup}
\newcommand{\st}{\,:\,}
\DeclareMathOperator{\supp}{supp}

\newcommand{\de}{\mathrm{d}}
\renewcommand{\d}{\,\mathrm{d}}

\DeclareMathOperator{\spann}{span}
\newcommand{\eps}{\varepsilon}
\DeclareMathOperator{\dist}{dist}
\DeclareMathOperator{\trace}{Tr}
\DeclareMathOperator{\Lip}{Lip}

\newcommand{\M}{\mathcal M}
\newcommand{\grad}{\nabla}

\renewcommand{\deg}{\operatorname{deg}}

\newcommand{\wsto}{\overset{\ast}{\rightharpoonup}}

\newcommand{\restr}{\,\mathbin{\vrule height 1.6ex depth 0pt width
0.13ex\vrule height 0.13ex depth 0pt width 1.3ex}}
\newcommand{\vol}{\mathrm{Vol}}
\newcommand{\sind}[1]{^{(#1)}}
\usepackage{nicefrac}

\usepackage{titlesec}
\titlespacing*{\paragraph}{\parindent}{2pt}{1em}
  \newenvironment{listi}
  {\begin{list} 
 {(\roman{broj})}
{ \usecounter{broj}}
       \setlength{\labelwidth}{25pt}
  }
{   \end{list} }
\newcounter{broj}
  \newcommand{\te}{\textrm}
     \newcommand{\red}{\color{red}}
\renewcommand{\red}{}     
\newcommand{\nc}{\normalcolor}

\definecolor{mygreen}{rgb}{0.1,0.75,0.2}

\definecolor{dred}{rgb}{0.6,0,0}
\definecolor{lgray}{rgb}{0.5,0.5,0.5}
\definecolor{dgrn}{rgb}{0.1,0.6,0.1}

\def\Xint#1{\mathchoice
   {\XXint\displaystyle\textstyle{#1}}%
   {\XXint\textstyle\scriptstyle{#1}}%
   {\XXint\scriptstyle\scriptscriptstyle{#1}}%
   {\XXint\scriptscriptstyle\scriptscriptstyle{#1}}%
   \!\int}
\def\XXint#1#2#3{{\setbox0=\hbox{$#1{#2#3}{\int}$}
     \vcenter{\hbox{$#2#3$}}\kern-.5\wd0}}

\def\dashint{\Xint-}

\newcommand{\keywords}[1]{\par\smallskip\noindent\textbf{Keywords. }#1\par}
\newcommand{\msc}[1]{\par\noindent\textbf{MSC2020 classes. }#1\par}

\begin{document}

\title{Convergence of graph Dirichlet energies and graph Laplacians on intersecting manifolds of varying dimensions}
\author{Leon Bungert\thanks{%
Institute of Mathematics, Center for Artificial Intelligence and Data Science (CAIDAS), University of Würzburg, Emil-Fischer-Str. 40, 97074 Würzburg, Germany. Email: \href{mailto:leon.bungert@uni-wuerzburg.de}{leon.bungert@uni-wuerzburg.de}
}
\and
Dejan Slep\v{c}ev\thanks{Department of Mathematical Sciences, Carnegie Mellon University, Pittsburgh, United States, Email: \href{mailto:slepcev@math.cmu.edu}{slepcev@math.cmu.edu}}
} 

\maketitle

\begin{abstract}
We study  $\Gamma$-convergence of graph Dirichlet energies and spectral convergence of graph Laplacians on unions of intersecting manifolds of potentially different dimensions. 
Our investigation is motivated by problems of machine learning, as real-world data often consist of parts or classes with different intrinsic dimensions.
An important challenge is to understand which machine learning methods adapt to such varied dimensionalities. 
We investigate the standard unnormalized and the normalized graph Dirichlet energies.
We show that the unnormalized energy and its associated graph Laplacian asymptotically only see the variations within the manifold of the highest dimension. On the other hand, we prove that the  normalized Dirichlet energy converges to a (tensorized) Dirichlet energy on the union of manifolds that adapts to all dimensions simultaneously. We also establish the related spectral convergence and present a few numerical experiments to illustrate our findings. 
 \end{abstract}

\msc{49J55, 49J45, 60D05, 49R50, 68R10, 62G20}

\tableofcontents

\section{Introduction}

\subsection{Motivation and Discussion}

We investigate  $\Gamma$-convergence of graph Dirichlet energies and spectral convergence of graph Laplacians on unions of intersecting manifolds of potentially different dimensions. In the process, we also refine the results on 
$\Gamma$-convergence of nonlocal Dirichlet energies on metric measure spaces, when specialized to unions of intersecting manifolds. 
The motivation for our work arises in machine learning where graph Laplacians are used  to capture and utilize the intrinsic geometry of the data \cite{belkin2001laplacian,belkin2003laplacian,bruna2013spectral,CoifmanLafon05DM,defferrard2016convolutional,ng2001spectral,vonLuxburg2007Tutorial,zhou20gnn}. In many (real life) data sets the intrinsic dimensionality of components or clusters may differ \cite{allegra19, brown2023verifying,medina2019heuristic,carter10} and, furthermore, the clusters can have some intersection. For example, in sets of hand-written digits the intrinsic dimensionality of sets of digit 4 is greater than that of digit 1. Furthermore, there is typically some overlap between sets of digits 1 and 7. 
This led to the development of specific algorithms tailored to discover the multi-manifold structure of the data \cite{wang2015multi,Arias-Castro_Chen_Lerman_2011,goldberg2009multi,wang2010multi,wang2011spectral,chen2025robust,babaeian2015multiple,chen2023largest,gong2012robust}.

While there have been many theoretical studies of graph Laplacians and their convergence, including \cite{WormellReich21, ChengWu22, CalderGTLewicka22, CalderGT22, ArmstrongVenkatraman25optimal, GTLi2025Venkatraman25minimax, hein2007graph, vonLuxburg2007Tutorial, vLBeBo08, singer2006graph, SinWu13, garcia2019variational, garcia2020error,calder2023rates,bungert2024poisson}, there are only a few works \cite{trillos2023large,ullrich2024} that theoretically analyze Laplacian-based algorithms for data of different intrinsic dimensionality. In particular, none of the rigorous works investigated the asymptotics of the standard  normalized graph Laplacians on unions of manifolds.
Here we show that the asymptotic behavior, as the number of available data points grows, of different, popular, forms of graph Laplacians greatly differs when applied to data sampled from manifolds of different dimensions. 
 We investigate the standard unnormalized and normalized graph Laplacians, 
 and show that the unnormalized graph Laplacian and the associated unnormalized Dirichlet energy asymptotically only see the variations within the manifold of the highest dimension, \cref{thm:unnormalized_Gamma,thm:unnormalized_spectral}. On the other hand, we prove that the normalized Dirichlet energy converges to a  Dirichlet energy on the union of manifolds that adapts to all dimensions simultaneously, \cref{thm:gamma,thm:spectral_normalized}.

Our results have important consequences for the understanding of graph Laplacian based  machine learning algorithms. Namely, the following are desirable properties of algorithms: (i) they capture information about parts of all dimensions present and (ii) separate smooth  manifolds that intersect at a positive angle, as these are thought to represent different data clusters. We show that the standard normalized graph Laplacian has both of these properties, except in the case that intersecting manifolds have the same intrinsic dimension and that the intersection is of codimension one, when (ii) fails. We note that having such intersection would be a very unusual situation for data lying in a high-dimensional ambient space, as is typically the case for images and other data classes. 
On the other hand, the standard, unnormalized, graph Laplacian only sees the variations in data of the highest dimension. Thus, our results provide a strong argument for using the normalized graph Laplacian over the unnormalized one.

A consequence of our results is that the normalized graph Laplacian can be used to separate data from manifolds of potentially different dimensions. We note that there exists a rich literature on algorithms for clustering data from multiple manifolds of different dimensions, which includes statistical guarantee. For recent results and a careful overview of the literature see \cite{aamari2024stratificationlearning}.
Establishing statistical guarantees for clustering algorithms based on the normalized graph Laplacian at finite sample size is an interesting direction for future work.

\paragraph{Outline.} 
We review the literature on convergence of graph Laplacians in \cref{sec:literature-graphs} and on convergence of nonlocal Laplacians in \cref{sec:literature-nonlocal}.
We introduce the setup and state the main results in \cref{sec:setup+main}. 
In particular, we describe the geometric setup in \cref{sec:unions}, graph construction in \cref{sec:graph_intro}, introduce the graph and nonlocal Dirichlet energies in \cref{sec:normalized_energies}, and state the results on $\Gamma$-convergence (\cref{thm:gamma}) and compactness (\cref{thm:compactness})  for graph-based Dirichlet energies in \cref{sec:main_results}. 
The result on spectral convergence of the normalized graph Laplacian (\cref{thm:spectral_normalized}) is stated and proved in \cref{sec:normalized_EVconvergence}, while the one for the unnormalized graph Laplacian (\cref{thm:unnormalized_spectral}) is established  in \cref{sec:unnormalized_EVconvergence}. 
In \cref{sec:numerics} we illustrate our theoretical findings with numerically computed graph Laplacian spectra. 
Results on $\Gamma$-convergence (\cref{thm:Gamma_NL}) and compactness (\cref{thm:compactness_NL}) for nonlocal Dirichlet energies are presented and proved in \cref{sec:nonlocal}.
Finally, the proofs of our main results,  \cref{thm:compactness} and \cref{thm:gamma}, are presented in \cref{sec:proofs_graphs}.

In the appendix we establish or recall some of the auxiliary facts. Specifically in \cref{sec:appendix_angle} we characterize the angle between intersecting manifolds, in \cref{sec:appendix_TLp} we recall the $TL^p$ metric that allows one to compare the discrete and continuum settings, in \cref{sec:appendix_Gamma} we recall existing results on $\Gamma$-convergence of graph Laplacians for data on manifolds, in \cref{sec:appendix_trace} we show that the boundedness of the energies as $n \to \infty$ in the codimension-one case implies that the desired trace condition is satisfied at the intersection,   in
\cref{sec:appendix_smooth_approx} we establish density of smooth functions in Sobolev space $H^1$ on union of intersecting manifolds in the setting where the manifolds are of same dimension and the intersection is a codimension-one submanifold, and in \cref{sec:appendix_Lip_dens} we show the density of Lipschitz functions in the same space, but for the case that the codimension of the intersection in at least one of the manifolds is two or more.

\subsection{Related literature} \label{sec:literature}
\subsubsection{Graph Laplacians} \label{sec:literature-graphs}

Graph Laplacians play an important role in machine learning.
Most data encountered in practice are high dimensional: signals include many measurements and  images have many pixels. 
Fortunately, most data possess some intrinsic low-dimensional structure that makes machine learning tasks tractable with a limited amount of data. To access and encode the low-dimensional structure in the data, a large family of algorithms is based on first creating a graph by connecting nearby data points. 
The graphs provide an implicit description of the geometry of the data which, in turn, enables computationally efficient approaches to many tasks. In this work, we study commonly used weighted geometric graphs, where edge weights are a function of the distance between data points, which are the vertices of the graph, as described in \cref{sec:graph_intro}. 

One of the key objects associated to such graphs is the graph Laplacian which plays a key role in a variety of machine learning methods such as spectral clustering \cite{belkin2001laplacian,belkin2003laplacian,vonLuxburg2007Tutorial,ng2001spectral}, diffusion maps \cite{CoifmanLafon05DM}, and graph neural networks \cite{bruna2013spectral,defferrard2016convolutional,zhou20gnn}. 
The main forms of Laplacians considered are the unnormalized graph Laplacian \labelcref{def:unnormalized_graph_Lapl}, the normalized graph Laplacian \labelcref{def:normalized_graph_Lapl}, and the random walk Laplacian which is spectrally equivalent to the normalized one (see \cite{vonLuxburg2007Tutorial}).
The quadratic forms associated to these operators are the unnormalized graph Dirichlet energy \labelcref{eq:unnormalized_discrete_dirichlet} and the normalized graph Dirichlet energy \labelcref{eq:normalized_discrete_dirichlet}.

Providing theoretical guarantees for such methods spurred on intense research that has seen remarkable progress over the past two decades.  
The works considered the limit of these and related Laplacians under the assumption that data lie in an open set in Euclidean space or on a smooth manifold embedded in the Euclidean space. 
Early works established criteria and error estimates for convergence of graph Laplacians towards the weighted Laplacians, when applied to a fixed smooth function \cite{hein2007graph,singer2006graph}, a notion known as pointwise convergence. 
Subsequently, based on the techniques developed in \cite{garcia2016continuum}, $\Gamma$-convergence of graph Dirichlet energies, which we recall in \cref{thm:Gamma-convergence_standard}, was established on Euclidean domains
\cite{GTS18} and manifolds \cite{laux2025large}. 
Spectral convergence establishes criteria for the eigenvalues and eigenvectors of graph Laplacians to converge to eigenvalues and eigenfunctions of weighted Laplacians or weighted Laplace--Beltrami operators in manifold setting, respectively. 
It is implied by $\Gamma$-convergence via the use of the Courant--Fischer formula for eigenpairs \cite{GTS18,garcia2019variational}, but can also be established using operator perturbation estimates, see \cite{belkin2007convergence,vLBeBo08,CalderGT22,SinWu13}.
A number of recent works established error estimates on the convergence of eigenvalues and eigenvectors of graph Laplacians~\cite{ArmstrongVenkatraman25optimal,garcia2020error,CalderGTLewicka22,GTLi2025Venkatraman25minimax,WormellReich21,ChengWu22}.
   
The works above make the manifold hypothesis, namely assume that data lie on a low-dimensional manifold in the high-dimensional ambient space, or study the simpler case that the data lie in an open set. 
As we argue in the introduction, real data are often comprised of parts that have different intrinsic dimensions. To come closer to capturing the structure of such data, in this paper we consider the case that data lie on a union of, potentially intersecting, manifolds whose intrinsic dimension can differ. 
There are only a few papers in the literature that investigate graph Laplacians on samples of data from intersecting manifolds.

Ullrich \cite{ullrich2024} considers glued manifolds of differing dimensions. However, his goal and  objects of interest are different. 
Namely, we show that on unions of intersecting  manifolds (similar results would hold for glued ones) if the dimensions differ, the limiting problem reduces to the one of a disjoint union of the component manifolds.
In particular, the resulting diffusion would remain on the manifold it started on. 
On the other hand, Ullrich studies how the Laplacian should be reweighted near the intersection (meaning that the diffusion is accelerated) so that the resulting process would visit both manifolds. He is able to obtain precise conditions on the scaling of the weights with respect to the spaces involved for the mixing to happen.
We remark that from the perspective of machining learning tasks such as clustering and classification, it is desirable if the manifolds representing different clusters would be seen as disjoint in the limit. 

Garc\'ia Trillos, He, and  Li \cite{trillos2023large} study the asymptotic behavior of graph Laplacians on samples of intersecting manifolds with the goal of modifying the graph construction so that the limiting problem would be one on the disjoint union of manifolds. 
In particular, they consider graph constructions that have few connections between manifolds (``sparse outer connectivity'' in their terminology). They show that for such graphs the unnormalized graph Laplacian on unions of manifolds of the same dimension indeed spectrally converges to the (weighted) Laplace--Beltrami operator on the disjoint union of manifolds. Furthermore, they provide error estimates for eigenvalues and eigenvectors. When the manifold dimensions differ, they showed that the limit problem only sees the manifolds of the highest dimension, similarly to the limit in such situation in \cref{thm:unnormalized_spectral}. Finally, they provide an example of the graph construction that satisfies the conditions needed.
There, nodes in ``annular proximity graphs with angle constraints'' are connected by an edge if there exists a relatively straight path in a standard random geometric graph between them.

\subsubsection{Nonlocal functionals} \label{sec:literature-nonlocal}
The connection between nonlocal functionals and Dirichlet forms
already existed in the foundational paper on Dirichlet spaces by Beurling and Deny \cite{beurling59}, and has
 been  extensively studied in the literature \cite{AlonsoBaudoin23,Fukushima,KorevaarSchoen93,KumagaiSturm05,Sturm98AP,Sturm98how}.
\red A nonlinear counterpart of this theory, concerned with energies of maps taking values in metric spaces rather than in the real line, was developed by Jost \cite{jost97,jost98}; in the real-valued setting relevant to the present work the connection goes back to Beurling and Deny. \nc
   Korevaar and Schoen \cite{KorevaarSchoen93} introduced the functional, now known as the Korevaar--Schoen energy,
\begin{equation}
    E^{KS}_\eps(u) := \int_{\M} \dashint_{B(x,\eps)} \frac{(u(x)-u(y))^2}{\eps^2} \d\mu(y) \d\mu(x)
\end{equation}
in order to consider Sobolev spaces on spaces lacking smoothness. We note that the functional can be rewritten as
\begin{equation}
    E^{KS}_\eps(u) = \frac{1}{2\eps^2} \int_{\M} \int_{\M} \left( \frac{1}{\eta_\eps(\abs{x-\cdot})\star\mu} +  \frac{1}{\eta_\eps(\abs{y-\cdot})\star\mu}\right) \eta_\eps(|x-y|)  (u(x)-u(y))^2 \d\mu(y) \d\mu(x)
\end{equation}
  for $\eta(s) = \chi_{[0,1]}$. The structure of this energy is similar to that of the normalized nonlocal Dirichlet energy \labelcref{eq:nonlocal_fctl}.
\red A closely related mean-value energy, together with a $\Gamma$-convergence approach, was introduced independently by Jost \cite{jost94} in the study of harmonic maps between metric spaces. \nc

In the sequence of works that followed, Sturm \cite{Sturm98AP,Sturm98how} and Kumagai--Sturm \cite{KumagaiSturm05} used the Korevaar--Schoen energy and related nonlocal Dirichlet energies to define Dirichlet forms on spaces lacking  manifold structure, such as fractals (this required replacing $\eps^2$ by a more general function $h(\eps^2)$). 
In particular, using the techniques of \cite{dalmaso93}, Dirichlet forms are defined as $\Gamma$-limits of nonlocal energies along sequences. The limiting form may depend on the sequence considered. If the measures considered satisfy a doubling property, the authors are able to obtain more detailed properties of the resulting forms. 
Alonso and Baudoin \cite{AlonsoBaudoin23} have studied the Dirichlet energies on Cheeger spaces and shown compactness of sequences and have improved $\Gamma$-convergence to Mosco convergence.
Cheeger spaces are metric measure spaces where the measure satisfies a doubling property and where Poincar\'e inequality with a Lipschitz semi-norm replacing the modulus of the gradient holds. 
We remark that in all of these works the $\Gamma$-convergence follows from abstract results, and the limit may depend on the subsequence taken.
Works such as \cite{AlonsoBaudoin23} further show that the limiting Dirichlet form has the Cheeger--Sobolev space $H^{1,2}(X,d,\mu)$ as its domain and is comparable, up to multiplicative constants, to the Cheeger energy $\tfrac12\int |\nabla u|_w^{\,2}\, d\mu$, where $|\nabla u|_w$ is the minimal weak upper gradient of $u$ (see \cite{cheeger99,ambrosio13density,ambrosio14heatflow}); however, the limit is not identified with the Cheeger energy itself in a general setting.

The contribution of our work in the context of Dirichlet forms on general spaces is that we provide a precise characterization of the limiting form via classical (gradient-based) Dirichlet forms on individual manifolds. 
This provides detailed information and useful examples of Dirichlet forms on metric measure spaces which do not satisfy measure-doubling property in general. 
In particular, when two manifolds $\M\sind{1},\M\sind{2}\subset\R^d$ embedded in $\R^d$ have different dimensions and their intersection $\M\sind{12}:=\M\sind{1}\cap\M\sind{2}$ is non-empty, their union is a metric measure space (equipped with the volume measure inherited from the volume measures of the individual manifolds) which does not satisfy the volume doubling property. 
To see this, consider points $x$ in $\M\sind{2}$ at distance $r$ from $\M\sind{12}$ and compare $\mu(B(x,r))$ and $\mu(B(x,2r))$ while letting $r \to 0$.
We also note that our limiting form does not depend on the sequence considered. 
Finally, our compactness result  \cref{thm:compactness_NL} is new.

\red
\subsection*{Notation}
For the reader's convenience we collect here the main notation used throughout the paper.
\begin{itemize} 
\addtolength{\itemsep}{2pt}%
\addtolength{\itemindent}{-35pt}%
\renewcommand{\labelitemi}{}%
\item $\M=\M\sind{1}\cup\M\sind{2}$ -- the union of  compact smooth manifolds embedded in $\R^N$; see \cref{sec:unions}
\item $\M\sind{i}$, $i=1,2$ -- the two manifolds, of intrinsic dimensions $d\sind{1}\le d\sind{2}$
\item $\M\sind{12}=\M\sind{1}\cap\M\sind{2}$ -- their intersection, a manifold of dimension $d\sind{12}<d\sind{1}$
\item $\theta$ -- the smallest principal angle between $\M\sind1$ and $\M\sind2$ along the intersection; see \eqref{eq:angle}
\item $\vol$, $\vol\sind{i}$ -- the Riemannian volume measure on $\M$, and its restriction to $\M\sind{i}$
\item $d_\M(x,y)$ -- the geodesic distance on $\M$; $\gamma_{xy}$ -- a unit-speed minimizing geodesic from $x$ to $y$; see \cref{lem:Cheeger--Colding}
\item $\dist(x,A)$, $\,d(x):=\dist(x,\M\sind{12})$ -- distance from $x$ to a set $A$, resp.\ to the intersection
\item $\mu=\alpha\sind1\mu\sind1+\alpha\sind2\mu\sind2$ -- the sampled probability measure on $\M$; see \cref{ass:densities}
\item $\mu\sind{i}=\rho\sind{i}\,\vol\sind{i}$ -- the component probability measure on $\M\sind{i}$
\item $\rho\sind{i}$, $\alpha\sind{i}$ -- the density of $\mu\sind{i}$ and the mixture proportion, $\alpha\sind1+\alpha\sind2=1$; see \eqref{eq:rho}
\item $V_n$, $\;\mu_n=\tfrac1n\sum_{x\in V_n}\delta_x$ -- the $n$ i.i.d.\ samples of $\mu$ and their empirical measure; see \cref{sec:graph_intro}
\item $n\sind{i}$ -- the number of samples lying on $\M\sind{i}$, with $\tfrac{n\sind{i}}{n}\to\alpha\sind{i}$; see \eqref{eq:alphas}
\item $\eta$, $\;\eta_\eps(\cdot)=\eta(\cdot/\eps)$ -- the edge weight profile and its rescaling; see \cref{ass:eta}
\item $\eps$, $\eps_n$ -- the bandwidth and its scaling with $n$; $\ell_n$ -- the rate in the lower bound on $\eps_n$; see \cref{ass:epsilon} and \eqref{eq:scaling}
\item $\beta_\eta\sind{i}=\int_{\R^{d\sind{i}}}\eta(\abs{x})\d x$, $\;\sigma_\eta\sind{i}=\int_{\R^{d\sind{i}}}\eta(\abs{x})\abs{x_1}^2\d x$ -- the zeroth and second moments of $\eta$; see \eqref{eq:sigma_beta}
\item $\eta_\eps(\abs{x-\cdot})\star\mu$ -- convolution with a measure; $\deg_{n,\eps}(x)=\eta_\eps(\abs{x-\cdot})\star\mu_n$ -- the weighted degree; see \eqref{eq:degree}
\item $E_{n,\eps}$ -- the discrete normalized graph Dirichlet energy; see \eqref{eq:normalized_discrete_dirichlet}
\item $\E_\eps$ -- the continuum nonlocal normalized Dirichlet energy; see \eqref{eq:nonlocal_fctl}
\item $\E$ -- its $\Gamma$-limit; $\;\E_v$ -- the same energy in the variable $v=u/\sqrt{\rho}$; see \eqref{eq:limit_fctl} and \eqref{eq:Ev}
\item $F_{n,\eps}$ -- the discrete unnormalized graph Dirichlet energy; $\;\mathcal G$ -- its $\Gamma$-limit; see \eqref{eq:unnormalized_discrete_dirichlet} and \eqref{eq:limit_fctl_standard}
\item $L^{N}_{n,\eps}$, $L^{U}_{n,\eps}$ -- the normalized and unnormalized graph Laplacians; see \eqref{def:normalized_graph_Lapl} and \eqref{def:unnormalized_graph_Lapl}
\item $\lambda_k^{(n)},u_k^{(n)}$ and $\lambda_k,u_k$ -- the $k$-th graph eigenpair and its limit; see \cref{thm:spectral_normalized}
\item $H^1_{\sqrt{\mu}}(\M)$ -- the weighted Sobolev space (with matching-trace condition in the codimension-one case); $\trace\sind{i}$ -- the trace operator onto $\M\sind{12}$; $\nabla_\M$ -- the manifold gradient; see \eqref{eq:def_H1mu}
\item $TL^2(\M)$, $d_{TL^2}$ -- the transportation-$L^2$ space and its metric, with the associated convergence; see \cref{sec:appendix_TLp} and \cref{def:TL_convergence}
\item $T_n$ -- a transport map with $(T_n)_\sharp\mu=\mu_n$ and $\norm{\operatorname{id}-T_n}_{L^\infty(\M)}\lesssim\ell_n$; see \cref{rem:scaling_condition} and \eqref{eq:special_maps}
\end{itemize}
\nc

\section{Setup and main results}
\label{sec:setup+main}

\subsection{Geometry of the space}
\label{sec:unions}

For notational simplicity we just consider a space which is a union of two intersecting manifolds. 
Let $\M := \M\sind{1} \cup \M\sind{2}$ be the union of two compact, smooth Riemannian manifolds without boundary  of dimension $d\sind{i}$, $i=1,2$, embedded in Euclidean space, $\M\sind{i}\subset\R^N$. Without a loss of generality we will assume in throughout the paper that $1 \leq d\sind{1}\leq d\sind{2}$.
We assume that the manifolds intersect in a \emph{nondegenerate} way, which means that
their intersection $\M\sind{12}:=\M\sind{1}\cap \M\sind{2}$ is a manifold of dimension $d\sind{12}<d\sind{1}$ embedded in $\R^N$ and for every $x \in \M\sind{12}$ the dimension of $\spann(T_x\M\sind{1}  \cup T_x\M\sind{2}) $ is $d\sind{1}+d\sind{2}-d\sind{12}$. 
We note that the smallest principal angle
\begin{align}\label{eq:angle}
\begin{split}
    \theta := \min_{x\in\M\sind{12}}\Big\{  \angle(u,v) \st & u \in  T_x\M\sind{1}  \cap (T_x\M\sind{1}  \cap T_x\M\sind{2})^\perp, \; u \neq 0\\
    & v \in  T_x\M\sind{2}  \cap (T_x\M\sind{1}  \cap T_x\M\sind{2})^\perp, \; v \neq 0  \Big\},
\end{split}
\end{align}
is positive.
Here, $\angle(u,v) := \arccos\frac{\abs{\langle u,v\rangle}}{\abs{u}\abs{v}}$ is the  angle between two vectors $u,v\in\R^N$. 
We point to \cref{lem:angle} in the appendix for an important consequence of the positivity of the principal angle.

We consider random samples of a probability measure supported on $\M$. More precisely, 
consider probability measures $\mu\sind{i} \in \mathcal P(\M\sind{i})$ supported in the manifolds $\M\sind{i}$, for $i=1,2$. We assume that $\mu\sind{i}$ has a density $\rho\sind{i}$ with respect to the volume measure $\vol\sind{i}$ on $\M\sind{i}$ and as a consequence it holds $\mu\sind{i}(\M\sind{12}) = 0$.
We pose the following set of standard assumptions for the densities:
\begin{assumption}\label{ass:densities}
The probability densities $\rho\sind{i}$ with $i\in\{1,2\}$ satisfy the following:
    \begin{listi}
        \item $\rho\sind{i}$ is Lipschitz continuous on $\M\sind{i}$.
        \item $\rho\sind{i} $ is bounded from above and below by positive numbers: $\frac{1}{C_\rho}\leq\rho\sind{i}\leq C_\rho$ for some $C_\rho>0$.
    \end{listi}
\end{assumption}
 To describe the proportions of samples from each of the two manifolds, we introduce numbers $ \alpha\sind{i}>0$ for $i=1,2$, satisfying $ \alpha\sind{1}+ \alpha\sind{2}=1$. 
We use them to define a mixture model for the union of manifolds by defining the probability measure $\mu \in \mathcal P(\M)$ via $\mu(A) := \sum_{i=1}^2\alpha\sind{i}\mu\sind{i}(A\cap\M\sind{i})$ for Borel subsets $A\subset\M$.
 The measure $\mu$ has density $\rho$ with respect to the volume measure $\vol$ on $\M$, defined via $\vol\restr_{\M\sind{i}}:=\vol\sind{i}$, where 
\begin{align} \label{eq:rho}
\rho\vert_{\M\sind{i}}:=\alpha\sind{i}\rho\sind{i} \quad \te{for }i\in\{1,2\}.
\end{align}
Let $V_n$ be a set of  $n$  \emph{i.i.d.} samples of the probability measure $\mu$. 
Let $n\sind{i}$ be the number of samples that lie in $\M\sind{i}$ for $i=1,2$. 
Note that $n\sind{i}\sim\mathrm{Binomial}(n,\alpha\sind{i})$ and hence the points in $V_n\cap\M\sind{i}$ constitute a so-called Binomial point process. 
Then, almost surely, $n\sind{1}+n\sind{2}=n$ and by the law of large numbers
       \begin{align}\label{eq:alphas}
     \lim_{n\to\infty}\frac{n\sind{i}}{n} = \alpha\sind{i}
    \qquad i\in\{1,2\}.
\end{align}
                                                   
\subsection{Graph construction}
\label{sec:graph_intro}

Given a set of $n$ points $V_n$ in $\R^N$ (samples of $\M$), a weight profile $\eta:[0,\infty)\to[0,\infty)$ and a bandwidth parameter $\eps>0$, we construct a weighted graph whose vertices are points in $V_n$ and where the vertices $x,y\in V_n$ are connected by an edge with the weight $\eta_\eps(\abs{x-y})$.
Here $\abs{\cdot}$ denotes the Euclidean distance in $\R^N$. We define $\eta_\eps(\cdot)=\eta(\cdot/\eps)$.
We remark that the weight function $\eta_\eps(\abs{x-y})$ will be the only quantity which explicitly depends on the ambient Euclidean metric in $\R^N$. 
    
We note that our definition of $\eta_\eps$ does not contain a dimension-dependent scaling of the bandwidth~$\eps$.  \red Papers dealing with manifold learning typically include such scaling, as it is needed when working with the unnormalized graph Dirichlet energy or Laplacian. 
However, for data coming from manifolds of different dimensions there is no single natural scaling that preserves information from both manifolds in the limit. For the normalized graph Laplacian the dimension-dependent scaling is not needed, as the normalization makes the quantities nondimensional. On the other hand, for the unnormalized graph Laplacian, contributions from manifolds of different dimensions scale differently with the bandwidth. Thus, when dealing with the unnormalized Laplacian we explicitly include the dimensional scaling in the statement of \cref{thm:unnormalized_Gamma,thm:unnormalized_spectral}. \nc
We impose the following requirements on the weight function $\eta$ that are standard in the field of graph-based learning \red and, in related forms, in the PDE literature\nc.
\begin{assumption}\label{ass:eta}
    The function $\eta:[0,\infty)\to[0,\infty)$ satisfies:
    \begin{listi} \addtolength{\itemsep}{-3pt}
        \item $t\mapsto\eta(t)$ is non-increasing.
        \item $\eta(0)>0$ and $\eta$ is continuous at $0$.
        \item $\supp\eta = [0,1]$.
     \end{listi}   
\end{assumption}
We first remark that condition (iii) can be, trivially, replaced by the requirement that the support is compact. 
Moreover, condition (iii) could be replaced by a decay condition at infinity. However, this would necessitate estimating the tail terms and hence make the proofs more complicated. 
For $i=1,2$ we define the following moments
\begin{subequations}\label{eq:sigma_beta}
\begin{align}
    \label{eq:sigma}
    \sigma_{\eta}\sind{i} &:= 
    \int_{\R^{d\sind{i}}} \eta(|x|) |x_1|^2 \d x, \\
    \label{eq:beta}
    \beta_{\eta}\sind{i} &:= \int_{\R^{d\sind{i}}} \eta(|x|) \d x,
\end{align}
\end{subequations}
which are finite if $\eta$ satisfies \cref{ass:eta}.
Finally, we let
\begin{align}\label{eq:degree}
    \deg_{n,\eps}(x) := \frac{1}{n} \sum_{y \in V_n} \eta_\eps(|x-y|)   
\end{align}
denote the weighted degree of $x$, or in other words, the average of all edge weights adjacent to $x$.

\subsection{Discrete and continuum normalized Dirichlet energies}
\label{sec:normalized_energies}

For $\eps>0$, on the graph with vertices $V_n$ and edge weights $\eta_\eps(|x-y|)$, the
discrete normalized Dirichlet energy of a function $u:V_n\to\R$ is defined by
\begin{align}\label{eq:normalized_discrete_dirichlet}
E_{n,\eps}(u) := \frac{1}{n^2 \eps^2} \sum_{x,y\in V_n}\eta_\eps (|x-y|) \abs{\frac{u(x)}{\sqrt{\deg_{n,\eps}(x)}}-\frac{u(y)}{\sqrt{\deg_{n,\eps}(y)}}}^2.
\end{align}
The empirical measure of points in $V_n$ is the probability measure
\begin{align*}
    \mu_n:=\frac{1}{n}\sum_{x\in V_n}\delta_x
\end{align*}
Since $n\sind{i}$ are both positive with (very) high probability, $\mu_n$ can be written as a convex combination of two probability measures
\begin{align*}
    \mu_n =
    \sum_{i=1}^2
    \frac{n\sind{i}}{n}
    \mu_n\sind{i} \qquad \te{where }\;
    \mu_n\sind{i} := \frac{1}{n\sind{i}}\sum_{x\in V_n\cap\M\sind{i}}\delta_x,
\end{align*}
are the empirical measures of the data points that lie in $\M\sind{i}$ for $i\in\{1,2\}$.
We can rewrite and reinterpret $E_{n,\eps}$ as functional on $L^2(\mu_n)$ as follows
\begin{align}\label{eq:graph_fctl_integral}
    E_{n,\eps}(u)
    :=
    \frac{1}{\eps^2} \int_\M\int_\M\eta_\eps (|x-y|) \abs{\frac{u(x)}{\sqrt{\eta_\eps(\abs{x-\cdot})\star\mu_n}}-\frac{u(y)}{\sqrt{\eta_\eps(\abs{y-\cdot})\star\mu_n}}}^2\d\mu_n(x)\de\mu_n(y),
\end{align}
where we used the notation
\begin{align*}
    \eta_\eps(\abs{x-\cdot})\star\mu_n
    :=
    \int_\M \eta_\eps(\abs{x-z})\d\mu_n(z)
    =
    \deg_{n,\eps}(x).
\end{align*}
This interpretation gives rise to a nonlocal functional defined on $L^2(\M)$ which arises as large sample limit of $E_{n,\eps}$ and takes the form
\begin{align}\label{eq:nonlocal_fctl}
    \E_\eps(u) := 
    \frac{1}{\eps^2} \int_\M\int_\M\eta_\eps (|x-y|) \abs{\frac{u(x)}{\sqrt{\eta_\eps(\abs{x-\cdot})\star\mu}}-\frac{u(y)}{\sqrt{\eta_\eps(\abs{y-\cdot})\star\mu}}}^2\d\mu(x)\de\mu(y),
\end{align}
where $\eta_\eps(\abs{x-\cdot})\star\mu := \int_\M\eta_\eps(\abs{x-z})\d\mu(z)$.

The main result of this paper is the proof that the functionals $E_{n,\eps}$ $\Gamma$-converge as $n\to\infty$ and $\eps\to 0$ to a normalized Dirichlet energy $\E: L^2(\M) \to [0, \infty]$, defined as
\begin{align}\label{eq:limit_fctl}
    \E(u) = 
    \begin{dcases}
    \sum_{i=1}^2
    \frac{\sigma_\eta\sind{i}}{\beta_\eta\sind{i}}
    \int_{\M\sind{i}}\abs{\nabla_{\M\sind{i}} \left( \frac{u}{\sqrt{ \alpha\sind{i} \rho\sind{i}}} \right) \alpha\sind{i}\rho\sind{i}}^2 \d\vol\sind{i}
    \quad &\te{if } u\in H^1_{\sqrt \mu}(\M),  \\
    +\infty   \quad &\te{else}.
    \end{dcases}
\end{align}
Here, $\alpha\sind{i}$  are as in \labelcref{eq:rho}, $\sigma_\eta\sind{i}$ as well as $\beta_\eta\sind{i}$ are defined in \labelcref{eq:sigma_beta}, and the Sobolev space $H^1_{\sqrt{\mu}}(\M)$ is defined as
\begin{align}\label{eq:def_H1mu}
\begin{split}
    H^1_{\sqrt \mu}(\M) 
    :=&  
    \Bigg\{ u \in L^2(\M) \st  \frac{u|_{\M\sind{i}}}{\sqrt{\rho\sind{i}}} \in H^1(\M\sind{i}) \te{ for } i=1,2 \te{ and }
    \\ 
   &\quad\te{if } d\sind{1} = d\sind{2} = d\sind{12} + 1 \te{ then } \trace\sind{1}\left(\frac{u}{\sqrt{\alpha\sind{1}\rho\sind{1}}}\right)=\trace\sind{2}\left(\frac{u}{\sqrt{\alpha\sind{2}\rho\sind{2}}}\right)  \Bigg\},
\end{split}
\end{align}
where $\trace\sind{i} : H^1(\M\sind{i}) \to H^{1/2}(\M\sind{12})$ is the trace operator from each manifold $\M\sind{i}$ for $i=1,2$ to the intersection $\M\sind{12}$. 
Note that the trace condition applies only in the special case that manifolds are of the same dimension and their intersection is a codimension-$1$ submanifold in each. Then we are asking for the regularity of $u/\sqrt{\rho}$ on the union of the manifolds, while $u$
itself can have a jump discontinuity between the manifolds (if $\alpha\sind{1}\rho\sind{1}$ and $\alpha\sind{2}\rho\sind{2}$ differ on $\M\sind{12}$).

\begin{remark} \label{rem:Gauss}
We note that if $\eta(s) = c \exp(-s^2/2)$, that is, if $\eta$ is any constant multiple of the Gaussian with unit variance, then the factor   $\frac{\sigma_\eta\sind{i}}{\beta_\eta\sind{i}}=1$ regardless of $i$ (or in other words regardless of the dimension).  In particular the weighted Laplacian corresponding to $\E$ does not have dimension-dependent factors. We note that this example does not fit the \cref{ass:eta} (iii), but as we remarked the theoretical results are expected to hold for $\eta$ that decays sufficiently fast to zero. 
\end{remark}

For the $\Gamma$-convergence to hold we require, as is standard, that $\eps = \eps_n$ converges to zero sufficiently slowly as $n \to \infty$. In particular we assume
\begin{assumption}\label{ass:epsilon}
    The sequence of graph bandwidths $(\eps_n)_{n\in\N}\subset(0,\infty)$ satisfies
    \begin{align}\label{eq:scaling}
        \ell_n \ll \eps_n \ll 1,
    \end{align}
    where we let $d:=\max\left\{d\sind{1}, d\sind{2}\right\}$ denote the maximal dimension of the manifolds and define
\red
    \begin{align*}
        \ell_n :=
        \begin{dcases}
        \sqrt{\frac{\log\log n}{n}} \quad&\text{if }d=1,
        \\
        \frac{(\log n)^{3/4}}{n^{1/2}} \quad&\text{if }d=2,
        \\
          \left(\frac{\log n}{n}\right)^\frac{1}{d} \quad&\text{if }d\geq 3.
        \end{dcases}
    \end{align*}
 The rate $\ell_n$ is precisely the  rate of decay of the $\infty$-Wasserstein distance $d_\infty(\mu,\mu_n)$ between $\mu$ and the empirical measure $\mu_n$; see \cref{rem:scaling_condition}. \nc
\end{assumption}
  \begin{remark}
Note that with the notation $a_n \ll b_n$ (or equivalently $b_n \gg a_n$) for two sequences of non-negative real numbers $(a_n)_{n\in\N}$ and $(b_n)_{n\in\N}$ we mean that
\begin{align*}
    \lim_{n\to\infty}\frac{a_n}{b_n} = 0.
\end{align*}
We also remark that the condition \eqref{eq:scaling} is essentially sharp in terms of scaling, as it is known that even on a single manifold with $d \geq 2$ the functionals do not converge to a desirable limit if $\eps_n \ll   \left(\nicefrac{\log n}{n}\right)^{1/d}$, see \cite{GTS18}.  
For continuum limits in the challenging percolation regime $\eps_n \sim \left(\nicefrac{\log n}{n}\right)^{1/d}$ on Poisson point clouds we refer to \cite{bungert2024ratio} for the graph infinity Laplacian and to \cite{ArmstrongVenkatraman25optimal} for the spectrum of the graph Laplacian.
\end{remark}
\begin{remark}\label{rem:scaling_condition}
    With probability of order $n^{-\alpha/2}$ for arbitrary $\alpha>2$ the lower bound $\ell_n$ for $\eps_n$ in \labelcref{eq:scaling} upper bounds the Wasserstein-$\infty$ distance of the absolutely continuous measure $\mu$ and the empirical sample measure $\mu_n$, see \cite{trillos2015rate,garcia2020error} for $d\geq 2$ and \cite{slepcev2019analysis} for the case $d=1$.
    As a consequence, there exists a constant $C>0$ such that with  probability one there are maps $T_n:\M\to\M$ which satisfy $(T_n)_\sharp\mu=\mu_n$ and
    \begin{align}\label{eq:special_maps} 
        \limsup_{n\to\infty}
        \frac{\norm{\operatorname{id}-T_n}_{L^\infty(\M)}}{\ell_n} \leq C.
    \end{align}
    \red Intuitively, $T_n$ is a transport map that pushes the measure $\mu$ onto the empirical measure $\mu_n$ while moving each point by at most $O(\ell_n)$. It provides the coupling underlying the $TL^2$ metric and is the device that lets us compare a continuum function on $\M$ with a function defined on the sample $V_n$ (by composition with $T_n$); it is used repeatedly in the proofs in \cref{sec:proofs_graphs}.
    We note that in dimension $d=2$ the rate $\ell_n=(\log n)^{3/4}/n^{1/2}$ exceeds the connectivity threshold $\left(\log n/n\right)^{1/2}$ by a factor $(\log n)^{1/4}$. We expect that the averaging approach of \cite[Proposition~2.12]{CalderGT22} (or \cite[Lemmas~3.1, 3.2]{Caroccia20}) can be used to improve the admissible rate in the $d=2$ case to $\left(\log n/n\right)^{1/2}$, matching the other dimensions, but we do not pursue this here. \nc
\end{remark}

\subsection{Main results}
\label{sec:main_results}

The following are our main results:
\begin{theorem}[Compactness]\label{thm:compactness}
Under \cref{ass:densities,ass:epsilon,ass:eta} let $u_n\subset L^2(\mu_n)$ be a sequence such that 
\begin{align*}
    \sup_{n\in\N} E_{n,\eps_n}(u_n)<\infty,\qquad \sup_{n\in\N} \norm{u_n}_{L^2(\mu_n)}<\infty.
\end{align*}
Then almost surely $(u_n)_{n\in\N}$ possesses a subsequence which converges in the $TL^2(\M)$-sense.
\end{theorem}
We note that the definition of $TL^2(\M)$ metric is recalled in \cref{sec:appendix_TLp}.

\begin{theorem}[$\Gamma$-convergence]\label{thm:gamma}
Under \cref{ass:densities,ass:epsilon,ass:eta} almost surely the functionals~$E_{n, \eps_n}$, defined in \labelcref{eq:normalized_discrete_dirichlet}, $\Gamma$-converge to $\E$, defined in \labelcref{eq:limit_fctl}, in the $TL^2(\M)$-sense, meaning that
\begin{enumerate}
    \item Almost surely for all sequences $(u_n)\subset L^2(\mu_n)$ with $u_n\overset{TL^2}{\longrightarrow}u$ for $u\in L^2(\M)$ it holds
    \begin{align*}
        \E(u) \leq \liminf_{n\to\infty}E_{n,\eps_n}(u_n);
    \end{align*}
    \item Almost surely for all $u\in L^2(\M)$ there exists a so-called recovery sequence $(u_n)\subset L^2(\mu_n)$ with $u_n\overset{TL^2}{\longrightarrow}u$ and
    \begin{align*}
        \limsup_{n\to\infty}E_{n,\eps_n}(u_n) \leq \E(u).
    \end{align*}
\end{enumerate}
\end{theorem} 

\cref{sec:proofs_graphs} is devoted to the proofs of these two theorems. 
In fact, the central difficulties of proving $\Gamma$-convergence of the graph functional $E_{n,\eps}$  to $\E$ already arise when one aims to prove $\Gamma$-convergence of its nonlocal counterpart $\E_\eps$ defined in \labelcref{eq:nonlocal_fctl}.
Since describing a proof of $\Gamma$-convergence for these continuum nonlocal functionals allows us to clearly present the main ideas and may be of independent interest, we dedicate \cref{sec:nonlocal} to nonlocal functionals, where we present arguments for proving their $\Gamma$-convergence as $\eps\to 0$.

\subsection{Convergence of eigenvalues and eigenvectors of normalized graph Laplacians}
\label{sec:normalized_EVconvergence}

The results on $\Gamma$-convergence can be directly applied to establish the convergence of eigenvalues and eigenvectors of the normalized graph Laplacian on the samples of the union of manifolds to the eigenvalues and eigenfunctions of the weighted Laplace--Beltrami operator on the union of manifolds. 
This proof relies on a variational description of the spectral problem via the Courant--Fisher formula and follows, almost verbatim, the proof of convergence of eigenvalues and eigenfunctions in Sections 3.1 and 3.2 of \cite{GTS18}. We therefore just state the result and refer to \cite{GTS18} for technical details.

We first note that the normalized graph Laplacian associated to the Dirichlet energy \labelcref{eq:normalized_discrete_dirichlet} has the following form: For $u: V_n \to \R$ and $x \in V_n$
  \begin{align} \label{def:normalized_graph_Lapl}
    L^{N}_{n,\eps} u(x) := \frac{2}{n \eps^2} 
    \left( u(x) - \sum_{y \in V_n} \frac{\eta_\eps(|x-y|)}{\sqrt{\deg_{n,\eps}(x)}\sqrt{\deg_{n,\eps}(y)}}\, u(y) \right)
\end{align}
In describing the limiting eigenvalue problem we distinguish  two cases:

\paragraph{Case 1, \texorpdfstring{$d\sind{2}-d\sind{12} \geq 2$}{d2-d12>=2}.}
In this case the limiting problem is a pair of eigenvalue problems on individual manifolds; the problems are entirely decoupled.
That is the eigenvalue problem for the operator  corresponding to functional defined in \labelcref{eq:limit_fctl} is as follows:
For $u \in H^1_{\sqrt{\mu}}(\M)$ (which in this case is equivalent to the restriction of $u$ to $\M\sind{i}$ lying in $H^1(\M\sind{i})$), we have for $i=1,2$ on $\M\sind{i}$:
\begin{align} \label{EVP1}
  \begin{split}
        - \frac{\sigma_\eta\sind{i}}{\beta_\eta\sind{i}}
      \frac{1}{ {\rho\sind{i}}^\frac{3}{2}} 
      \div_{\M\sind{i}} \left( {\rho\sind{i}}^2 \nabla_{\M\sind{i}} \left( \frac{u\sind{i}}{\sqrt{\rho\sind{i}}} \right)\right) & = \lambda u\sind{i}.
  \end{split}
\end{align}
We note that if $(\lambda,u\sind{i})$ is a  normalized (in $L^2_{\rho\sind{i}}$) eigenpair for the problem 
\labelcref{EVP1} considered on the \emph{individual} manifold $\M\sind{i}$, and if we define 
$u$ to be $u\sind{i}/\sqrt{\alpha\sind{i}}$ on $\M\sind{i}$ and zero on the other manifold, 
then $(\lambda,u)$ is a normalized (in $L^2_{\rho}$) eigenpair for the full problem \labelcref{EVP1}.
From the standard elliptic theory (see for example \cite{evans,GTS18}) we know that on each individual manifold the problem \labelcref{EVP1} has a  sequence of eigenpairs $(\lambda_k\sind{i},u_k\sind{i})$ such that $0=\lambda_1\sind{i} < \cdots \leq \lambda_k\sind{i} \leq \lambda_{k+1}\sind{i} \leq \cdots$ and  $\{u_k\sind{i}\}_{k \in \N}$ form an orthonormal basis of $L^2_{\rho\sind{i}} (\M\sind{i})$. 
Since $L^2_{\rho}(\M) \cong L^2_{\rho\sind{1}}(\M\sind{1}) \times L^2_{\rho\sind{2}}(\M\sind{2}) $,  extending $u_k\sind{i}$ by zero to the other manifold, as above, and combining the resulting eigenvectors for \labelcref{EVP1} into one sequence provides an orthornomal basis for $L^2_{\rho}(\M)$.
We conclude that there exists a sequence of eigenvalues $0=\lambda_1=\lambda_2 < \lambda_3 \leq \lambda_4 \leq \cdots$ which converges to infinity, and a corresponding orthonormal in $L^2_\rho$ sequence of eigenfunctions.

\paragraph{Case 2, \texorpdfstring{$d\sind{1}=d\sind{2}=:d$}{d1=d2=:d} and \texorpdfstring{$d-d\sind{12}=1$}{d-d12=1}.}
In this case the limiting eigenvalue problem still takes the form \labelcref{EVP1}, just that the fact that $u \in H^1_{\sqrt{\mu}}(\M)$ couples the two problems on $\M\sind{1}$ and $\M\sind{2}$. This problem too has a nondecreasing sequence of nonnegative eigenvalues and a corresponding orthonormal  basis of eigenfunctions. 
To establish this it is convenient to restate the Dirichlet form and the eigenvalue problem in terms of $v = \frac{u}{\sqrt\rho}$, where $\rho$ is as in \labelcref{eq:rho}. Then $v \in H^1(\M)$ and the Dirichlet form in \labelcref{eq:limit_fctl} can be restated as 
\begin{align}\label{eq:Ev}
    \E_v(v) := 
    \begin{dcases}
    \frac{\sigma_\eta}{\beta_\eta}
    \int_{\M}\abs{\nabla_{\M} v}^2 \rho^2 \d\vol
    \quad &\te{if } v\in H^1(\M),  \\
    +\infty   \quad &\te{else},
    \end{dcases}
\end{align}
where we dropped the indices on $\sigma_\eta$ and $\beta_\eta$ as they do not depend on $i$, and $\nabla_\M$ is considered to mean $\nabla_{\M\sind{i}}$ for points in $\M\sind{i}$. 
The eigenvalue problem is 
\begin{align}\label{EVPV}
      - \frac{\sigma_\eta}{\beta_\eta}
      \frac{1}{\rho^2} 
      \div_{\M} \left(\rho^2 \nabla_{\M} v \right)  = \lambda v.
\end{align}
The proof of the spectral theorem for the Laplacian in Euclidean domains (see for example \cite{evans}) readily extends to the present situation. Since this is long, but elementary, we omit the full proof. We just note 
that $\F_v(v) = \E_v(v) + \|v\|_{L^2(\M)}^2$ is a quadratic form that is equivalent to the squared $H^1$-norm.
Thus, the existence of solutions of elliptic problems on $\M$ follows using the Riesz representation lemma. The compactness of the solution operator as a mapping from $H^1(\M)$ to $L^2(\M)$ follows immediately by Rellich--Kondrachov theorem since if $w \in H^1(\M)$, then the restriction of $w$ to $\M\sind{i}$ belongs to $H^1(\M\sind{i})$ which compactly embeds in $L^2(\M\sind{i})$. Thus one can use the Fredholm alternative for compact operators to show that the operator has a complete discrete spectrum.

\begin{theorem} \label{thm:spectral_normalized}
Suppose \cref{ass:densities,ass:epsilon,ass:eta} hold.
Let $\lambda^{(n)}_k$ be the $k$-th eigenvalue of $L^N_{n,\eps_n}$ and let $u^{(n)}_k$ be the associated eigenfunction, normalized in $L^2(\mu_n)$. 
Let $\lambda_k$ be be the $k$-th eigenvalue for \cref{EVP1}.
  
Then almost surely
\begin{listi}
\item (Convergence of eigenvalues) For every $k \in \N$,
    \begin{align*}
       \lim_{n \to \infty} \lambda^{(n)}_k = \lambda_k.
    \end{align*}
\item (Convergence of eigenfunctions)
  The sequence of eigenfunctions $\{u^{(n)}_k\}_{n\in\N}$ (as well as any of its subsequences) has a convergent subsequence and the limit of any such subsequence is a normalized eigenfunction for \cref{EVP1} corresponding to eigenvalue $\lambda_k$.
\end{listi}
\end{theorem}
The proof of the theorem relies on variational description of the spectrum using the Courant--Fisher formula and the $\Gamma$-convergence of the associated energies. This argument is carried out in detail in the proofs of Theorems 1.2 and 1.5 in \cite{GTS18}.

\subsection{The unnormalized (standard) graph Laplacian: \texorpdfstring{$\Gamma$}{Gamma}-convergence and spectral convergence}
\label{sec:unnormalized_EVconvergence}

As we indicate in the introduction, we advocate for the use of the normalized graph Dirichlet energy (and hence the associated normalized graph Laplacian) as a tool to preserve information from all manifolds in a multi-manifold setting. 
In order to be able to contrast the  behavior of the normalized graph Laplacian to the unnormalized one, we here study the asymptotic behavior (as $n \to \infty$) of the unnormalized graph Dirichlet energy and the associated  graph Laplacian. The main point is that when $d\sind{1} < d\sind{2}$ then only the information about $\M\sind{2}$ is retained in the limit.

Following the setting of  \cref{sec:normalized_energies}, we define the discrete Dirichlet energy of a graph function $u:V_n\to\R$ as
\begin{align}\label{eq:unnormalized_discrete_dirichlet}
    F_{n,\eps}(u) := \frac{1}{n^2 \eps^2} \sum_{x,y\in V_n}\eta_\eps (|x-y|) \abs{u(x)-u(y)}^2.
\end{align}
The associated (unnormalized) graph Laplacian has the form: For $x \in V_n$
\begin{align} \label{def:unnormalized_graph_Lapl}
    L^{U}_{n,\eps} u(x) := \frac{2}{n \eps^2}  \sum_{y \in V_n} \eta_\eps(|x-y|)
    \left( u(x) - u(y)  \right).
\end{align}
We note that the definitions above differ, by a power of $\eps$, from the definition of unnormalized Dirichlet energy and unnormalized Laplacian in the works that study their convergence, such as \cite{GTS18,laux2023large,laux2025large,CalderGT22,calder2018game,bungert2024poisson}, cf. \cref{thm:Gamma-convergence_standard} in the appendix. 
In those works one would also divide by $\eps^d$ where $d$ is the dimension of the space considered. In the setting of the manifolds with different dimension there is no power of~$\eps$ that we can divide by to properly normalize the functional. Thus we here define the unnormalized Dirichlet energy and the associated Laplacian without the dimension dependent rescaling factor. The proper rescalings are stated and discussed in the results below, namely \cref{thm:unnormalized_compactness,thm:unnormalized_Gamma,thm:unnormalized_spectral}.

The $\Gamma$-limit of a suitable rescaling of the energies $F_{n,\eps}$ is given by the following Dirichlet energy $\mathcal G: L^2(\M) \to [0, \infty]$:
\begin{align}\label{eq:limit_fctl_standard}
    \mathcal G(u) = 
    \begin{dcases}
    \sum_{i=1}^2
    \sigma_\eta\sind{i} \int_{\M\sind{i}}\abs{\nabla_{\M\sind{i}}u}^2\left(\alpha\sind{i}\rho\sind{i}\right)^2\d\vol\sind{i}
     &\te{if }d\sind{1}=d\sind{2} \te{ and } u\in H^1(\M),  \\
    \sigma_\eta\sind{2}
\int_{\M\sind{2}}\abs{\nabla_{\M\sind{2}}u}^2\left(\alpha\sind{2}\rho\sind{2}\right)^2\d\vol\sind{2}
    &\te{if }d\sind{1}<d\sind{2},\, u\in H^1(\M),\te{ }u\vert_{\M\sind{1}}\equiv const. \\
    +\infty   \quad &\te{else}.
    \end{dcases}
\end{align}
Here, since the densities $\rho\sind{i}$ do not appear within the gradient term, the Sobolev space $H^1(\M)$ is defined as in \labelcref{eq:def_H1mu} for $\mu\equiv const.$, i.e.,
\begin{align}
\label{eq:def_H1mv}
\begin{split}
    H^1(\M) 
    :=&  
    \Bigg\{ u \in L^2(\M) \st  u|_{\M\sind{i}}\in H^1(\M\sind{i}) \te{ for } i=1,2 \te{ and }
    \\ 
   &\quad\te{if } d\sind{1} = d\sind{2} = d\sind{12} + 1 \te{ then } \trace\sind{1}u=\trace\sind{2}u \Bigg\}.
\end{split}
\end{align} 
We now turn to $\Gamma$-convergence and compactness of unnormalized graph Dirichlet energy. 
\begin{theorem}[Compactness] \label{thm:unnormalized_compactness}
Under \cref{ass:densities,ass:epsilon,ass:eta} let $u_n\subset L^2(\mu_n)$ be a sequence such that 
\begin{align*}
    \sup_{n\in\N} \eps_n^{-d\sind{2}} F_{n,\eps_n}(u_n)<\infty,\qquad \sup_{n\in\N} \norm{u_n}_{L^2(\mu_n)}<\infty.
\end{align*}
If $d\sind{1}=d\sind{2}=d\sind{12}+1$, then almost surely $(u_n)_{n\in\N}$ possesses a subsequence which converges in the $TL^2(\M)$-sense.
If $d\sind{1}<d\sind{2}$, then almost surely $(u_n\vert_{\M\sind{2}})_{n\in\N}$ possesses a subsequence which converges in the $TL^2(\M\sind{2})$-sense.
\end{theorem}
\begin{theorem}[$\Gamma$-convergence] \label{thm:unnormalized_Gamma}
Under \cref{ass:densities,ass:epsilon,ass:eta} almost surely the functionals~$\eps_n^{-d\sind{2}}F_{n, \eps_n}$ $\Gamma$-converge to $\mathcal G$ in the $TL^2(\M)$-sense, meaning that
\begin{listi}
    \item Almost surely for all sequences $(u_n)\subset L^2(\mu_n)$ with $u_n\overset{TL^2}{\longrightarrow}u$ for $u\in L^2(\M)$ it holds
    \begin{align*}
        \mathcal G(u) \leq \liminf_{n\to\infty}\eps_n^{-d\sind{2}}F_{n,\eps_n}(u_n);
    \end{align*}
    \item Almost surely for all $u\in L^2(\M)$ there exists a so-called recovery sequence $(u_n)\subset L^2(\mu_n)$ with $u_n\overset{TL^2}{\longrightarrow}u$ and
    \begin{align*}
        \limsup_{n\to\infty}\eps_n^{-d\sind{2}}F_{n,\eps_n}(u_n) \leq \mathcal G(u).
    \end{align*}
\end{listi}
\end{theorem} 
The steps of the proof of the $\Gamma$-convergence of the unnormalized graph Dirichlet energy are either completely analogous to the normalized graph Laplacian or are much simpler and follow the standard techniques presented here and in \cite{GTS18}. We omit them for brevity. 

We note that if the manifold dimensions are unequal, the $\Gamma$-limit $\mathcal{G}$ forces constant functions on the lower-dimensional manifold and just depends on the Dirichlet energy on the higher-dimensional manifold. 
This behavior is consistent with the spectral convergence for multi-manifold clustering with adapted weight constructions developed in \cite{trillos2023large}.

\begin{remark}
    Note that another potential scaling would be to consider $\eps_n^{-d\sind{1}}F_{n,\eps_n}$.
    However, in this case, we would only get compactness in the $TL^2(\M\sind{1})$ topology and we would lose it on the higher-dimensional manifold $\M\sind{2}$.
    The corresponding $\Gamma$-limit would be the standard Dirichlet energy on~$\M\sind{1}$.
\end{remark}
\begin{remark}
    Note that in the case $d\sind{1}=d\sind{2}=d\sind{12}+1$ the $\Gamma$-limit $\mathcal{G}$ is very similar to $\E$ in \labelcref{eq:limit_fctl} in that it is a sum of Dirichlet energies on the two manifolds together with trace constraint on their intersection. 
    The differences between the unnormalized and the normalized models become more apparent in case of different dimensions $d\sind{1}<d\sind{2}$, where the $\Gamma$-limit $\mathcal{G}$ of the unnormalized energies forces functions to be constant on $\M\sind{1}$ but is otherwise independent of the lower dimensional manifold.
\end{remark}
\medskip

To describe the eigenvalue problem associated to the unnormalized Dirichlet energy consider the following equation for $u \in H^1(\M)$, which  takes the following form on $\M\sind{i}$ for $i=1,2$
\begin{align} \label{EVP2}
L^U u:=  - \frac{\sigma\sind{i}_\eta}{{\alpha\sind{i} \rho\sind{i}}} \div_{\M\sind{i}} \left( ({\alpha\sind{i} \rho\sind{i}})^2 \nabla_{\M\sind{i}} u\right) = \lambda u
\end{align}
If $d\sind{1}=d\sind{2}$ the eigenfunctions are nonzero $u \in H^1(\M)$ which solve \cref{EVP2} on all of $\M$. 
If  $d\sind{1}<d\sind{2}$, then eigenfunctions $u$  are  nonzero functions in $H^1(\M) = H^1(\M\sind{1}) \times H^1(\M\sind{2})$, with the operator in the  equation on $\M\sind{1}$ replaced by 
\[  L^U u =
\begin{cases}
    0 \quad & \te{if }\; u = \te{const.} \\
    \infty & \te{otherwise.}
\end{cases} \]
We remark that when $d\sind{1}<d\sind{2}$, $\lambda=0$ is an eigenvalue of multiplicity 2 with eigenspace consisting of functions which are constant on both $\M\sind{1}$ and $\M\sind{2}$. Furthermore if $\lambda_k$ is the $k$-th eigenvalue of this problem and $\lambda_{k,2}$ the $k$-th eigenvalue of \cref{EVP2} considered on $\M\sind{2}$ alone, then for $k>1$, $\lambda_k = \lambda_{k-1,2}$. Furthermore, the corresponding eigenfunctions satisfy
$u_k = u_{k-1,2}$ on $\M\sind{2}$ and $u_k = 0$
  on $\M\sind{1}$ for $k>2$.

\begin{theorem} \label{thm:unnormalized_spectral}
Suppose \cref{ass:densities,ass:epsilon,ass:eta} hold.
Let $\lambda^{(n)}_k$ be the $k$-th eigenvalue of $\eps_n^{-d\sind{2}} L^U_{n,\eps_n}$ and let $u^{(n)}_k$ be the associated eigenfunction, normalized in $L^2(\mu_n)$. 
Let $\lambda_k$ be be the $k$-th eigenvalue for \cref{EVP2}.

Then almost surely
\begin{listi}
\item (Convergence of eigenvalues) For every $k \in \N$,
    \begin{align*} 
       \lim_{n \to \infty} \lambda^{(n)}_k = \lambda_k.
    \end{align*}
\item (Convergence of eigenfunctions)
  The sequence of eigenfunctions $\{u^{(n)}_k\}_{n\in\N}$ (as well as any of its subsequences) has a convergent subsequence and the limit of any such subsequence is a normalized eigenfunction for \cref{EVP1} corresponding to eigenvalue $\lambda_k$.
\end{listi}
\end{theorem}
                        We remark that in the cases of different dimensions a result of similar nature has been obtained in  \cite[Theorem 10]{trillos2023large}, where the authors also obtain error estimates for eigenvalues and eigenvectors.

Similarly to the  proof of \cref{thm:spectral_normalized}, the proof of 
\cref{thm:unnormalized_spectral} is a straightforward modification 
of arguments in the proof of \cite[Theorems 1.2]{GTS18}.

Here we make a remark regarding the eigenvectors corresponding to the first two eigenvalues when  $d\sind{1}<d\sind{2}$. Then the limiting problem has eigenspace of dimension two corresponding to eigenvalue zero, spanned by normalized eigenfunctions $u_1 = c \chi_{\M}$ and $u_2 = c_1 \chi_{\M\sind{1}} + c_2 \chi_{\M\sind{2}}$.
Note that $\lambda_1^{(n)} =0$ with high probability for all $n$ large and thus, due to the connectivity of the graph, $u_1^{(n)}$ is a constant function for all $n$ large. Hence, any subsequential limit of $\{u_1^{(n)}\}_{n=1, 2, \dots}$ belongs to $\pm u_1$. Thus the convergence of eigenvectors of \cref{thm:unnormalized_spectral} and orthogonality imply that  any subsequential limit of $\{u_2^{(n)}\}_{n=1, 2, \dots}$ belongs to $\pm u_2$. In particular we see that $u_2^{(n)}$ separate the manifolds when $n$ is large, as expected. We note that this is not yet the case for $n$ and $\eps$ considered in experiments for \cref{fig:unnormalized}.

\subsection{Numerical illustrations} \label{sec:numerics}

In this section we illustrate some of the findings on a union of a 2D rectangle and a line segment, $\M=\M\sind{1} \cup \M\sind{2}$ where $\M_1 = \{0\} \times \{0\} \times [-0.65, 0.65]$ and $\M_2 = [-0.7,0.7] \times [-0.5, 0.5] \times \{0\}$. While this does not exactly fit the assumptions of our theorems as $\M\sind{1}$ and $\M\sind{2}$ have a boundary, we note that the theory is expected to carry over to manifolds with boundary where for eigenvalue problem one obtains Neumann  conditions at the boundary, as \cite{GTS18} establishes for Euclidean domains.

\begin{figure}[h!]
\centering
\subfloat[Eigenvector $u_2^{(n)}$]{\includegraphics[width=0.5\textwidth]{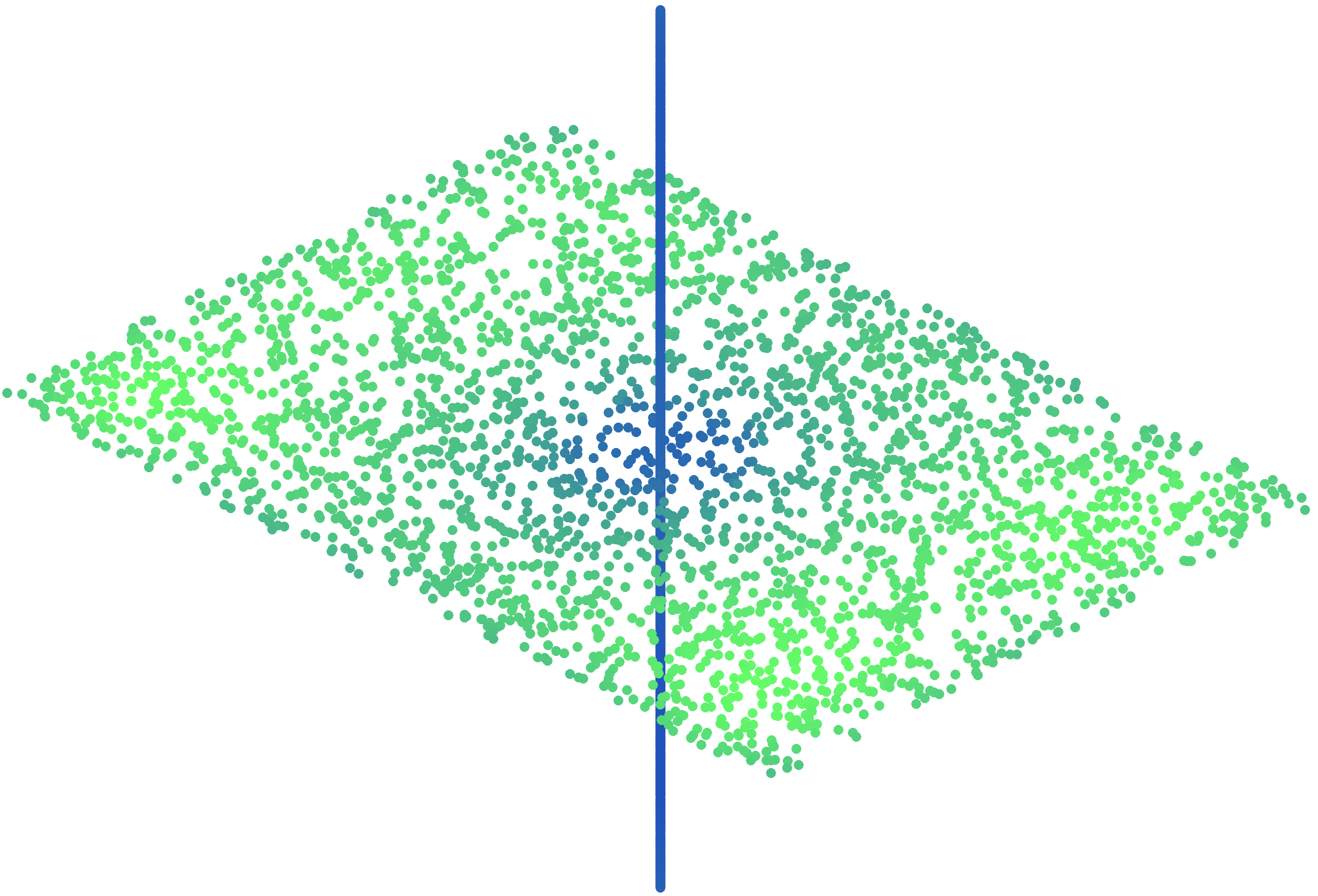}}
\subfloat[Eigenvector $u_3^{(n)}$]{\includegraphics[width=0.5\textwidth]{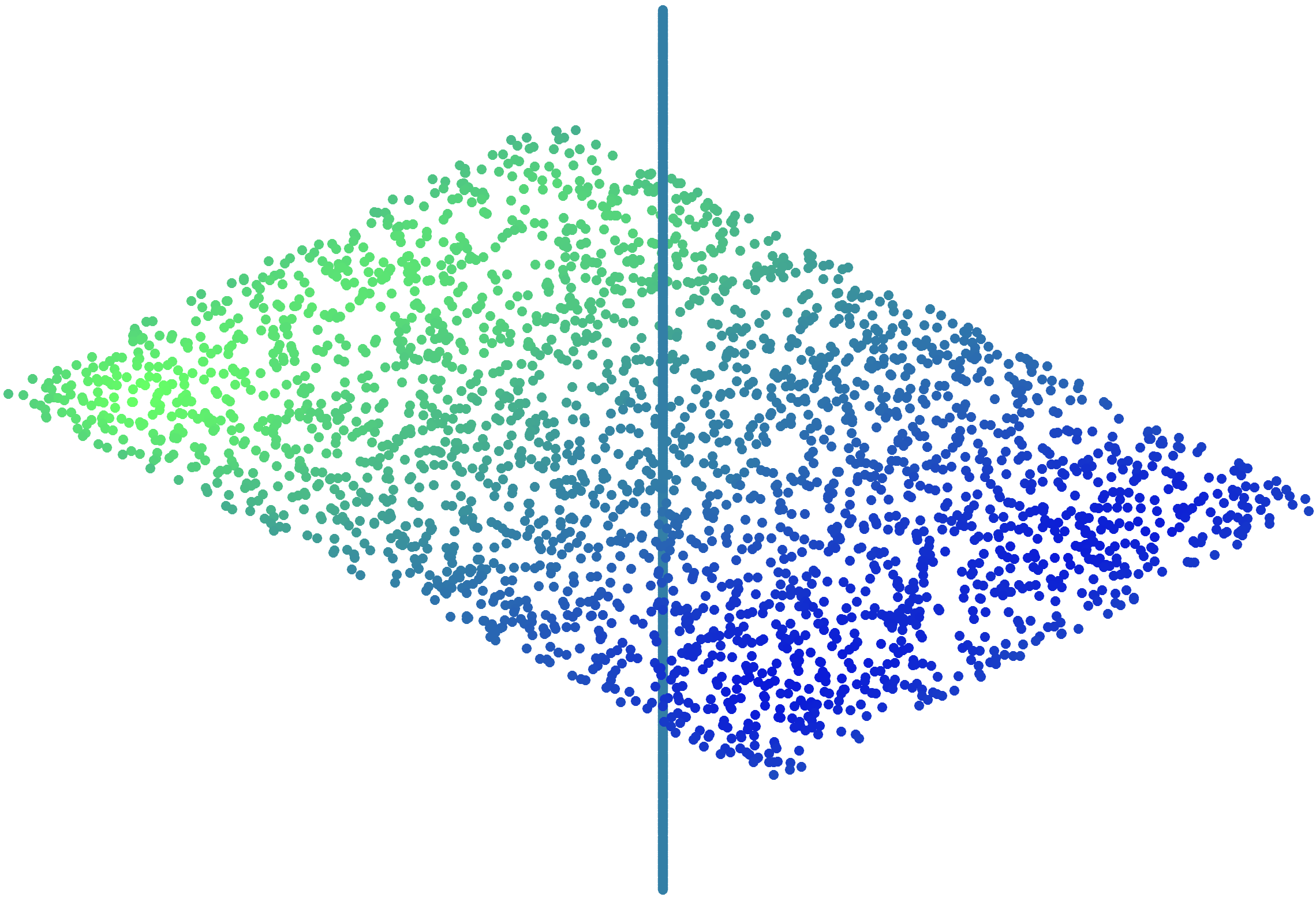}} \\
\subfloat[Eigenvector $u_4^{(n)}$]{\includegraphics[width=0.5\textwidth]{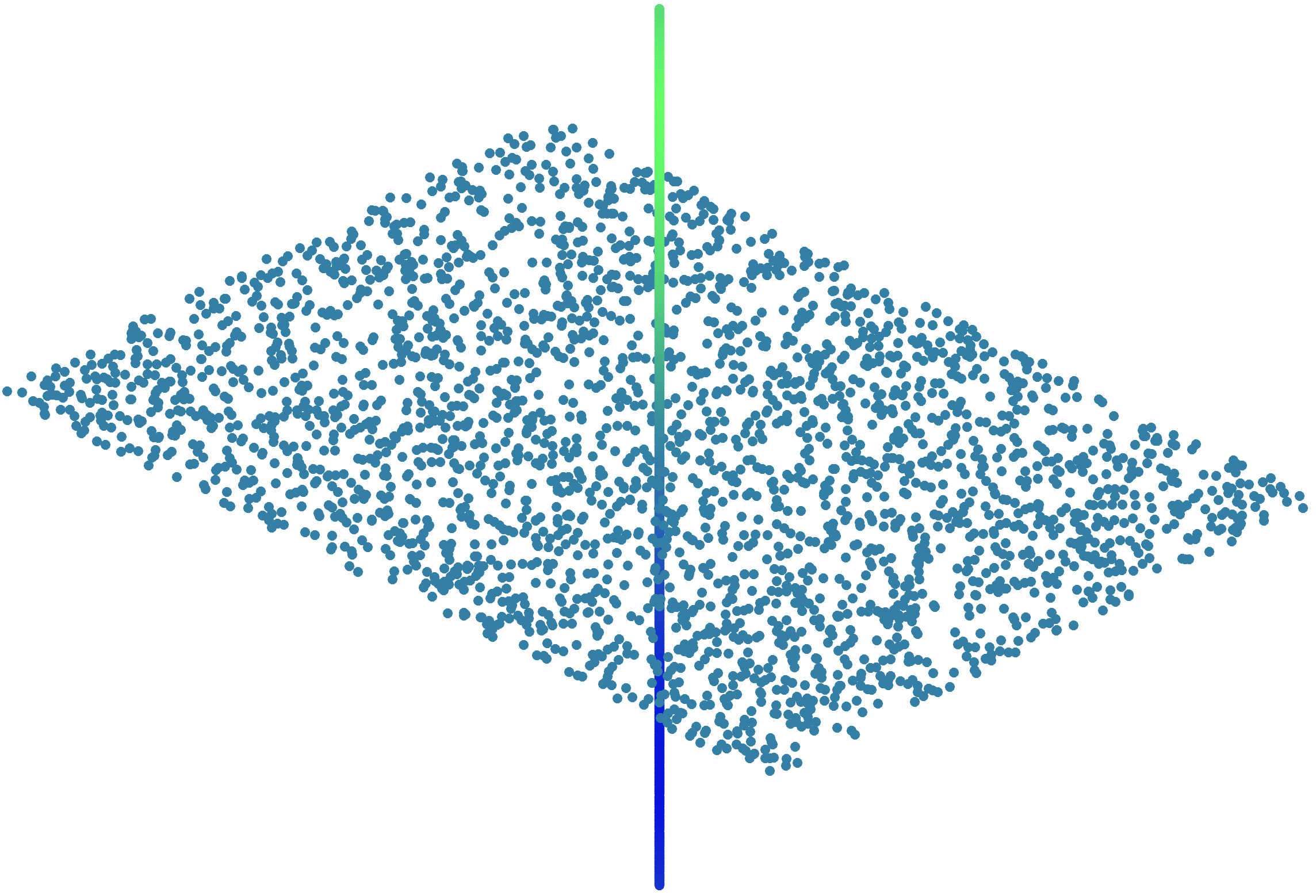}}
\subfloat[Eigenvector $u_5^{(n)}$]{\includegraphics[width=0.5\textwidth]{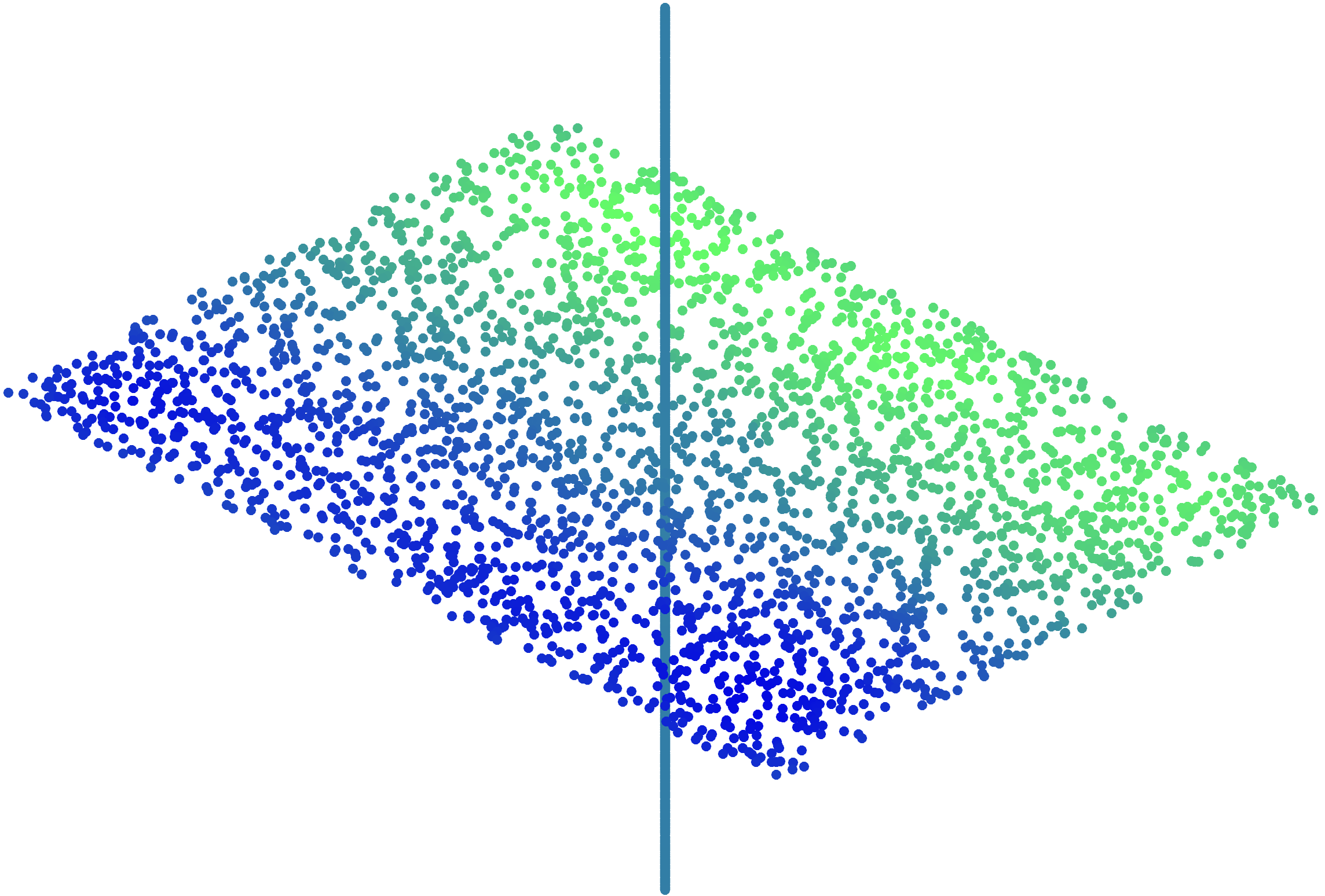}} 
\caption{Second to fifth eigenvector of the normalized Laplacian \labelcref{def:normalized_graph_Lapl} for $\eta = \chi_{[0,1]}$, $\eps = 0.13$,
 $n\sind{1}=2,600$ and $n\sind{2}=2,800$.}
\label{fig:normalized} 
\end{figure}

In our experiments we take $\eta = \chi_{[0,1]}$ and $\eps=0.13$ and hence use random geometric graphs. Direct computation gives $\sigma^{(1)} = \nicefrac{2}{3}$, $\beta^{(1)} = 2$, $\sigma^{(2)} = \nicefrac{\pi}{4}$, and $\beta^{(2)} = \pi$. We consider uniform measure on both the rectangle and the line segment.
Thus for the normalized Laplacian the first few eigenvalues of the problem on $\M\sind{1}$ are $0,\frac{\pi^2}{3\cdot 1.69},\frac{4\cdot \pi^2}{3\cdot 1.69}, \dots$, and for the problem on $\M\sind{2}$ are $0,\frac{\pi^2}{4\cdot 1.96},\frac{\pi^2}{4}, \frac{\pi^2}{4\cdot 1.96} + \frac{\pi^2}{4}, \dots$.
Comparing the values gives that the eigenvalues for the joint problem \cref{eq:Ev} on $\M$ are $0,0,\frac{\pi^2}{4\cdot 1.96},\frac{\pi^2}{3\cdot 1.69},\frac{\pi^2}{4},\frac{\pi^2}{4\cdot 1.96} + \frac{\pi^2}{4}, \dots$. We observe that this is in agreement with \cref{fig:normalized}, where we see that \red the second eigenvector is approximately piecewise constant, taking distinct values on the segment $\M\sind{1}$ and on the bulk of the rectangle $\M\sind{2}$, with a transition layer contained in $\M\sind{2}$ near the intersection, and thus approximately separates \nc
the manifolds, the third and fifth eigenvectors correspond to Laplacian eigenfunction on the rectangle and fourth eigenvector corresponds to the first nonconstant Laplacian eigenfunction on the line segment, where all problems are considered with Neumann boundary conditions. In particular we see that the normalized Laplacian is able to see variations in data of both dimensions. 
Let us also remark that the fact that for $u_2^{(n)}$ the values on $\M\sind{1}$ are almost constant and the transition between the values $\M\sind{1}$ and the bulk of $\M\sind{2}$ happens in $\M\sind{2}$, close to the intersection is similar to the nature of transition layers constructed in the limsup arguments in \cref{sec:NL_limsup} and \cref{sec:limsup}.

The eigenvectors for the unnormalized graph Laplacian are displayed on \cref{fig:unnormalized}. We observe that only variations within the rectangle are captured by the first seven eigenvectors. This is in agreement with the theoretical prediction. From the perspective of using unnormalized Laplacians in machine learning tasks, this is undesirable,  as the information contained in $\M\sind{1}$ is largely lost (beyond just recognizing that data belong to $\M\sind{1}$). We note that the almost piecewise constant eigenvector corresponds to the third eigenvalue, while our results show that such vectors could converge to the second eigenfunction of the limiting eigenvalue problem as the number of points goes to infinity and $\eps_n\to 0$. The reason for this discrepancy in the nature of the second eigenvector/eigenfunction is that we are still far from the limiting regime when $\eps \to 0$; in particular the transition layers in $u_2^{(n)}$ and $u_3^{(n)}$ on \cref{fig:unnormalized} have comparable lengths. From the perspective of machine learning tasks it would be desirable if the second eigenvector was able to distinguish the manifolds. In our experiments, we observed that this is more of an issue for the unnormalized Laplacian than for the normalized Laplacian. Heuristic explanation is that the degrees of the nodes are high near the intersection, so for the normalized Laplacian the differences are suppressed near the intersection making the energy of the transition layer lower. 

\begin{figure}[h!]
\centering
\hspace*{-3pt}
\subfloat[Eigenvector $u_2^{(n)}$]{\includegraphics[width=0.34\textwidth]{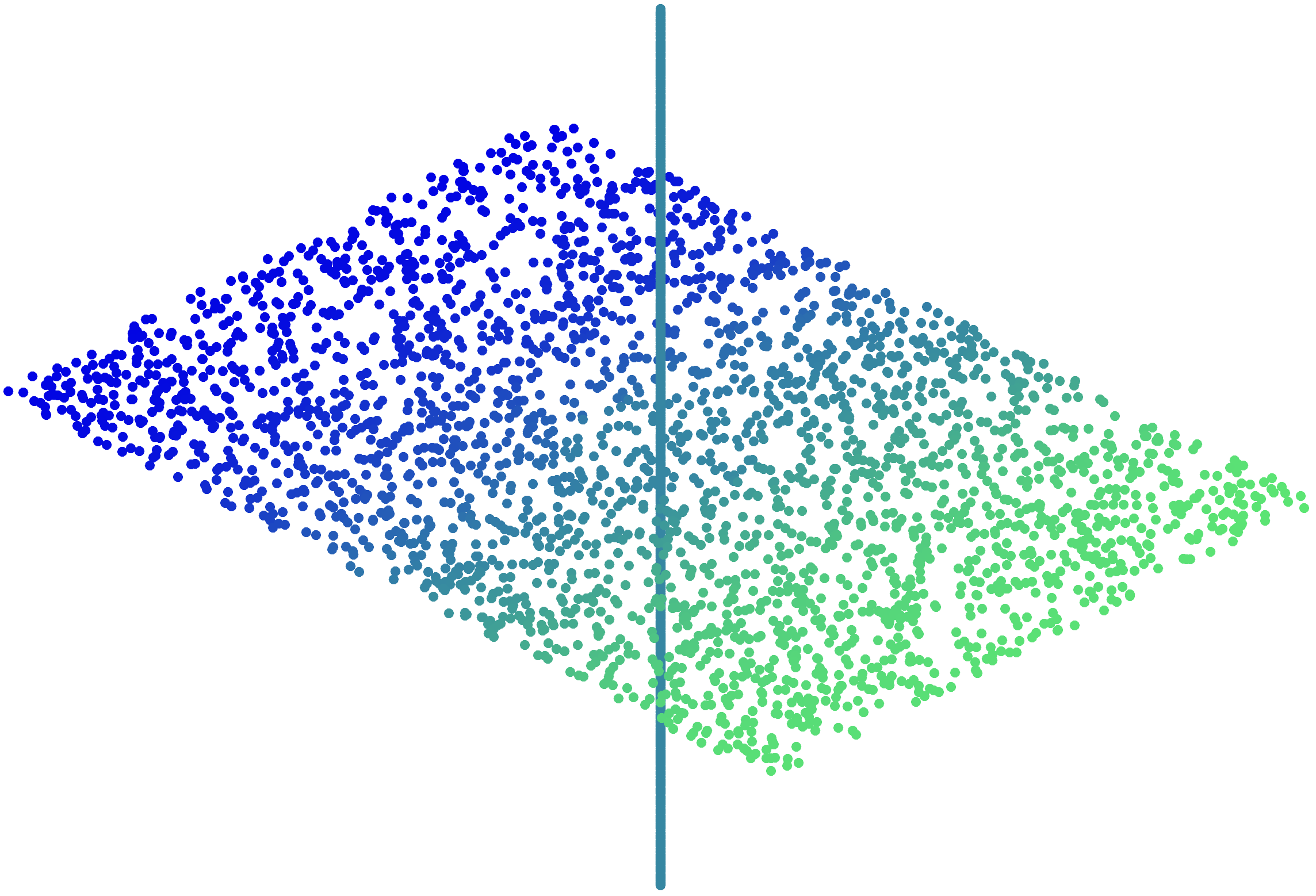}}
\hspace*{-3pt}
\subfloat[Eigenvector $u_3^{(n)}$]{\includegraphics[width=0.34\textwidth]{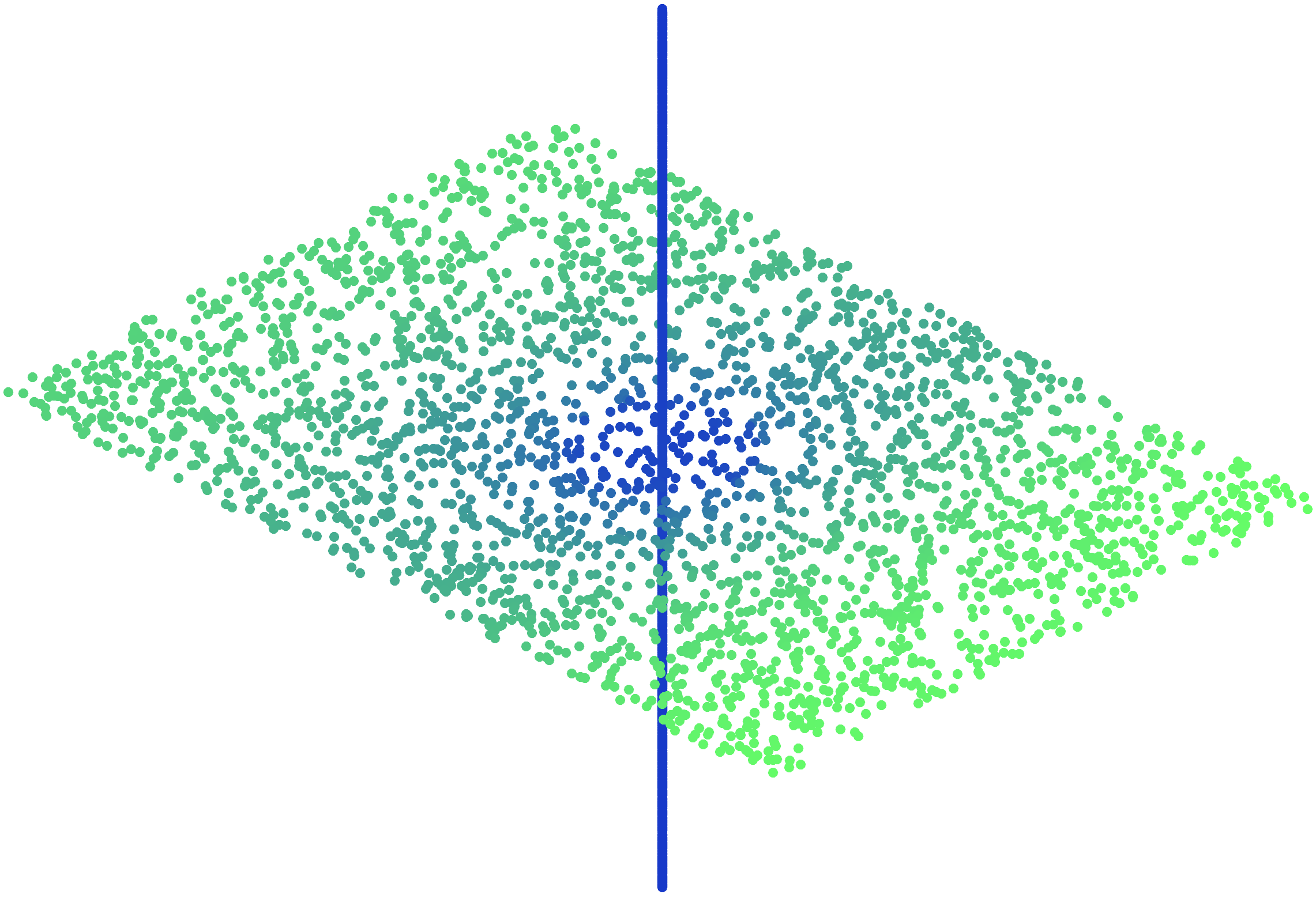}} 
\hspace*{-3pt}
\subfloat[Eigenvector $u_4^{(n)}$]{\includegraphics[width=0.34\textwidth]{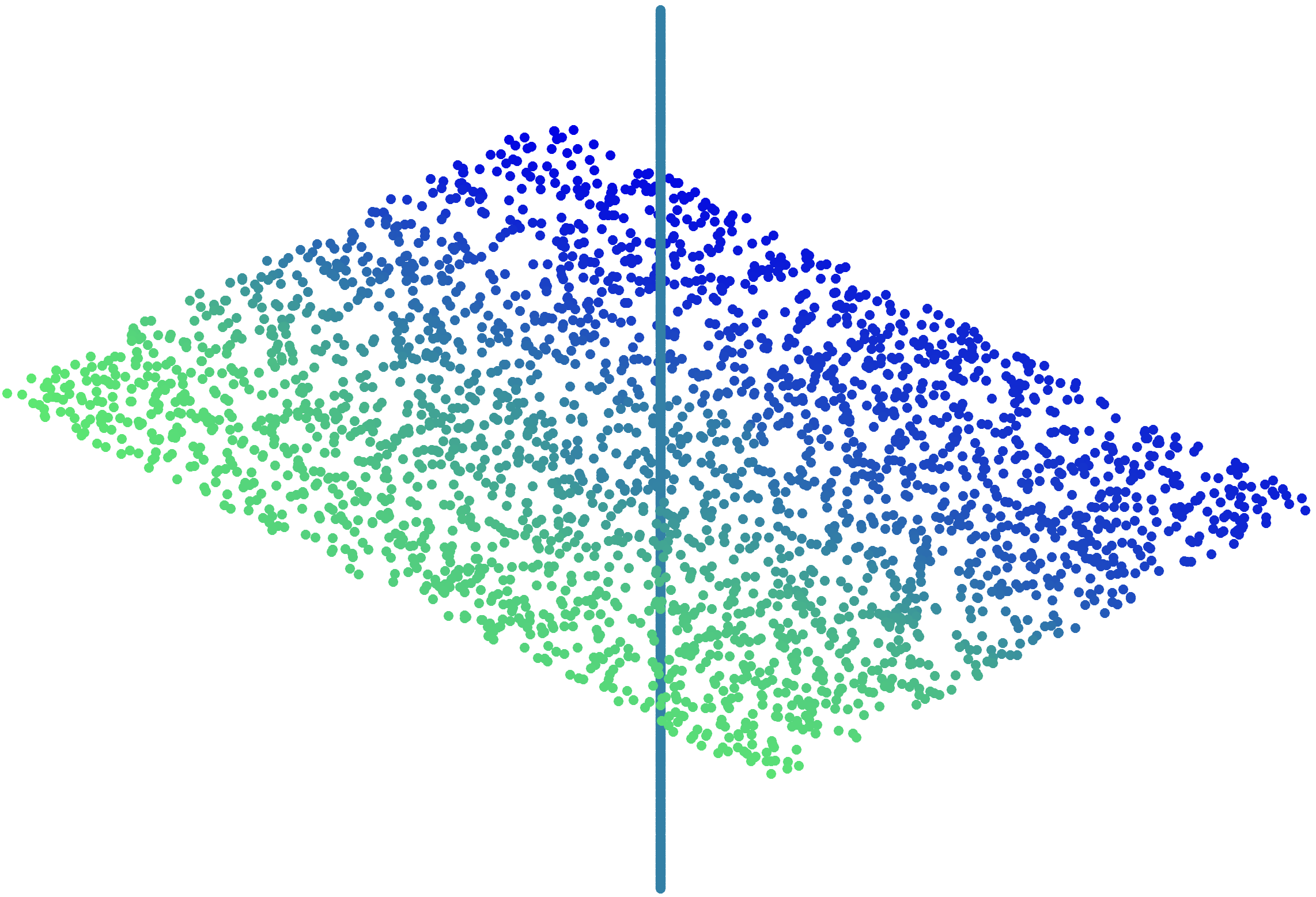}} \\
\hspace*{-3pt}
\subfloat[Eigenvector $u_5^{(n)}$]{\includegraphics[width=0.34\textwidth]{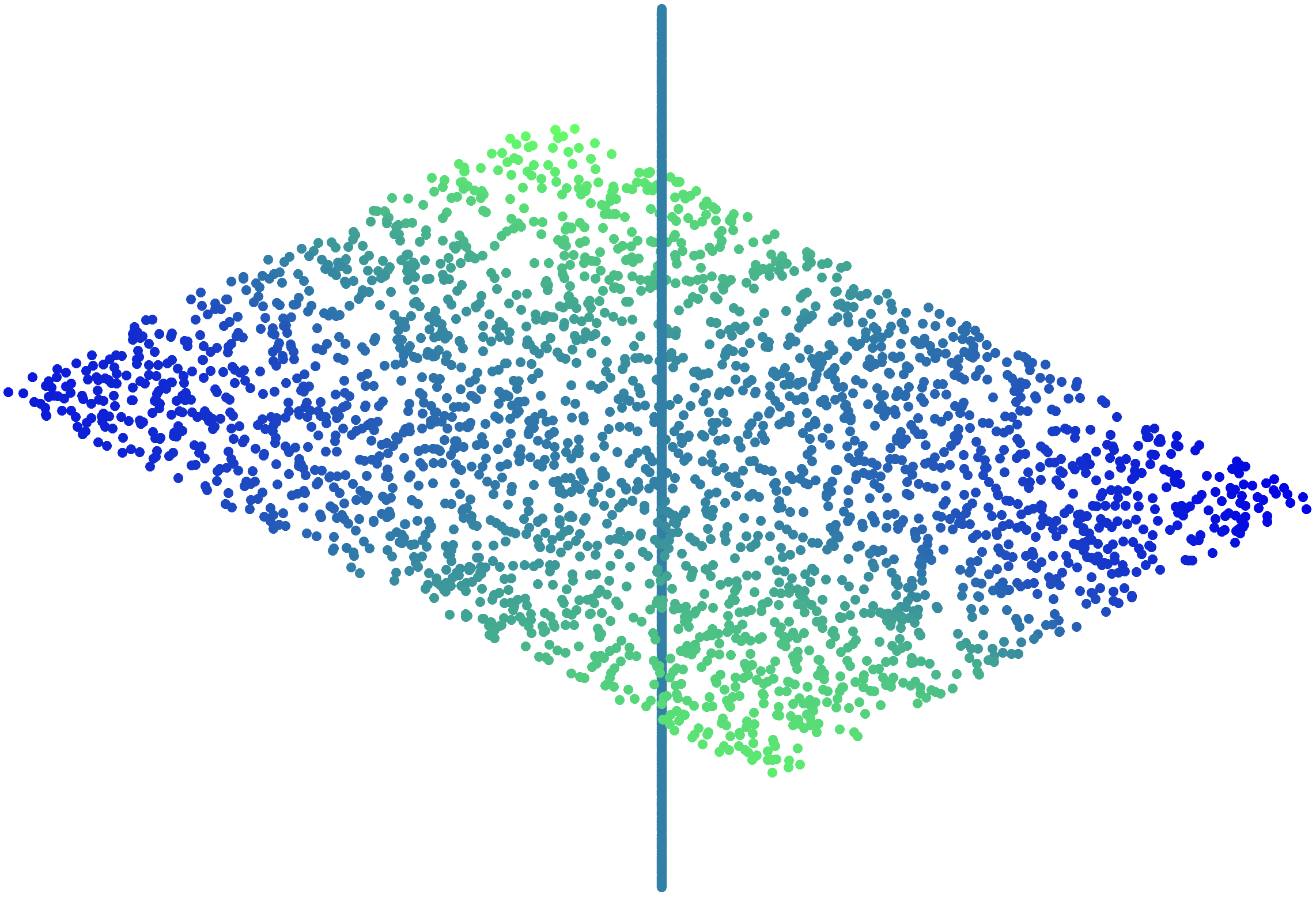}}
\hspace*{-3pt}
\subfloat[Eigenvector $u_6^{(n)}$]{\includegraphics[width=0.34\textwidth]{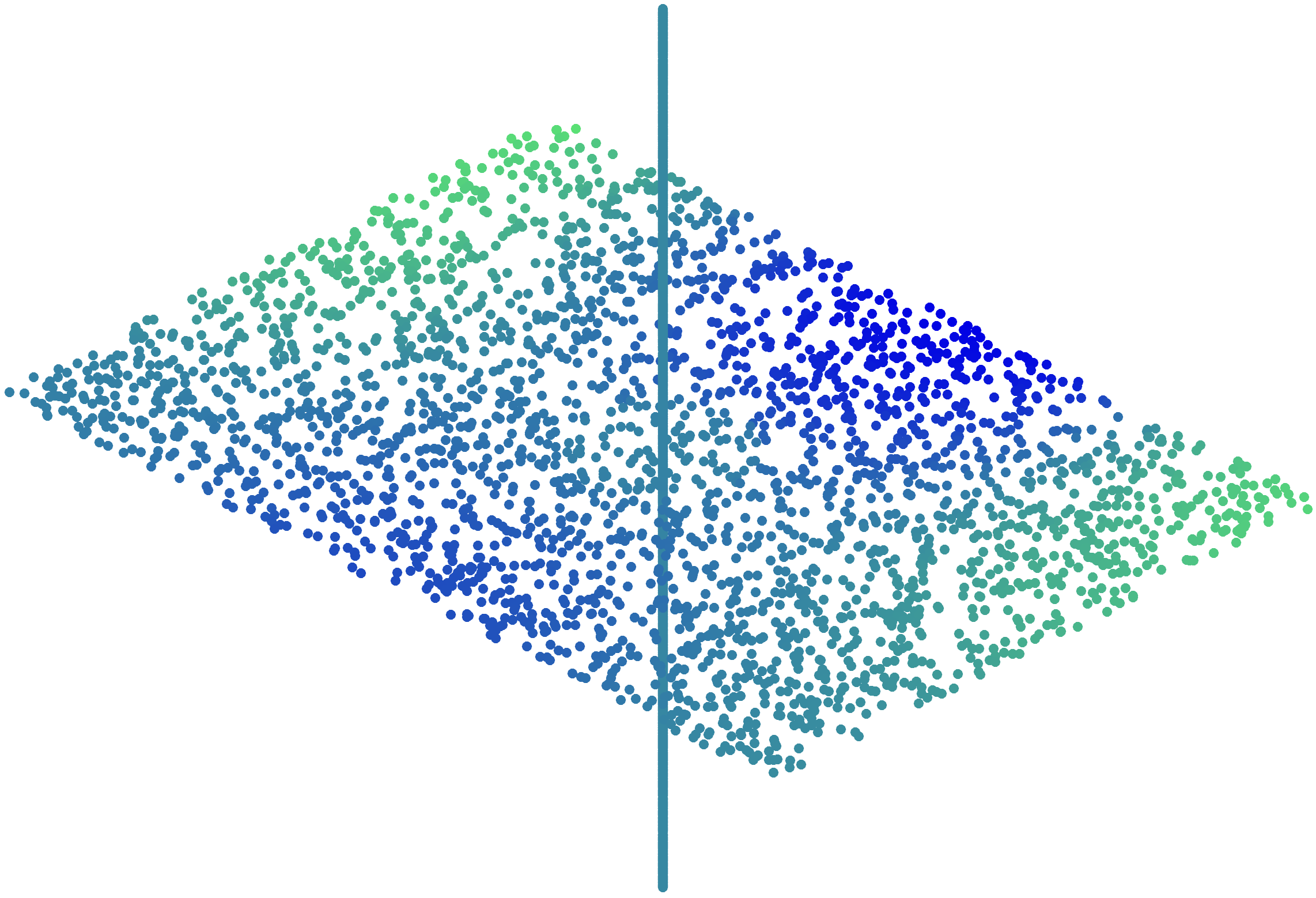}} 
\hspace*{-3pt}
\subfloat[Eigenvector $u_7^{(n)}$]{\includegraphics[width=0.34\textwidth]{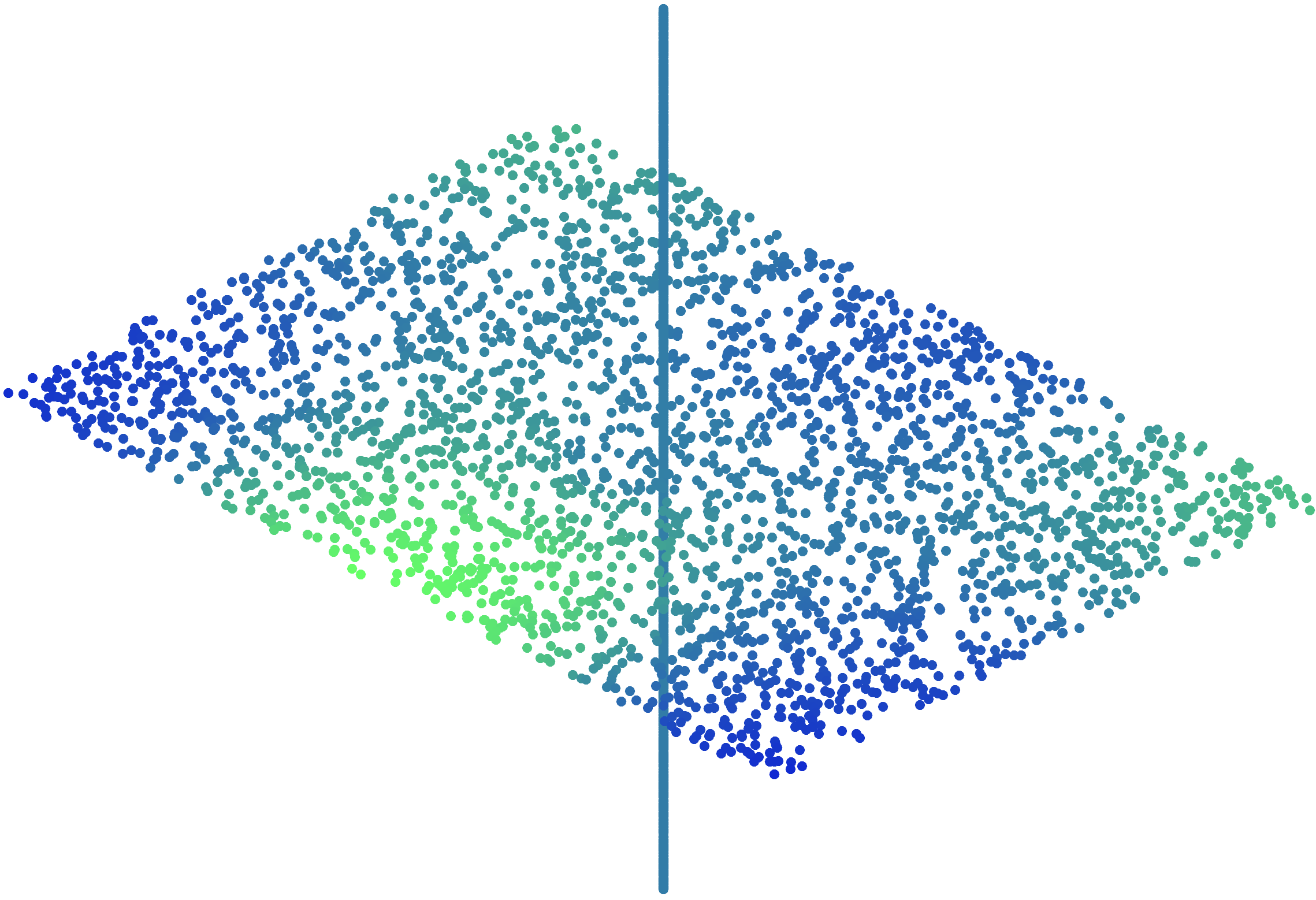}} 
\caption{Second to seventh eigenvector of the unnormalized Laplacian \labelcref{def:unnormalized_graph_Lapl} for $\eta = \chi_{[0,1]}$, $\eps = 0.13$, $n\sind{1}=2,600$, and $n\sind{2}=2,800$.}
\label{fig:unnormalized} 
\end{figure}
 \section{Continuum nonlocal-to-local limit}
\label{sec:nonlocal}

In this section, we consider nonlocal normalized Dirichlet energies, $\E_\eps$, on unions of manifolds and their limit $\E$ as the bandwidth $\eps \to 0$. This formally corresponds to the case $n=\infty$ in the graph-based setting of the previous section. These results are of independent interest and add to the rich field connecting nonlocal and local energies that we discussed in the introduction. 
However, our main aim here is to provide the key ideas of the proof of our main results \cref{thm:compactness,thm:gamma}, in an accessible way. 
    
Let us first recall the nonlocal and local normalized Dirichlet energies, defined in \labelcref{eq:nonlocal_fctl,eq:limit_fctl}, respectively: for $u\in L^2(\M)$
\begin{align*}
    \E_\eps(u) 
    &= 
    \frac{1}{\eps^2} \int_\M\int_\M\eta_\eps (|x-y|) \abs{\frac{u(x)}{\sqrt{\eta_\eps(\abs{x-\cdot})\star\mu}}-\frac{u(y)}{\sqrt{\eta_\eps(\abs{y-\cdot})\star\mu}}}^2\d\mu(x)\de\mu(y),
    \\
    \\
    \E(u) 
    &= 
    \begin{dcases}
    \sum_{i=1}^2
    \frac{\sigma_\eta\sind{i}}{\beta_\eta\sind{i}}
    \int_{\M\sind{i}}\abs{\nabla_{\M\sind{i}} \left( \frac{u}{\sqrt{\alpha\sind{i} \rho\sind{i}}} \right) \alpha\sind{i} \rho\sind{i}}^2 \d\vol\sind{i}
    \quad &\te{if } u\in H^1_{\sqrt \mu}(\M),  \\
    +\infty   \quad &\te{else}.
    \end{dcases}
\end{align*}
The proofs of $\Gamma$-convergence on union of manifolds uses the existing results of $\Gamma$-convergence on a single manifold. For the nonlocal-to-local convergence of the nonlocal Dirichlet energy in the continuum setting the relevant $\Gamma$-convergence result is \cref{thm:Gamma-convergence_standard_NL} from~\cite{laux2025large}.
 Similarly, the proofs of our main results \cref{thm:compactness,thm:gamma}, which we present in \cref{sec:proofs_graphs}, use \cref{thm:Gamma-convergence_standard}. 
in the appendix which provides a $\Gamma$-convergence and a compactness property for the standard graph Dirichlet energy on a single manifold. 
More precisely we shall use \cref{cor:Gamma-convergence_standard} which generalizes \cref{thm:Gamma-convergence_standard} to Binomial point processes.

We first state the compactness result, which is similar to \cref{thm:compactness}.
\begin{theorem}\label{thm:compactness_NL}
Under \cref{ass:densities,ass:eta} for any sequence $(\eps_n)_{n\in\N}\subset(0,\infty)$ converging to zero and any sequence of functions $(u_n)_{n \in N} \subset L^2(\M)$ satisfying
\begin{align*}
    \sup_{n\in\N} \E_{\eps_n}(u_n)<\infty,\qquad \sup_{n\in\N} \norm{u_n}_{L^2(\M)}<\infty,
\end{align*}
there exists a subsequence of  which converges in $L^2(\M)$.
\end{theorem} 
The main result of this section is the $\Gamma$-convergence of nonlocal normalized Dirichlet energies.
\begin{theorem}\label{thm:Gamma_NL}
    Suppose that \cref{ass:densities,ass:eta} hold and $(\eps_n)_{n\in\N}\subset(0,\infty)$ converges to zero. Then $\E_{\eps_n}$ $\Gamma$-converges in $L^2(\M)$ to $\E$.
\end{theorem}
For proving these theorems it will be convenient to denote by $\mathcal G_{\eps}(\cdot;\M\sind{i})$ and $\mathcal G(\cdot;\M\sind{i})$ the nonlocal and the standard Dirichlet energy from \cref{thm:Gamma-convergence_standard_NL}, defined in the manifold $\M\sind{i}$ and equipped with the probability measures $\mu\sind{i}=\rho\sind{i}\vol\sind{i}$ for $i\in\{1,2\}$. 

\subsection{Proof of compactness for nonlocal normalized Dirichlet energies}
\label{sec:NL_compactness}

Here we provide a proof of \cref{thm:compactness_NL}, where the main difficulty arises when proving compactness on the larger dimensional manifold $\M\sind{2}$ close to its intersection with $\M\sind{1}$, where the convolutions $x\mapsto\eps_n^{-d\sind{2}}\eta_{\eps_n}(\abs{x-\cdot})\star\mu$ scale like $\eps_n^{d\sind{1}-d\sind{2}}$ and hence are potentially unbounded.

\begin{proof}[Proof of \cref{thm:compactness_NL}]
    Defining the sequence of functions $v_n \in L^2(\M)$ via
    $$
    v_n(x) := \frac{u_n(x)}{\sqrt{\eps_n^{-d\sind{i}}\eta_{\eps_n}(\abs{x-\cdot})\star\mu}}
    $$ 
    for $x\in\M\sind{i}$ and $i\in\{1,2\}$, and using that $\mu\restr\M\sind{i}=\alpha\sind{i}\mu\sind{i}$ we have that
    \begin{align*}
        \E_{\eps_n}(u_n) 
        \geq 
        \left(\alpha\sind{1}\right)^2\mathcal G_{\eps_n}(v_n; \M\sind{1})
        +
        \left(\alpha\sind{2}\right)^2
        \mathcal G_{\eps_n}(v_n; \M\sind{2}).
    \end{align*}
    In the case that the manifolds have the same dimension $d\sind{1}=d\sind{2}=:d$ we can bound the $L^2$-norm of $v_n$ by noting that thanks to \cref{ass:eta} there exists $\tau,C_\eta>0$ such that $\eta(t)\geq C_\eta$ for all $t\leq\tau$ and hence we have
    \begin{align*}
        \eta_{\eps_n}(\abs{x-\cdot})\star\mu 
        \geq 
        C_\rho 
        \int_{\M}\eta_{\eps_n}(\abs{x-y})
        \d\vol(y)
        \geq 
        C_\rho
        C_\eta
        \vol(B_{\tau\eps}(x)\cap\M)
        \sim 
        \eps^{d},
    \end{align*}
    where $a(\eps) \sim b(\eps)$ means $0 < \liminf_{\eps \to 0} a/b \leq \limsup _{\eps \to 0} a/b < \infty$. This
     implies $\norm{v_n}_{L^2(\M)}\leq C\norm{u_n}_{L^2(\M)}$ for a constant $C$ that depends on $\rho,\eta$ and $\M$.
     
    If the dimensions of the manifolds are unequal, we have
    \begin{align} \label{eq:convinM1scale}
         \eps_n^{-d\sind{1}}\eta_{\eps_n}(\abs{x-\cdot})\star\mu\sim O(\eps_n^{d\sind{2}-d\sind{1}}) + 1 \sim 1
    \end{align}
    for $x\in\M\sind{1}$ close to the intersection
    and
    \begin{align} \label{eq:convinM2scale}
    \eps_n^{-d\sind{2}}\eta_{\eps_n}(\abs{x-\cdot})\star\mu\sim \eps_n^{d\sind{1}-d\sind{2}}\to \infty
    \end{align}
    for $x\in\M\sind{2}$ close to the intersection.
    Hence, $\norm{v_n}_{L^2(\M)}$ is bounded in all cases, and thus \cref{thm:Gamma-convergence_standard_NL} implies that (up to a subsequence that we do not relabel)  $v_n\to v$ in $L^2(\M)$.
    Taking \labelcref{eq:convinM1scale} and the definition of $v_n$ into account, the dominated convergence theorem implies $u_n\to u$ in $L^2(\M\sind{1})$, where
    \begin{align}\label{eq:def_u_compactness_NL}
        u(x):=\sqrt{\alpha\sind{i}\beta_\eta\sind{i}\rho\sind{i}(x)}\, v(x),
        \qquad\; \te{for }
        x\in\M\sind{i}.
    \end{align}
    \red Here, and throughout the proofs, $\alpha\sind{i}$ is as in \labelcref{eq:rho} and $\beta_\eta\sind{i}$, $\sigma_\eta\sind{i}$ are the dimension-dependent constants defined in \labelcref{eq:sigma_beta}. \nc
    To argue that this convergence also holds true in $L^2(\M\sind{2})$, we first introduce some notation:
    Let $A_{\eps_n} := \M\sind{1} \cup \{ x \in \M\sind{2} \st  d(x,\M\sind{1}) > \eps_n\}$ and note that on $A_{\eps_n}$, for $x \in \M\sind{i}$
    \begin{align} \label{eq:convinAscale}
         \eps_n^{-d\sind{i}}\eta_{\eps_n}(\abs{x-\cdot})\star\mu\sim 1
    \end{align}
    Define, furthermore, $u^{good}_n := \chi_{A_{\eps_n}} u_n$ and $u^{bad}_n := u_n - u^{good}_n$.

    Let us make a few observations:
    Since $v_n \to v$ in $L^2(\M)$,  we also have $\chi_{A_{\eps_n}}  v_n \to v$ in $L^2(\M)$. 
    Furthermore, \red by the uniform convergence on the manifold of the mollified density built from the ambient-distance kernel, established with explicit rates in \cite[Theorem~2.2]{wu2022strong}, \nc we have
    \begin{align} \label{eq:mollification}
        \eps_n^{-d\sind{i}}\eta_{\eps_n}(\abs{x-\cdot})\star\mu\restr \M\sind{i} \to  \alpha\sind{i}\beta_\eta\sind{i}\rho\sind{i}(x) \quad\, \te{uniformly in } \M\sind{i} \te{ as } n\to \infty. 
    \end{align}
    Thus, due also to \labelcref{eq:convinM1scale}, the following function on $\M$
    \begin{equation} \label{eq:molli_uniform}
     x\mapsto\chi_{A_{\eps_n}}(x)  \eps_n^{-d\sind{i}}\eta_{\eps_n}(\abs{x-\cdot})\star\mu - \chi_{A_{\eps_n}}(x)  \alpha\sind{i}\beta_\eta\sind{i}\rho\sind{i}(x)
    \end{equation}
    (where $i$ corresponds to the manifold $x$ belongs to) converges uniformly to zero. 
    From these observations, and taking the definition of $v_n$ and \labelcref{eq:convinAscale} into account, it follows that 
    $u^{good}_n$ converges in $L^2(\M)$ to $u\in L^2(\M)$, defined in \labelcref{eq:def_u_compactness_NL}.
     So to prove $L^2(\M)$ convergence of $u_n$ (along a subsequence) it remains to show that $u^{bad}_n \to 0$ in $L^2(\M\sind{2})$, at least along a subsequence. 
    Note that we already know $u^{bad}_n\to 0$ in $L^2(\M\sind{1})$ since $u_n\to u$ and $u_n^{good}\to u$ in the same topology.
     
    The idea that guides the remainder of the proof is the following: If the $L^2$ norm of $u^{bad}_n$ is bounded from below by a positive number, since the energy $\E_{\eps_n}(u_n)$ is small, the cross-term of the energy would also force the $L^2$ norm of $u_n$ restricted to a vanishing (as $n \to \infty)$ region of $\M\sind{1}$ close to the intersection to be bounded from below by a positive number. However, this is not possible since $u^{good}_n$ converges in $L^2(\M\sind{1})$. 
    The precise argument is a bit more complicated since we are not able to execute the above in one step, but instead iterate over an expanding sequence of sets. 

    Let us define the ``good'' set $G^1_{\eps_n} := \{x \in \M\sind{1} \st d(x, \M\sind{2}) < \eps_n\}$ on which we already know that $u_n\to u$ as well as the ``strip'' 
    $S^1_{\eps_n} := \{x \in \M\sind{2} \st d(x, \M\sind{1}) < \eps_n/2\}$ of points where we want to deduce this convergence.
    For $k>1$ we update the good and the strip sets iteratively via
    $G^k_{\eps_n} := G^{k-1}_{\eps_n} \cup S^{k-1}_{\eps_n}$ and 
    $S^k_{\eps_n} = \{x \in \M\sind{2} \st d(x, G^{k-1}_{\eps_n}) < {\eps_n}/2 \}$.

    Let $\bar k>1$ be the smallest $k$ such that for all $n\in\N$, $\{ x \in \M\sind{2} \st  d(x, \M\sind{1}) \leq {\eps_n}\} \subset G^k_{\eps_n}$. We know by \cref{lem:angle} that $\bar k$ is finite, but that depending on the angle of intersection of manifolds may be bigger than two. 

    We show by induction that $\|u_n\|_{L^2(G^k_{\eps_n})} \to 0$ as $n \to \infty$. Since
    $\M\setminus A_{\eps_n} = \{ x \in \M\sind{2} \st  d(x, \M\sind{1})\} \leq \eps_n \} \subset G^{\bar k}_{\eps_n}$ this implies that $\|u^{bad}_n\|_{L^2(\M)} \to 0$, which is the desired conclusion.

    \paragraph{Case $k=1$.} We note that the fact that $\|u_n\|_{L^2(G^1_{\eps_n})} \to 0$ follows from the fact that $u_n$ restricted to $\M\sind{1}$ converge in $L^2$.

    \paragraph{Case $k>1$.} Observe that for all $k>1$, due to the definition of the sets $G^k_\eps$ and $S^k_\eps$, and \cref{ass:eta}(iii), we have that for all $x \in S^k_{\eps_n}$, 
    $\;\eta_{\eps_n}(\abs{x-\cdot})\star \chi_{G^k_{\eps_n}} > c_k>0$ and hence
    $\;\eta_{\eps_n}(\abs{x-\cdot})\star (\mu\restr{G^k_{\eps_n}}) \geq C_k \eta_{\eps_n}(\abs{x-\cdot})\star \mu $ for some $C_k>0$.
    Furthermore, for all $x\in\M$ we trivially have that $\eta_{\eps_n}(\abs{x-\cdot})\star(\mu\restr S_{\eps_n}^k)\leq \eta_{\eps_n}(\abs{x-\cdot})\star\mu$.     
    Therefore
    \begin{align*}
    C_k & \|u_n\|^2_{L^2(\mu\restr{S^k_{\eps_n}})}  - \|u_n\|^2_{L^2(\mu\restr{G^k_{\eps_n}})} 
    \\
    & 
    \leq 
    \int_{S^k_{\eps_n}} \frac{\eta_{\eps_n}(\abs{x-\cdot})\star (\mu\restr{G^k_{\eps_n}})}{\eta_{\eps_n}(\abs{x-\cdot})\star \mu } \,u_n^2(x) \de\mu(x)
    - \int_{G^k_{\eps_n}} \frac{\eta_{\eps_n}(\abs{y-\cdot})\star (\mu\restr{S^k_{\eps_n}})}{\eta_{\eps_n}(\abs{y-\cdot})\star \mu } \, u_n^2(y)\de\mu(y)  
    \\
    & = 
    \int_{S^k_{\eps_n}} \int_{G^k_{\eps_n}} \eta_{\eps_n}(\abs{x-y}) \left( \frac{u_n^2(x)}{\eta_{\eps_n}(\abs{x-\cdot})\star \mu } - \frac{u_n^2(y)}{\eta_{\eps_n}(\abs{y-\cdot})\star \mu }\right)\de\mu(y)\de\mu(x) \\
    & \leq \sqrt{ \int_{S^k_{\eps_n}} \int_{G^k_{\eps_n}}   \eta_{\eps_n}(\abs{x-y}) \abs{\frac{u_n(x)}{\sqrt{\eta_{\eps_n}(\abs{x-\cdot})\star \mu }} - \frac{u_n(y)}{\sqrt{\eta_{\eps_n}(\abs{y-\cdot})\star \mu}}}^2\de\mu(y)\de\mu(x)} \\
    & \phantom{\leq\;} \times \sqrt{ \int_{S^k_{\eps_n}} \int_{G^k_{\eps_n}}  \eta_{\eps_n}(\abs{x-y})  \left( \frac{u_n(x)}{\sqrt{\eta_{\eps_n}(\abs{x-\cdot})\star \mu }} + \frac{u_n(y)}{\sqrt{\eta_{\eps_n}(\abs{y-\cdot})\star \mu }}\right)^2\de\mu(y)\de\mu(x)} \\
    & \leq 2\eps_n \sqrt{ \E_{\eps_n}(u_n)} \|u_n\|_{L^2(\mu)}
    \end{align*}
    Therefore 
    \[ \|u_n\|^2_{L^2(\mu\restr{S^k_{\eps_n}})} \leq \frac{1}{C_k} \left( \|u_n\|^2_{L^2(\mu\restr{G^k_{\eps_n}})}  + 2 \eps_n \sqrt{ \E_{\eps_n}(u_n)} \|u_n\|_{L^2(\mu)} \right) \]
    Since, using also the induction assumption, the right-hand side converges to zero as $n \to \infty$, we conclude that $\|u_n\|^2_{L^2(\mu\restr{S^k_{\eps_n}})} \to 0$ and therefore $\|u_n\|^2_{L^2(\mu\restr{G^{k+1}_{\eps_n}})} \to 0$ as $n \to \infty$. 
                                      \end{proof}

For the full proof of \cref{thm:compactness} which follows similar ideas we refer to \cref{sec:compactness}.

\subsection{Proof of the nonlocal \texorpdfstring{$\Gamma$}{Gamma}-liminf inequality}

Here we prove the liminf inequality part of $\Gamma$-convergence claimed in \cref{thm:Gamma_NL}. That is we prove that  for any sequence $(u_n)\subset L^2(\M)$ with $u_n\overset{L^2}{\to}u$ for $u\in L^2(\M)$ it holds
\begin{align*}
        \E(u) \leq \liminf_{n\to\infty}\E_{\eps_n}(u_n).
\end{align*}
The proof follows similar ideas as used in proving \cref{thm:compactness_NL} in \cref{sec:NL_compactness}. 
Here, however, the main challenge consist in establishing the trace condition in the definition of $H^1_{\sqrt{\mu}}(\M)$ in the case of codimension~$1$.
       \begin{proof}[Proof of the liminf inequality for \cref{thm:Gamma_NL}]
    Let $v_n(x) := \frac{u_n(x)}{\sqrt{\eps_n^{-d\sind{i}}\eta_{\eps_n}(\abs{x-\cdot})\star\mu}}$ 
    as before and $v(x):=\frac{u(x)}{\sqrt{\alpha\sind{i}\beta_\eta\sind{i}\rho\sind{i}(x)}}$ for $x\in\M\sind{i}$ and $i\in\{1,2\}$. We first prove that $v_n\to v$ in $L^2(\M)$ by computing
    \begin{align*}
        &\phantom{{}={}}
        \lim_{n\to\infty}
        \sum_{i=1}^2
        \int_{\M\sind{i}}
        \abs{v_n-v}^2\d\vol\sind{i}
        \\
        &=
        \lim_{n\to\infty}
        \sum_{i=1}^2
        \int_{\M\sind{i}}
        \abs{u_n(x)-\sqrt{\frac{\eps_n^{-d\sind{i}}\eta_{\eps_n}(\abs{x-\cdot})\star\mu}{\alpha\sind{i}\beta_\eta\sind{i}\rho\sind{i}(x)}}u(x)}^2\frac{\d\vol\sind{i}(x)}{\eps_n^{-d\sind{i}}\eta_{\eps_n}(\abs{x-\cdot})\star\mu}
        \\
        &\leq 
        2
        \lim_{n\to\infty}
        \sum_{i=1}^2
        \int_{\M\sind{i}}
        \abs{u_n(x)-u(x)}^2\frac{\d\vol\sind{i}(x)}{\eps_n^{-d\sind{i}}\eta_{\eps_n}(\abs{x-\cdot})\star\mu}
        \\
        &\qquad
        +
        2
        \lim_{n\to\infty}
        \sum_{i=1}^2
        \int_{\M\sind{i}}
        \abs{\frac{\sqrt{\alpha\sind{i}\beta_\eta\sind{i}\rho\sind{i}(x)}-\sqrt{\eps_n^{-d\sind{i}}\eta_{\eps_n}(\abs{x-\cdot})\star\mu}}{\sqrt{\eps_n^{-d\sind{i}}\eta_{\eps_n}(\abs{x-\cdot})\star\mu} \sqrt{\alpha\sind{i}\beta_\eta\sind{i}\rho\sind{i}(x)}}}^2
        \abs{u(x)}^2
        \d\vol\sind{i}(x).
    \end{align*}
    Using that $u_n\to u$ in $L^2$ and employing the dominated convergence theorem together with \labelcref{eq:convinM1scale,eq:convinM2scale,eq:mollification} for the second term, we get that $v_n\to v$ in $L^2(\M\sind{i})$.

    By definition of $v_n$ and $v$ have    
    \begin{align*}
        \E_{\eps_n}(u_n) \geq 
        \left(\alpha\sind{1}\right)^2\mathcal G_{\eps_n}(v_n; \M\sind{1})
        +
        \left(\alpha\sind{2}\right)^2
        \mathcal G_{\eps_n}(v_n; \M\sind{2}).
    \end{align*}
    and, using the liminf inequality from \cref{thm:Gamma-convergence_standard_NL}, we obtain the inequality
    \begin{align*}
        \sum_{i=1}^2
        \frac{\alpha\sind{i}\sigma_\eta\sind{i}}{\beta_\eta\sind{i}}
        \int_{\M\sind{i}}
        \abs{\nabla_{\M\sind{i}}
        \left(
        \frac{u}{\sqrt{\rho\sind{i}}}
        \right)
        \rho\sind{i}
        }^2
        \d\vol\sind{i}
        &=
        \left(\alpha\sind{1}\right)^2\mathcal G_{\eps_n}(v; \M\sind{1})
        +
        \left(\alpha\sind{2}\right)^2
        \mathcal G_{\eps_n}(v; \M\sind{2})
        \\
        &\leq 
        \liminf_{n\to\infty}\E_{\eps_n}(u_n).
    \end{align*}      
    To conclude the proof of the liminf inequality, it only remains to show the trace condition in the definition of the Sobolev space $H^1_{\sqrt{\mu}}(\M)$ from \labelcref{eq:def_H1mu} in the case $d\sind{1}=d\sind{2}=:d=d\sind{12}+1$. In this case  the intersection $\M\sind{12}$ locally divides the remaining manifolds into parts $\M\sind{i}_\pm$, one on either side of $\M\sind{12}$, see \cref{fig:trace}.
    In particular, there exists a bi-Lipschitz change of variables, $\Psi$, mapping  the intersection of a ball around a point $x\in\M\sind{12}$ with $\M\sind{1}_+\cup\M\sind{2}_+$ (the red set on the left) to an open subset of $\R^d$ (the red set on the right). Then the liminf inequality from \cref{thm:Gamma-convergence_standard_NL} implies that $u\circ\Psi^{-1}$ is a $H^1$-function and hence has a well-defined trace on $\Psi(\M\sind{12})$. Thus the two traces of $u$ $\M\sind{12}$ coincide on the neighborhood of $x$. Since $x$ was arbitrary the traces agree on all of $\M\sind{12}$.
    \begin{figure}[htb]
    \centering
    \begin{tikzpicture}
     \definecolor{manifoldA}{RGB}{70,130,180} 
    \definecolor{manifoldB}{RGB}{255,165,0}  
    \definecolor{intersectionColor}{RGB}{220,20,60} 
    \definecolor{ballColor}{RGB}{144,238,144} 
    \definecolor{highlightColor}{RGB}{0,100,0} 
     \draw[thick, manifoldA] plot[smooth, tension=1] coordinates {(-2,2) (0,0) (2,-1)};
    \draw[thick, manifoldB] plot[smooth, tension=1] coordinates {(-2,-1) (0,0) (2,2)};
     \filldraw[intersectionColor] (0,0) circle (2pt);
     \draw[very thick, intersectionColor] plot[smooth, tension=1] coordinates {(-0.92,0.8) (-0.5,0.4) (0,0)};
    \draw[very thick, intersectionColor] plot[smooth, tension=1] coordinates {(-1.02,-0.65) (-0.5,-0.35) (0,0)};
     \draw[manifoldA, dashed, thick] plot[smooth, tension=1] coordinates {(-1.1,2) (-0.7,0.9) (0,0.5) (1,1) (2,0) };
    \draw[manifoldB, dashed, thick] plot[smooth, tension=1] coordinates {(-1.2,-0.2) (-0.75,1) (0,0.5) (1,0.2) (2,1)};
    \filldraw[intersectionColor] (0,0.5) circle (1pt);
       \node[anchor=north east] at (-2,2) {\textcolor{manifoldA}{$\M\sind{1}$}};
    \node[anchor=south east] at (-2,-1) {\textcolor{manifoldB}{$\M\sind{2}$}};
    \node[anchor=west] at (0.1,0) {\textcolor{intersectionColor}{$x\in\M\sind{12}$}};
      \draw[-{Latex[round]}, thick] (3,0) to[out=20, in=160] (5,0);
    \node[anchor=south] at (4,0.7) {\textbf{bi-Lipschitz} $\Psi$};
     \draw[very thick, intersectionColor] plot[smooth, tension=1] coordinates {(6,0) (8,0)};
    \filldraw[intersectionColor] (7,0) circle (2pt);
    \node[anchor=south] at (7,0) {$\Psi(x)$};
    \draw[manifoldB, dashed, thick] plot[smooth, tension=1] coordinates {(6,1.3) (6.5,1.7) (7,1)};
    \draw[manifoldA, dashed, thick] plot[smooth, tension=1] coordinates {(7,1) (7.6,1.3) (8,1.7)};
     \filldraw[intersectionColor] (7,1) circle (1pt);
    \node[anchor=north west] at (7.6,1.3) {$u\circ\Psi^{-1}$};
    
    \end{tikzpicture}
    \caption{Validity of the trace condition as proved in \cref{lem:trace}. The ``one-sided'' neighborhoods of $x$ in $\M\sind{1}$ and $\M\sind{2}$ (red) are flattened together by $\Psi$.
    The dotted lines represent the values of the function $u$  (left plot) and  $u\circ\Psi^{-1}$ (right plot).}\label{fig:trace}
    \end{figure}
    
    This idea is made rigorous in \cref{lem:trace} in the appendix. 
    Applying it to $v_n$ and $v$ as defined above and assuming, without loss of generality, that $\liminf_{n\to\infty}\E_{\eps_n}(u_n)<\infty$ we obtain that $\trace\sind{1}(v)=\trace\sind{2}(v)$ on $\M\sind{12}$ and therefore, using also that $\beta_\eta\sind{1}=\beta_\eta\sind{2}$, we have 
    $$
    \trace\sind{1}\left(\frac{u}{\sqrt{\alpha\sind{1}\rho\sind{1}}}\right)=\trace\sind{2}\left(\frac{u}{\sqrt{\alpha\sind{2}\rho\sind{2}}}\right).
    $$ 
    This implies $u\in H^1_{\sqrt{\mu}}(\M)$ and concludes the proof.
\end{proof}
\begin{remark}
We note that the argument does not require that $u_n \to u$ in $L^2(\M)$, but only that 
$$
\int_{\M\sind{i}}  \abs{u_n(x)-u(x)}^2\frac{\d\vol\sind{i}(x)}{\eps_n^{-d\sind{i}}\eta_{\eps_n}(\abs{x-\cdot})\star\mu}\to 0,\quad n\to\infty,
$$ 
for $i=1,2$, which is strictly weaker if $d\sind{1}<d\sind{2}$ since $\eps_n^{-d\sind{i}}\eta_{\eps_n}(\abs{x-\cdot})\star\mu$ is bounded from below by a positive number and, by \labelcref{eq:convinM2scale}, on $\M\sind{2}$ near $\M\sind{1}$ diverges like $\eps_n^{d\sind{1}-d\sind{2}} \gg 1$. 
\end{remark}

We refer to \cref{sec:liminf} for the proof of the liminf inequality for \cref{thm:gamma}.

\subsection{Proof of nonlocal \texorpdfstring{$\Gamma$}{Gamma}-limsup inequality}
\label{sec:NL_limsup}

Finally, we discuss the proof of the limsup inequality for the $\Gamma$-convergence claimed in \cref{thm:Gamma_NL}, namely that for all $u\in L^2(\M)$ there exists a so-called recovery sequence $(u_n)\subset L^2(\M)$ such that $u_n\overset{L^2}{\to}u$ and
\begin{align*}
    \limsup_{n\to\infty}\E_{\eps_n}(u_n) \leq \E(u).
\end{align*}
The proof requires considering a few distinct cases with respect to the codimension of the intersection manifold $\M\sind{12}$.
In order to avoid repetition, we do not present the argument for the critical codimension $d\sind{2}-d\sind{12}=2$ in full generality but stick to a less technical description of the key ideas in this section. 
Note that later we will give the full proof of the discrete limsup inequality in \cref{prop:limsup_discrete}.
       \begin{proof}[Detailed proof sketch of the limsup inequality for \cref{thm:Gamma_NL}]
Here we go over the key steps and main difficulties in going from the standard argument on one manifold (see for example \cite{laux2025large}) to the union of manifolds.

\paragraph{Smooth approximations.} For proving the limsup inequality it is advantageous to reduce the argument  to creating recovery sequences for the energy $\E(u)$ for $u$ sufficiently regular, e.g., Lipschitz continuous. If $d\sind{1}=d\sind{2}=d\sind{12}+1$ one needs to establish the density of such functions among $H^1$ functions on the union of manifolds that respect the trace condition \labelcref{eq:def_H1mu}. This requires a detailed argument, which we present in \cref{lem:density} in the appendix.
 If $d\sind{2} \geq d\sind{12}+2$ the trace condition is of no interest and smoothing can be easily done, since the energy $\E$ decouples the two manifolds, cf. \labelcref{eq:limit_fctl}, and we can use the density of smooth functions in $H^1$ on each manifold separately.  
Hence, one can work with regular $u$ on both of the manifolds and construct individual recovery sequences.

\paragraph{Recovery sequence for codimensions larger than two.}
 When $d\sind{1}-d\sind{12}>2$, i.e., the codimension of the intersection in both of the manifolds is larger than two, one can use constant recovery sequences. By density we can take $u$ such that restriction of $u$ to each of the manifolds is smooth. 
Let $v\sind{i}=u/\sqrt{\alpha\sind{i}\beta_\eta\sind{i}\rho\sind{i}}$ on $\M\sind{i}$. 
By the proof of \cref{thm:Gamma-convergence_standard_NL} (where a constant recovery sequence is taken)
\begin{align}\label{eq:recov_standard}
\limsup_{n\to\infty}\mathcal G_{\eps_n}(v\sind{i};\M\sind{i})\leq\mathcal{G}\left(v\sind{i};\M\sind{i}\right)\quad\text{for }i\in\{1,2\}.
\end{align}
Then for $A_{\eps} :=  \M\sind{1}  \cup \{x \in \M\sind{2}\st  d(x, \M\sind{1})>\eps\}$ and $S_{\eps} := \M \setminus A_{\eps}$ we have
\begin{align*}
   \E_{\eps_n}(u) & = \frac{1}{\eps_n^2}
    \int_{\M}
    \int_{\M}
    \eta_{\eps_n}(\abs{x-y})
    \abs{
    \frac{v(x) \sqrt{\alpha\sind{i}\beta_\eta\sind{i}\rho\sind{i}(x)}}{\sqrt{\eta_{\eps_n}(\abs{x-\cdot})\star\mu}}
    -
    \frac{v(y) \sqrt{\alpha\sind{i}\beta_\eta\sind{i}\rho\sind{i}(y)}}{\sqrt{\eta_{\eps_n}(\abs{y-\cdot})\star\mu}}
    }^2
    \d\mu(x)\d\mu(y)
    \\
    &\leq 
    \frac{1}{\eps_n^2}
    \int_{A_{\eps_n}}
    \int_{A_{\eps_n}}
    \eta_{\eps_n}(\abs{x-y})
    \abs{
    \frac{v(x) \sqrt{\alpha\sind{i}\beta_\eta\sind{i}\rho\sind{i}(x)}}{\sqrt{\eta_{\eps_n}(\abs{x-\cdot})\star\mu}}
    -
    \frac{v(y) \sqrt{\alpha\sind{i}\beta_\eta\sind{i}\rho\sind{i}(y)}}{\sqrt{\eta_{\eps_n}(\abs{y-\cdot})\star\mu}}
    }^2
    \d\mu(x)\d\mu(y) \\
    & \phantom{\leq}  
      + \frac{2}{\eps_n^2}
    \int_{\M}
    \int_{S_{\eps_n}}
    \eta_{\eps_n}(\abs{x-y})
    \abs{
    \frac{v(x) \sqrt{\alpha\sind{i}\beta_\eta\sind{i}\rho\sind{i}(x)}}{\sqrt{\eta_{\eps_n}(\abs{x-\cdot})\star\mu}}
    -
    \frac{v(y) \sqrt{\alpha\sind{i}\beta_\eta\sind{i}\rho\sind{i}(y)}}{\sqrt{\eta_{\eps_n}(\abs{y-\cdot})\star\mu}}
    }^2
    \d\mu(x)\d\mu(y) \\
    & =: \E^{far}_{\eps_n}(u) +\E^{near}_{\eps_n}(u).
\end{align*}
Thanks to \labelcref{eq:mollification,eq:recov_standard} we have that 
\[ \limsup_{n \to \infty} \E^{far}_{\eps_n}(u) \leq \limsup_{n\to\infty} \sum_{i=1}^2 
\left(\alpha\sind{i}\right)^2
\mathcal G_{\eps_n}(v\sind{i};\M\sind{i})\leq  \sum_{i=1}^2  \left(\alpha\sind{i}\right)^2
\mathcal{G}\left(v\sind{i};\M\sind{i}\right) = \E(u).
\]
So it remains to show that $\lim_{n\to \infty} \E^{near}_{\eps_n}(u) = 0$.
From \labelcref{eq:mollification},  it follows that $\eps_n^{-d\sind{i}} \eta_{\eps_n}(\abs{x-\cdot})\star\mu \geq  \frac12 \alpha\sind{i}\beta_\eta\sind{i}\rho\sind{i}(x)$ on $\M\sind{i}$, for $n$ large enough. 
We estimate
 \begin{align*}
  \E^{near}_{\eps_n}&(u)   \lesssim
    \frac{1}{\eps_n^2}
    \int_{S_{2 \eps_n}}
    \int_{S_{2\eps_n}}
    \eta_{\eps_n}(\abs{x-y})
    \eps_n^{-d\sind{2}}
    \left(v(x)^2+ v(y)^2\right)
    \d\mu\sind{2}(x)\d\mu\sind{2}(y)
    \\
  &  \phantom{(u) \leq} 
    +
    \frac{1}{\eps_n^2}
    \int_{\M\sind{2}}
    \int_{\M\sind{1}}
    \eta_{\eps_n}(\abs{x-y})
    \left(\eps_n^{-d\sind{1}} v(x)^2+ \eps_n^{-d\sind{2}} v(y)^2\right)
    \d\mu\sind{1}(x)\d\mu\sind{2}(y)
    \\
    & \lesssim 
    \frac{C_\rho \norm{v}_{\infty}^2}{\eps_n^2}
      \left( \mu\sind{2}(S_{2\eps_n})
      + \eps_n^{d\sind{2} - d\sind{1}} \mu\sind{1}\left(\{x \in \M\sind{1} \! \st  d(x, \M\sind{2})<\eps_n\}\right) 
      + \eps_n^{d\sind{1} - d\sind{2}} \mu\sind{2}(S_{\eps_n}) 
      \right) \\
        & \lesssim  
    \frac{\norm{v}_{\infty}^2}{\eps_n^2}
      \left(\eps_n^{d\sind{2}- d\sind{12}}
      + \eps_n^{d\sind{2} - d\sind{1}}\eps_n^{d\sind{1}- d\sind{12}}
      + \eps_n^{d\sind{1} - d\sind{2}} \eps_n^{d\sind{2}- d\sind{12}}
      \right) \\
    & \lesssim     \norm{v}_{\infty}^2
      \left( \eps_n^{d\sind{2}- d\sind{12} -2} + 
      \eps_n^{d\sind{1}- d\sind{12}-2}
      \right) 
    \to 0
    \qquad
    \text{as }n\to\infty.
\end{align*}

\begin{remark}
We remark that the argument above can be extended to the case when $d\sind{1} - d\sind{12} \leq 2$ and  $d\sind{2} - d\sind{12} >2$.  Namely it is not hard to show that functions $v$ which are compactly supported in $\M^2 \setminus \M\sind{12}$ are dense in $H^1(\M^2)$ whenever  $d\sind{2} - d\sind{12} \geq 2$. Thus, by a density argument, it suffices to consider $v$ which are equal to zero in $S_{2\eps_n}$ for $n$ large enough. Thus in the estimate above we can obtain 
$\E^{near}_{\eps_n}(u) \lesssim \|v\|_{L^\infty}^2 \eps_n^{d\sind{2}- d\sind{12} -2}$ which converges to zero as $n \to \infty$ when $d\sind{2} - d\sind{12} >2$, regardless of $d\sind{1} - d\sind{12} >2$.
\end{remark}

\paragraph{Recovery sequence for intermediate codimensions.}
Here we deal with the case that $d\sind{1} - d\sind{12} \leq 2$ and 
$d\sind{2} - d\sind{12} \geq 2$.
The critical case arises if $d\sind{2} - d\sind{12} = 2$, i.e., the larger of the two codimensions is two, where one can think of a one-dimensional manifold piercing a two-dimensional one.
Even in this simple case with $d\sind{1}=1$, $d\sind{2}=2$, and $d\sind{12}=0$, if one tries to create a recovery sequence a function $u\in H^1_{\sqrt{\mu}}(\M)$ which is constant on each of the manifolds $\M\sind{i}$ with value $u\sind{i}\in\R$ but not constant on $\M$ as a whole, one observes that 
\begin{align*}
    \E_{\eps_n}(u) 
    &= 
    \frac{1}{\eps^2}\int_{\M\sind{1}}\int_{\M\sind{2}}
    \eta_{\eps_n}(\abs{x-y})\abs{\frac{u\sind{1}}{\sqrt{\eta_{\eps_n}(\abs{x-\cdot})\star\mu}} - \frac{u\sind{2}}{\sqrt{\eta_{\eps_n}(\abs{x-\cdot})\star\mu}}}^2\d\mu(x)\d\mu(y)
    \\
    &\sim 
    \frac{1}{\eps^2}
    \eps^{1+2-0} \abs{\frac{u\sind{1}-u\sind{2}}{\sqrt{\eps}}}^2
    \sim \abs{u\sind{1}-u\sind{2}} \neq 0 = \E(u)
\end{align*}
Hence, the constant sequence is not a recovery sequence, unlike in the case where the codimensions are larger than two. 
To fix this issue, we need to interpolate the value of $u$ on $\M\sind{1}$ to the higher dimensional manifold $\M\sind{2}$ in such a way that the energy converges.

The idea is that since the harmonic capacity of $\M\sind{12}$ in $\M\sind{2}$ is zero one can use an interpolation based on the fundamental solution of Laplacian in 2D (the codimension) to interpolate the values. 
Here we sketch this approach in the simplified setting outlined above. 
In particular, we assume $d\sind{1}=1$, $d\sind{2}=2$, $d\sind{12}=0$, and that $\M\sind{12}$ is a single point, w.l.o.g. $\M\sind{12}=\{0\}$.
To make computations simple we also assume that $\M\sind{2}:=\mathbb T^2$ is a two-dimensional flat torus, pierced orthogonally by a one-dimensional manifold $\M\sind{1}$ which is flat in a neighborhood of their intersection.
Let $u\in H^1_{\sqrt{\mu}}(\M)$ be such that $u\sind{i}:=u\vert_{\M\sind{i}}$ are Lipschitz continuous.
We define a recovery sequence $u_n \in L^2(\M)$ via 
\begin{align*}
    u_n(x) := u\sind{1}(x)\sqrt{\frac{\eps_n^{-1}\eta_{\eps_n}(\abs{x-\cdot})\star\mu}{\alpha\sind{1}\beta_\eta\sind{1}\rho\sind{1}(x)}},
    \quad
    x\in\M\sind{1},
\end{align*}
on the one-dimensional manifold $\M\sind{1}$ and
\begin{align*}
    u_n(x) := 
    \begin{dcases}
        u\sind{1}(0)
        \sqrt{\frac{\eps_n^{-1}
        \eta_{\eps_n}(\abs{x-\cdot})\star\mu}{\alpha\sind{1}\beta_\eta\sind{1}\rho\sind{1}(0)}}
        , \quad&\text{if }\abs{x}<\eps_n,\\
        u\sind{2}(x)
        \sqrt{\frac{\eps_n^{-2}
        \eta_{\eps_n}(\abs{x-\cdot})\star\mu}{\alpha\sind{2}\beta_\eta\sind{2}\rho\sind{2}(x)}}
        +
        \frac{\log\sqrt{\eps_n} - \log\abs{x}}{\log\sqrt{\eps_n}-\log\eps_n}
        \times&\\
        \times
        \left(
        u\sind{1}(0)
        \sqrt{\frac{\eps_n^{-1}
        \eta_{\eps_n}(\abs{x-\cdot})\star\mu}{\alpha\sind{1}\beta_\eta\sind{1}\rho\sind{1}(0)}}
        -
        u\sind{2}(x)
        \sqrt{\frac{\eps_n^{-2}
        \eta_{\eps_n}(\abs{x-\cdot})\star\mu}{\alpha\sind{2}\beta_\eta\sind{2}\rho\sind{2}(x)}}
        \right),
        \quad&\text{if }\eps_n\leq\abs{x}<\sqrt{\eps_n},\\
        u\sind{2}(x)
        \sqrt{\frac{\eps_n^{-2}
        \eta_{\eps_n}(\abs{x-\cdot})\star\mu}{\alpha\sind{2}\beta_\eta\sind{2}\rho\sind{2}(x)}}
        ,\quad&\text{if }\sqrt{\eps_n}\leq\abs{x},
    \end{dcases}
\end{align*}
                                for $x\in\M\sind{2}$ on the two-dimensional one.
Here we assume that $n\in\N$ is sufficiently large so that $\eps_n<1$.
Note that the square root terms are all very close to one or even equal to one (at least all but the first one) if the densities $\rho\sind{i}$ are constant on each of the two manifolds. 
We plot this simplified situation in \cref{fig:log-interp} and it should be noted that the recovery sequence equals the value of $u\sind{1}$ in an $\eps_n$-neighborhood of the intersection, logarithmically interpolates between the values of $u\sind{1}$ and $u\sind{2}$, in an annulus of diameter $\sqrt{\eps_n}$, and equals $u\sind{2}$ outside of this annulus.
By construction, the terms $\sqrt{\eta_{\eps_n}(\abs{x-\cdot})\star\mu}$ will drop out once we plug $u_n$ into the nonlocal energy $\E_{\eps_n}$.
To show that $u_n$ is indeed a recovery sequence, we first need to argue that $u_n$ converges to $u$ in $L^2$.
This is a simple consequence of the Lipschitz continuity of $u\sind{i}$ and of the behavior of the convolutions $\eta_{\eps_n}(\abs{x-\cdot})\star\mu$ as analyzed in the compactness proof in \cref{sec:NL_compactness}.
\begin{figure}
    \centering
    \begin{tikzpicture}
    \definecolor{manifoldA}{RGB}{70,130,180} 
    \definecolor{manifoldB}{RGB}{255,165,0}  
    \definecolor{intersectionColor}{RGB}{220,20,60} 
     \draw[->,manifoldB, very thick] (-0.2,0) -- (6,0) node[right] {\color{manifoldB}$\M\sind{2}$};
    \draw[->,very thick] (0,-0.2) -- (0,4.5) node[above] {$u_n$};
     \coordinate (A) at (0,0);
    \coordinate (B) at (1,0);
    \coordinate (C) at (3,0);
     \fill[intersectionColor] (A) circle (2pt);
    \fill[black] (B) circle (1pt);
    \fill[black] (C) circle (1pt);
       \draw[scale=1, domain=0:1, smooth, variable=\x, intersectionColor, dashed, thick] plot
    ({\x}, {2.95*sqrt(1-\x*\x)+0.25});
        \draw[thick,manifoldB] (3,1.6) .. controls (3.7,1.3) and (4,1.2) .. (5.6,1.15);
     \draw[scale=1, domain=1:3, smooth, variable=\x, blue, dashed, thick] plot ({\x}, {2.5 - (2.5-0.25)*(log2(3)-log2(0.5*\x))/(log2(3)-log2(0.5))});
      \node[anchor = north east] at (A) {\color{intersectionColor}$\{0\}$};
    \node[anchor = north] at (B) {\footnotesize$\dist=\eps_n$};
    \node[anchor = north] at (C) {\footnotesize$\dist=\sqrt{\eps_n}$}; 
    \node[anchor = south west] at (5.4,1.2) {\color{manifoldB}$u\sind{2}$}; 
    \node[anchor = east] at (0,3.2) {\color{intersectionColor}$\approx u\sind{1}(0)$};    
\end{tikzpicture}
    \caption{
    The recovery sequence approximates the value $u\sind{1}(0)$ on the intersection $\M\sind{12}=\{0\}$ and the function $u\sind{2}$ far from the intersection for $\eps_n\ll 1$, interpolating logarithmically in between.
     }
    \label{fig:log-interp}
\end{figure}
                                                                                           The interesting part is to prove that it satisfies the inequality $\limsup_{n\to\infty}\E_{\eps_n}(u_n)\leq\E(u)$.
Thanks to the regularity of $u$ (and $\rho$) on each of the two manifolds, the only non-trivial region to study is the interpolation region $\eps_n\leq\abs{x}\leq\sqrt{\eps_n}$.
In the end, it suffices to argue that the interactions within this region and those between this interpolation region and the rest of $\M\sind{2}$ in the nonlocal energy go to zero.
The main idea is that, since the harmonic capacity of a point in $\R^2$ is zero, one can use a logarithmic interpolation in the larger-dimensional manifold to match the values of the functions between the manifolds with little cost in terms of the (nonlocal) Dirichlet energy.
The corresponding term can be controlled (after a change of variables using the Cheeger--Colding segment inequality in \cref{lem:Cheeger--Colding} below, and by employing the Lipschitz regularity of $u$ and $\rho$) by the following integral
\begin{align}
\begin{split}
    \int_{\eps_n\leq\abs{x}\leq\sqrt{\eps_n}} &
    \abs{\nabla_x\frac{\log\sqrt{\eps_n} - \log\abs{x}}{\log\sqrt{\eps_n}-\log\eps_n}}^2\d x
    =
    \int_{\eps_n\leq\abs{x}\leq\sqrt{\eps_n}}
    \abs{\frac{\frac{x}{\abs{x}^2}}{\log\sqrt{\eps_n}-\log\eps_n}}^2\d x
    \\
    &=
    \frac{4}{\abs{\log\eps_n}^2}
    \int_{\eps_n\leq\abs{x}\leq\sqrt{\eps_n}}
    \frac{1}{\abs{x}^2}\d x
    =
    \frac{8\pi}{\abs{\log\eps_n}^2}
    \int_{\eps}^{\sqrt{\eps_n}}
    \frac{1}{t}
    \d t
    =
    \frac{4\pi}{\abs{\log\eps_n}}
\end{split} \label{eq:aux_log}
\end{align}
which goes to zero as $n\to\infty$.
Note that the same could not be achieved with other interpolations like, e.g., the linear one $\frac{\sqrt{\eps_n}-\abs{x}}{\sqrt{\eps_n}-\eps_n}$.
Instead, we used the logarithmic growth of the fundamental solution of the Laplace equation in codimension two for this construction to work.
      
\paragraph{Recovery sequence for codimension one: \texorpdfstring{$d\sind{1} = d\sind{2} =:d$}{d1=d2=:d} and \texorpdfstring{$d-d\sind{12} =1$}{d-d12=1}.}
In this case, the energy $\E$ couples both manifolds through the trace condition encoded in $H^1_{\sqrt{\mu}}(\M)$.
Hence, aiming to apply a density argument we cannot simply regularize a function $u\in H^1_{\sqrt{\mu}}(\M)$ on each of the manifolds separately since this violate the trace condition. 
Instead of approximating $u$ directly, we approximate $v := u/\sqrt{\rho}$ which satisfies $\trace\sind{1}(v)=\trace\sind{2}(v)$.
We  approximate $v$ by regular functions globally on $\M$ such that the trace condition is preserved and the energies of the approximations converge.
By \cref{lem:density} we can find a smooth approximation $v_\eps$ of $v$ such that $v_\eps\vert_{\M\sind{i}}\to v\vert_{\M\sind{i}}$ in $H^1(\M\sind{i})$ for $i\in\{1,2\}$ as $\eps\to 0$ and $\trace\sind{1}(v_\eps)=\trace\sind{2}(v_\eps)$.
We set $u_\eps := v_\eps \sqrt{\rho}$. Thus
$   \frac{u_\eps}{\sqrt{\alpha\sind{i}\rho\sind{i}}} = v_\eps$ on $\M\sind{i}$
and 
\begin{align*}
    \trace\sind{1}\left(
    \frac{u_\eps}{\sqrt{\alpha\sind{1}\rho\sind{1}}}\right)
    =
    \trace\sind{1}(v_\eps)
    =
    \trace\sind{2}(v_\eps)
    =
    \trace\sind{2}\left(
    \frac{u_\eps}{\sqrt{\alpha\sind{2}\rho\sind{2}}}\right),   
\end{align*}
which means that $u_\eps$ satisfies the trace condition.
The Lipschitz continuity and the lower and upper bounds of $\rho\sind{i}$ then imply that $u_\eps\vert_{\M\sind{i}}\in H^1(\M\sind{i})$ for $i\in\{1,2\}$ and hence we have showed that
\begin{align*}
    u_\eps \in H^1_{\sqrt{\mu}}(\M).
\end{align*}
Furthermore  $v_\eps \to v=u/\sqrt{\rho}$ in $L^2(\M)$ implies that $u_\eps \to u$ in $L^2(\M)$.
As before we note that 
\begin{align}\label{eq:energy_in_terms_of_unnormalized}
    \E(u) = 
    \left(\alpha\sind{1}\right)^2
    \mathcal{G}\left(\frac{u}{\sqrt{\alpha\sind{1}\beta_\eta\sind{1}\rho\sind{1}}};\M\sind{1}\right)
    +
    \left(\alpha\sind{2}\right)^2
    \mathcal{G}\left(\frac{u}{\sqrt{\alpha\sind{2}\beta_\eta\sind{2}\rho\sind{2}}};\M\sind{2}\right).
\end{align}
Taking into account \labelcref{eq:energy_in_terms_of_unnormalized}, the homogeneity of $\mathcal{G}(\cdot;\M\sind{i})$ as well as $\mathcal{G}(v_\eps;\M\sind{i}) \to \mathcal{G}(v;\M\sind{i})$ as $\eps\to 0$ by the $H^1$-convergence of $v_\eps$, we deduce that $\E(u_\eps)\to\E(u)$.

By the density above  (combined with a diagonal argument) it suffices to construct a recovery sequence for $u$ such that $v = u/\sqrt{\rho}$ is Lipschitz continuous.  
     To be able to control cross terms near the intersection, instead of the constant recovery sequence, we consider the sequence
\begin{align*}
    u_n(x) := u(x)\sqrt{\frac{\eps_n^{-d}\eta_{\eps_n}(\abs{x-\cdot})\star\mu}{\alpha\sind{i}\beta_\eta\rho\sind{i}(x)}},
    \qquad
    x\in\M\sind{i},
\end{align*}
where we dropped the index $i$ on $\beta_\eta$ since $d\sind{1}=d\sind{2}=d$.
The rationale behind this choice is that the square-root term converges to one uniformly (yielding $u_n\to u$ in $L^2(\M)$, cf. the proof of \cref{thm:compactness_NL}) and the numerator cancels once we plug $u_n$ into the energy $\E_{\eps_n}$.
Using the fact that  $v := \frac{u}{\sqrt{\alpha\sind{i}\beta_\eta\rho\sind{i}}}$ on $\M\sind{i}$, we have
\begin{align*}
    \E_{\eps_n}(u_n)
    &=
    \sum_{i=1}^2
    \left(\alpha\sind{i}\right)^2
    \mathcal G_{\eps_n}
    \left(
    v   
    ;\M\sind{i}\right)
    \\
    &\qquad
    +
    \frac{2}{\eps_n^{2+d}}
    \int_{\M\sind{2}}\int_{\M\sind{1}}
    \eta_{\eps_n}(\abs{x-y})
    \abs{v(x)-v(y)}^2\rho\sind{1}(x)\d\vol\sind{1}(x)\rho\sind{2}(y)\d\vol\sind{2}(y).
\end{align*}
Note that $v$ is globally Lipschitz on $\M$. 
In particular, by \cref{thm:Gamma-convergence_standard_NL} (in the proof of which a constant recovery is used for Lipschitz functions) we have
\begin{align*}
    \limsup_{n\to\infty}\sum_{i=1}^2
    \left(\alpha\sind{i}\right)^2
    \mathcal G_{\eps_n}
    \left(
    v   
    ;\M\sind{i}\right)
    =
    \sum_{i=1}^2
    \left(\alpha\sind{i}\right)^2    
    \mathcal G
    \left(
    v   
    ;\M\sind{i}\right)
    =
    \E(u).
\end{align*}
It remains to show that the cross term in the above energy decomposition of $\E_{\eps_n}(u_n)$ tends to zero for which we shall use Lipschitzness of $v$:
\begin{align*}
    &\phantom{{}={}}
    \frac{2}{\eps_n^{2+d}}
    \int_{\M\sind{2}}\int_{\M\sind{1}}
    \eta_{\eps_n}(\abs{x-y})
    \abs{v(x)-v(y)}^2\rho\sind{1}(x)\d\vol\sind{1}(x)\rho\sind{2}(y)\d\vol\sind{2}(y)
    \\
    &\lesssim
    \frac{1}{\eps_n^d}
    \int_{\M\sind{2}}    
    \int_{\M\sind{1}}
    \eta_{\eps_n}(\abs{x-y})
    \d\vol\sind{1}(x)\d\vol\sind{2}(y)
    \lesssim
    \frac{1}{\eps_n^d}
    \eps_n^{d-d\sind{12}}
    \eps_n^d
    =
    \eps_n \to 0,\qquad n\to\infty.
\end{align*}
\end{proof}
The proof of the limsup inequality for \cref{thm:gamma} can be found in \cref{sec:limsup} and follows the ideas and case distinctions outlined above.
In particular, we refer to \cref{lem:nonlocal_error} therein for the above mentioned consequence of the Cheeger--Colding segment inequality that allows one to upper bound the nonlocal by the local energy.

\section{Proofs of \texorpdfstring{$\Gamma$}{Gamma}-convergence and compactness for graph Dirichlet energies}
\label{sec:proofs_graphs}

In this section we prove our main results \cref{thm:compactness} and \cref{thm:gamma}.
We first establish the following uniform convergence for the degrees, which are nothing but kernel density estimates.
Let us consider the contribution to the degree from vertices in each manifold:
\[    \deg\sind{i}_{n,\eps}(x) := \frac{1}{n} \sum_{y \in V_n \cap \M\sind{i}} \eta_\eps(|x-y|).   \]
Note that $\deg_{n,\eps}(x)=\deg\sind{1}_{n,\eps}(x)+\deg\sind{2}_{n,\eps}(x)$.

\begin{lemma}[Uniform convergence of kernel density estimators]\label{lem:uniform_cvgc}
    Under \cref{ass:densities,ass:epsilon,ass:eta} it holds with 
       probability one  that
    \begin{align*}
        \lim_{n\to\infty}
        \max_{i\in\{1,2\}}
        \sup_{\substack{x\in\M\sind{i}}}
        \abs{\eps_n^{-d\sind{i}}\deg\sind{i}_{n,\eps_n}(x)-\alpha\sind{i}\beta_\eta\sind{i}\rho\sind{i}(x)}
        =
        0.
    \end{align*}
\end{lemma}
\begin{proof}
    The result follows from \cite[Theorem 2.3 (2)]{wu2022strong}, combined with the Borel--Cantelli lemma.
\end{proof}
The next result allows us to  take into account the interactions between manifolds. We use it in the proof of compactness.
 We show  that for $x \in \M$ the degree $\deg_{n,\eps_n}(x)$ with high probability scales like $\vol(B(x,\eps_n)\cap\M)$, i.e., the volume of the intersection of a Euclidean ball with radius $\eps_n$ with $\M$.
\begin{lemma}\label{lem:degrees}
Under \cref{ass:densities,ass:eta} there exist constants $C_1,C_2,C_3,C_4>0$ 
and $\tau\in(0,1]$ \nc only depending on $\rho$, $\eta$, and $\M$ such that for $n\in\N$ sufficiently large it holds
\begin{align*}
    \begin{split}
    &\phantom{{}={}}
    \mathbb P\Big((\forall x \in V_n) \,
     C_1 \vol(B(x,\tau\eps_n)\cap \M)
    \leq
    \nc 
    \deg_{n,\eps_n}(x)\leq C_2\vol(B(x,\eps_n)\cap\M)\;\Big)
    \\
    &\qquad\geq 
    1- C_3\exp\left(-C_4 n \eps_n^{d\sind{2}}+\log n\right).
    \end{split} 
\end{align*}
            \end{lemma}
\begin{proof} 
    \red We recall a convenient version of Bernstein's inequality (see, e.g., \cite[Section~2.8]{boucheronLugosiMassart2013}): if $Z_1,\dots,Z_n$ are i.i.d.\ random variables with $0\le Z_i\le b$ and $\mathbb E[Z_i^2]\le\nu$, then for every $t\ge 0$
    \begin{align*}
        \mathbb P\left(\abs{\tfrac1n\textstyle\sum_{i=1}^n Z_i-\mathbb E[Z_1]}>t\right)\le 2\exp\left(-\frac{n\,t^2}{2\left(\nu+bt/3\right)}\right).
    \end{align*}
    \nc
    For fixed $x\in\M$, applying this inequality to the random variable $\deg_{n,\eps_n}(x)$ \red (that is, with $Z_i=\eta_{\eps_n}(\abs{x-y_i})$, so that $b=\eta(0)$ and $\nu\le\eta(0)\,\mathbb E[Z_1]\lesssim\vol(B(x,\eps_n)\cap\M)$)\nc\ we obtain that for all $\lambda\in[0,1]$ it holds
\begin{align}\label{eq:bernstein}
    \begin{split}
 \!\!\!\! \!\!\!\!  \mathbb P&\left(\abs{\deg_{n,\eps_n}(x) - \mathbb E\left[\deg_{n,\eps_n}(x)\right]} > C_\rho   \vol(B(x,\eps_n)\cap\M) \lambda \right)
    \\
    & \qquad \qquad \leq 2\exp\left(-\frac{1}{4} C_\rho n \vol(B(x,\eps_n)\cap\M)\lambda^2\right)
    \leq 2\exp\left(-Cn\eps_n^{d\sind{2}}\lambda^2\right),
    \end{split}
\end{align}
where $C>0$ depends on $C_\rho$ and $\M$. We also used that for $n \in \N$ sufficiently large the volume $\vol(B(x, \eps_n))$ is smallest for points $x \in \M\sind{2}$, in the larger dimensional manifold.

Furthermore, we can use that the vertices $y\in V_n$ are i.i.d. samples from $\mu$ to compute
\begin{align*}
    \mathbb E\left[\deg_{n,\eps_n}(x)\right]
    =
    \mathbb E\left[\frac1n\sum_{y\in V_n}\eta_{\eps_n}(\abs{x-y})\right]
    =
    \int_\M \eta_{\eps_n}(\abs{x-y})\de\mu(y).
\end{align*}

Hence, thanks to \cref{ass:eta} (arguing as in \cref{sec:NL_compactness}
there exist constants $C_1, C_2, \tau > 0$ (depending only on $\eta$; \red here $\tau$ is any value for which $\eta\ge C_1>0$ on $[0,\tau]$, which exists since $\eta(0)>0$ and $\eta$ is continuous at $0$ by \cref{ass:eta}\nc) such that
\begin{align*}
     C_1 \vol(B(x,\tau\eps_n)\cap \M)
    \leq 
    \nc
    \mathbb E\left[\deg_{n,\eps_n}(x)\right]
    \leq 
    C_2 \vol(B(x,\eps_n)\cap \M).
\end{align*}
         Using this estimate, we can find a suitable $\lambda\in(0,1]$ in \labelcref{eq:bernstein} and constants $C_3,C_4>0$ such that
\begin{align*}
    \mathbb P\Big( 
     C_1 \vol(B(x,\tau\eps_n)\cap \M)
    \leq
    \nc
    \deg_{n,\eps_n}(x)\leq C_2\vol(B(x,\eps_n)\cap\M)\Big)
    \geq 1- C_3\exp\left(-C_4 n \eps_n^{d\sind{2}}\right).
\end{align*}
       By conditioning on $x\in V_n$, using a union bound and redefining $C_1,C_2,C_3,C_4$ we obtain the desired statement.
\end{proof}

\begin{remark}
    For the probability in \cref{lem:degrees} to be close to one we need \cref{ass:epsilon} to be satisfied which guarantees that $n\eps_n^{d\sind{2}}\gg \log n$.
\end{remark}

\begin{remark}\label{rem:scaling_volumes}
    Under assumptions of \cref{lem:degrees}, we note that for $x\in\M\sind{i}$
    \begin{alignat*}{2}
        \vol(B(x,\eps_n)\cap\M) & \leq  C\eps_n^{d\sind{i}}, \quad \; && \text{if } \dist(x,\M\sind{3-i})>\eps_n,\\
        \vol(B(x,\eps_n)\cap\M) & \geq C \eps_n^{d\sind{1}}
         &&\text{if } \dist(x,\M\sind{3-i})<\eps_n/2,
    \end{alignat*}
    for some $C>0$, depending on $\M$.
               Thus it can be expected that the lower dimensional manifold dominates the scaling of the degrees $\deg_{n,\eps_n}(x)$ of points $x$ close to the intersection of the two manifolds.
\end{remark}

From above lemmas and remarks we have the following corollary.
\begin{corollary}[Uniform convergence of kernel density estimators II]\label{cor:uniform_cvgc2}
    Under \cref{ass:densities,ass:epsilon,ass:eta}, if $d\sind{1}< d\sind{2}$ it holds with 
       probability one that
    \begin{align*}
        \lim_{n\to\infty}
        \sup_{\substack{x\in\M\sind{1}}}
        \abs{\eps_n^{-d\sind{1}}\deg_{n,\eps_n}(x)-\alpha\sind{1}\beta_\eta\sind{1}\rho\sind{1}(x)}
        =
        0.
    \end{align*}
\end{corollary}
This follows from \cref{lem:uniform_cvgc} since $\deg_{n,\eps_n} = \deg\sind{1}_{n,\eps_n} + \deg\sind{2}_{n,\eps_n}$ and, by \cref{lem:degrees,rem:scaling_volumes}, $\deg\sind{2}_{n,\eps_n}(x) \leq C\eps_n^{d\sind{2}} \ll \eps_n^{d\sind{1}}$.
  
\begin{remark} \label{rem:degrees_lb}
    Under assumptions \cref{ass:densities,ass:epsilon,ass:eta}, \red with the transport maps $T_n$ from \cref{rem:scaling_condition}, \nc we almost surely have that
    \begin{align*}
        \limsup_{n\to\infty}\max_{i\in\{1,2\}}\sup_{x\in\M\sind{i}}\frac{1}{\eps_n^{-d\sind{i}}\deg_{n,\eps_n}(T_n(x))} < \infty.
    \end{align*}
This follows from \cref{lem:degrees} and \cref{rem:scaling_volumes}.
In particular, if $x \in \M\sind{1} $ and 
$T_n(x) \in \M\sind{2}$ then since $|T_n(x) - x | < \eps_n/2$ we have that 
$d(T_n(x), \M\sind{1})< \eps_n/2$ so \cref{rem:scaling_volumes} applies. 
 \end{remark}

Using the previous results we can prove the following lemma, which is needed in proving $\Gamma$-convergence and compactness.
\begin{lemma}\label{lem:weights_converge}
    Under assumptions \cref{ass:densities,ass:epsilon,ass:eta} 
       consider $v\in L^2(\M)$ and  let $T_n:\M\to\M$ be a map which almost surely satisfies $(T_n)_\sharp\mu=\mu_n$ and \labelcref{eq:special_maps}.
    Then it holds with probability one that
    \begin{align*}
        \lim_{n \to\infty}
        \max_{i\in\{1,2\}}
        \int_{\M\sind{i}}
        \abs{\sqrt{\frac{\alpha\sind{i}\beta_\eta\sind{i}\rho\sind{i}(x)}{\eps_n^{-d\sind{i}}\deg_{n,\eps_n}(T_n(x))}}-1}^2\abs{v(x)}^2\d\vol\sind{i}(x)
        =0.
    \end{align*}
\end{lemma}
\begin{proof}
    Fix $\delta>0$ and let $x\in\M\sind{i}$ with $\dist(x,\M\sind{12})>\delta$.
    We have
    \begin{align*}
    \frac{\alpha\sind{i}\beta_\eta\sind{i}\rho\sind{i}(x)}{\eps_n^{-d\sind{i}}\deg_{n,\eps_n}(T_n(x))}
    =
    \frac{\alpha\sind{i}\beta_\eta\sind{i}\rho\sind{i}(T_n(x))}{\eps_n^{-d\sind{i}}\deg_{n,\eps_n}(T_n(x))}
    +
    \frac{\alpha\sind{i}\beta_\eta\sind{i}\left(\rho\sind{i}(x)-\rho\sind{i}(T_n(x))\right)}{\eps_n^{-d\sind{i}}\deg_{n,\eps_n}(T_n(x))}.
    \end{align*}
     We note that the presence of the other manifold only increases the observed degree. Therefore,
    combining \cref{lem:degrees,ass:epsilon} 
    with the Borel--Cantelli lemma and \cref{rem:scaling_volumes} we almost surely have 
    \begin{align*}
        \liminf_{n\to\infty}  \inf_{x \in \M}  \eps_n^{-d\sind{i}}\deg_{n,\eps_n}(T_n(x))>0,
    \end{align*}
    as established in \cref{rem:degrees_lb}, where $i$ corresponds to the manifold that $x$ belongs to. 

    Using also \cref{ass:densities,eq:special_maps,lem:uniform_cvgc,cor:uniform_cvgc2} we get for any $\delta>0$ and for all $x \in \M$ with $d(x,\M\sind{12})>\delta$
    \begin{align*}
        \lim_{n\to\infty}\abs{\sqrt{\frac{\alpha\sind{i}\beta_\eta\sind{i}\rho\sind{i}(x)}{\eps_n^{-d\sind{i}}\deg_{n,\eps_n}(T_n(x))}}-1}^2
        \leq 
        C
        \lim_{n\to\infty}
        \norm{T_n - \operatorname{id}}_{L^\infty(\M)}
        =0,
    \end{align*}
    for some $C>0$.
    Hence, the dominated convergence theorem implies that with probability one
    \begin{align*}
        \lim_{n \to\infty}
        \int_{\{x\in\M\sind{i}\st\dist(x,\M\sind{12})>\delta\}}
        \abs{\sqrt{\frac{\alpha\sind{i}\beta_\eta\sind{i}\rho\sind{i}(x)}{\eps_n^{-d\sind{i}}\deg_{n,\eps_n}(T_n(x))}}-1}^2\abs{v(x)}^2\d\vol\sind{i}(x)
        =0.
    \end{align*}
    On the other hand, using the lower bound on the degrees from above, we almost surely have that
    \begin{align*}
        \limsup_{n\to\infty}
        \int_{\{x\in\M\sind{i}\st\dist(x,\M\sind{12})<\delta\}} &
        \abs{\sqrt{\frac{\alpha\sind{i}\beta_\eta\sind{i}\rho\sind{i}(x)}{\eps_n^{-d\sind{i}}\deg_{n,\eps_n}(T_n(x))}}-1}^2\abs{v(x)}^2\d\vol\sind{i}(x) \\
       & \quad \leq C \int_{\{x\in\M\sind{i}\st\dist(x,\M\sind{12})<\delta\}}
       \abs{v(x)}^2\d\vol\sind{i}(x) .
    \end{align*}
    Combining the two estimates we almost surely have
    \begin{align*}
        \lim_{n \to\infty}
        \int_{\M\sind{i}}
        \abs{\sqrt{\frac{\alpha\sind{i}\beta_\eta\sind{i}\rho\sind{i}(x)}{\eps_n^{-d\sind{i}}\deg_{n,\eps_n}(T_n(x))}}-1}^2\abs{v(x)}^2\d\vol\sind{i}(x)
        =0.
    \end{align*}
\end{proof}

\subsection{Compactness}
\label{sec:compactness}

In this section we provide a proof of \cref{thm:compactness}. We start by a simple lemma that has broader applications.
\begin{lemma} \label{lem:TLp_vanishing}
Let $\M$ and $\mu$ be as above.\footnote{This can be extended to a general setting of metric measure spaces.}
Let  $\mu_n$ be a sequence of probability measures on $\M$ converging in Wasserstein distance to $\mu$. Let $v_n \in L^2(\mu_n)$ and $v \in L^2(\mu)$ be such that $v_n$ converge in $TL^2$ distance to $v$. Let $A_n$ be a sequence of sets of vanishing mass, 
$\mu_n(A_n) \to 0$ as $n \to \infty.$
 Then $v_n \chi_{A_n}$ converges to zero in the $TL^2$ sense.
\end{lemma}
\begin{proof}
Let $T_{n, \sharp} \mu = \mu_n$ be a stagnating sequence of transport maps. Let $H_n=T_n^{-1}(A_n)$. Note that $\mu(H_n) =\mu_n(A_n) \to 0$ as $n\to \infty$. Thus $\|v \chi_{H_n}\|_{L^2(\mu)}\to 0$ as $n \to \infty$. By properties of $TL^2$ convergence, \cite[Lemma 3.11]{garcia2016continuum}, we have 
 $\int |v(x) - v_n(T_n(x)|^2\de\mu(x) \to 0$ as $n \to \infty$ and thus
 \[ \int_{H_n} |v_n(T_n(x)|^2\de\mu(x) \leq 2 \int_{H_n} |v(x) - v_n(T_n(x)|^2  + |v(x)|^2\de\mu(x)  \to 0 \qquad \te{as } n\to \infty.\]
Hence  $ \int_{A_n} |v_n(y)|^2\de\mu_n(y) \to 0$ as $n \to \infty$.
\end{proof}

\begin{proof}[Proof of \cref{thm:compactness}]
Since by \cref{ass:eta} $\eta \geq c_1 \chi_{[0,c_2]}$ for some $c_1,c_2>0$, it suffices to show the statement for $\eta = \chi_{[0,1]} $.
Let us define the functions $v_n(x) := \frac{u_n(x)}{\sqrt{\eps_n^{-d\sind{i}}\deg_{n,\eps_n}(x)}}$ for $x \in \M\sind{i}\cap V_n$. Using  that the $L^2(\mu_n)$-norms of $u_n$ are uniformly bounded and \cref{rem:degrees_lb}, we have that almost surely the $L^2(\mu_n)$-norms of $v_n$ are uniformly bounded as $n\to\infty$.
Using $\M=\M\sind{1}\cup\M\sind{2}$ in \labelcref{eq:graph_fctl_integral} and omitting cross-terms we have
\begin{align*}
    E_{n,\eps_n}(u_n) 
    &\geq
    \frac{1}{\eps_n^2}
    \sum_{i=1}^2
    \left(\frac{n\sind{i}}{n}\right)^2 \!
    \int_{\M\sind{i}} \! \int_{\M\sind{i}}\eta_{\eps_n} (|x-y|) \abs{\frac{u_n(x)}{\sqrt{\deg_{n,\eps_n}(x)}}-\frac{u_n(y)}{\sqrt{\deg_{n,\eps_n}(y)}}}^2\d\mu_n\sind{i}(x)\de\mu_n\sind{i}(y).
    \\
    &=
    \sum_{i=1}^2
    \left(\frac{n\sind{i}}{n}\right)^2
    \frac{1}{\eps_n^{d\sind{i}+2}}
    \int_{\M\sind{i}}\int_{\M\sind{i}}\eta_{\eps_n} (|x-y|) \abs{v_n\sind{i}(x) - v_n\sind{i}(y)}^2\d\mu_n\sind{i}(x)\de\mu_n\sind{i}(y).
\end{align*}
\cref{cor:Gamma-convergence_standard} together with the fact that almost surely we have $\lim_{n\to\infty}\frac{n\sind{i}}{n}=\alpha\sind{i}>0$ implies that the sequence $v_n$ has a subsequence (which we do not relabel) which converges almost surely as $n\to\infty$ in the $TL^2(\M)$-sense to some $v\in L^2(\M)$.

We now argue that this implies the $TL^2$ convergence of $u_n$ towards
\begin{align*}
    u(x):=\sqrt{\alpha\sind{i}\beta_\eta\sind{i}\rho\sind{i}(x)}\,v(x),
    \qquad
    x\in\M\sind{i}.
\end{align*}
Let us choose transport maps $T_n$ which satisfy \labelcref{eq:special_maps}.
Using the fact that $v_n\to v$ in $TL^2(\M)$ as well as \cref{lem:weights_converge} we obtain almost surely
\begin{align} \label{eq:local-TL2estimate}
  \begin{split}
     \lim_{n\to\infty}\sum_{i=1}^2\int_{\M\sind{i}} & \abs{u_n(T_n(x))-u(x)}^2\frac{\d\vol\sind{i}(x)}{\eps_n^{-d\sind{i}}\deg_{n,\eps_n}(T_n(x))}
    \\
    &=    \lim_{n\to\infty}\sum_{i=1}^2\int_{\M\sind{i}}\abs{v_n(T_n(x))-\sqrt{\frac{\alpha\sind{i}\beta_\eta\sind{i}\rho\sind{i}(x)}{\eps_n^{-d\sind{i}}\deg_{n,\eps_n}(T_n(x))}}v(x)}^2 \d\vol\sind{i}(x)
    \\
    &\leq     2\lim_{n\to\infty}\sum_{i=1}^2\int_{\M\sind{i}}\abs{v_n(T_n(x))-v(x)}^2 \d\vol\sind{i}(x)
    \\
    &\qquad
    +
    2    \lim_{n\to\infty}\sum_{i=1}^2\int_{\M\sind{i}}\abs{
    1-\sqrt{\frac{\alpha\sind{i}\beta_\eta\sind{i}\rho\sind{i}(x)}{\eps_n^{-d\sind{i}}\deg_{n,\eps_n}(T_n(x))}}
    }^2 \abs{v(x)}^2\d\vol\sind{i}(x)
    =0
   \end{split}
\end{align}
Let, as before, $A_\eps = \M\sind{1} \cup \{ x \in \M\sind{2} \st  d(x,\M\sind{1}) > \eps\}$.
    By \cref{lem:uniform_cvgc,cor:uniform_cvgc2}, with probability one, 
\[ 
    \lim_{n\to\infty}
        \sup_{\substack{x\in A_{\eps_n}\cap \M\sind{i}}}
        \abs{\eps_n^{-d\sind{i}}\deg_{n,\eps_n}(x)-\alpha\sind{i}\beta_\eta\sind{i}\rho\sind{i}(x)}=0, 
\]
and hence $\eps_n^{-d\sind{i}}\deg_{n,\eps_n}(x)$ is bounded from below uniformly in $x$ on $A_{\eps_n}$. Thus \labelcref{eq:local-TL2estimate} implies that 
\[ \lim_{n \to \infty}  \int_{T_n^{-1}(A_{\eps_n})} |u_n(T_n(x)) - u(x)|^2\de\mu(x) \to 0 \qquad \te{as } n \to \infty.\]
Consequently, it remains to prove that 
$\| u_n  \|_{L^2(\mu_n \restr A_{\eps_n}^c)} \to 0$ as $n\to 0$
where $A_{\eps_n}^c := \M \setminus A_{\eps_n} = \{ x \in \M\sind{2} \st  d(x, \M\sind{1}) < \eps_n\}$. We argue as in the 
proof of \cref{thm:compactness_NL}.
Let $S^k_{\eps}$ and $G^k_\eps$ be as in that proof. 

We claim that there exists $C_k>0$ such that for all $n$ large enough for all $x \in S^k_{\eps_n}$
\begin{align} \label{eq:deg_partial_lb}
  \eta_{\eps_n}(\abs{x-\cdot})\star (\mu_n\restr{G^k_{\eps_n}}) > C_k \eta_{\eps_n}(\abs{x-\cdot})\star \mu_n  = C_k\deg_{n, \eps_n}(x).  
\end{align}
For $k=1$ this follows from \cref{lem:degrees}, since for $x \in S^1_{\eps_n}$,
$\,B(x, \eps_n) \cap \M\sind{1} =B(x, \eps_n) \cap G^1_{\eps_n} $. 
For $k>1$ note that 
\[\eta_{\eps_n}(\abs{x-\cdot})\star (\mu_n\restr{G^k_{\eps_n}\cap \M\sind{1}}) = \eta_{\eps_n}(\abs{x-\cdot})\star (\mu_n \restr{ \M\sind{1})}\]
while the fact that 
\[\eta_{\eps_n}(\abs{x-\cdot})\star (\mu_n\restr{G^k_{\eps_n}\cap \M\sind{2}}) \geq C_k\eta_{\eps_n}(\abs{x-\cdot})\star (\mu_n \restr{ \M\sind{2})}\]
for some $C_k>0$ follows from a concentration estimate analogous to ones in \cref{lem:degrees}, since $\vol\restr\M\sind{2} (B(x,\eps_n) \cap G^k_{\eps_n}) > C \vol\restr\M\sind{2} (B(x,\eps_n)) $ for some $C>0$. Combining the two facts implies \eqref{eq:deg_partial_lb}.
Therefore,
  \begin{align*}
    C_k & \|u_n\|^2_{L^2(\mu_n\restr{S^k_{\eps_n}})}  - \|u_n\|^2_{L^2(\mu_n\restr{G^k_{\eps_n}})} \\
    & \leq \int_{S^k_{\eps_n}} \frac{\eta_{\eps_n}(\abs{x-\cdot})\star (\mu_n\restr{G^k_{\eps_n}})}{\eta_{\eps_n}(\abs{x-\cdot})\star \mu_n } \,u_n^2(x)\de\mu_n(x)
    - \int_{G^k_{\eps_n}} \frac{\eta_{\eps_n}(\abs{y-\cdot})\star (\mu_n\restr{S^k_{\eps_n}})}{\eta_{\eps_n}(\abs{y-\cdot})\star \mu_n } \, u_n^2(y)\de\mu(y)  \\
    & = \int_{S^k_{\eps_n}} \int_{G^k_{\eps_n}}   \eta_{\eps_n}(\abs{x-y}) \left( \frac{u_n^2(x)}{\eta_{\eps_n}(\abs{x-\cdot})\star \mu_n } - \frac{u_n^2(y)}{\eta_{\eps_n}(\abs{y-\cdot})\star \mu_n }\right)\de\mu_n(x)\de\mu_n(y) \\
    & \leq \sqrt{ \int_{S^k_{\eps_n}} \int_{G^k_{\eps_n}}   \eta_{\eps_n}(\abs{x-y}) \left( \frac{u_n(x)}{\sqrt{\eta_{\eps_n}(\abs{x-\cdot})\star \mu_n }} - \frac{u_n(y)}{\sqrt{\eta_{\eps_n}(\abs{y-\cdot})\star \mu_n} }\right)^2\de\mu_n(x)\de\mu_n(y)} \\
    & \phantom{\leq\;} \times \sqrt{ \int_{S^k_{\eps_n}} \int_{G^k_{\eps_n}}  \eta_{\eps_n}(\abs{x-y})  \left( \frac{u_n(x)}{\sqrt{\eta_{\eps_n}(\abs{x-\cdot})\star \mu_n }} + \frac{u_n(y)}{\sqrt{\eta_{\eps_n}(\abs{y-\cdot})\star \mu }}\right)^2\de\mu_n(x)\de\mu_n(y)} \\
    & \leq 2 \eps_n \sqrt{ E_{n,\eps_n}(u_n)} \|u_n\|_{L^2(\mu_n)},
    \end{align*}
where we also used that $\eta_{\eps_n}(\abs{x-\cdot})\star\mu_n = \deg_{n,\eps_n}(x)$ for $x\in\M$.
Hence
 \[ \|u_n\|^2_{L^2(\mu_n\restr{S^k_{\eps_n}})} \leq \frac{1}{C_k} \left( \|u_n\|^2_{L^2(\mu_n\restr{G^k_{\eps_n}})}  + 2 \eps_n \sqrt{ E_{n,\eps_n}(u_n)} \|u_n\|_{L^2(\mu_n)} \right) \]
We claim that 
 the right-hand side converges to zero as $n \to \infty$.
 Note that for $k=1$,  $\|u_n\|^2_{L^2(\mu_n\restr{G^1_{\eps_n}})} $ converges to zero by \cref{lem:TLp_vanishing} since $u_n \chi_{A_{\eps_n}}$  converges in $TL^2$, $G^1_{\eps_n} \subset A_{\eps_n} $, and
 $\mu_n(G^1_{\eps_n}) \to 0$ as $n \to \infty$, so the right-hand side converges to zero. For $k>1$ the conclusion follows by induction.
\end{proof}

\subsection{\texorpdfstring{$\Gamma$}{Gamma}-liminf inequality}
\label{sec:liminf}

In this section we prove the $\liminf$ inequality of \cref{thm:gamma}, using \cref{lem:trace}.

\begin{proposition}[Liminf inequality]
Under \cref{ass:densities,ass:epsilon,ass:eta} if $u_n\overset{TL^2}{\longrightarrow} u$, it almost surely holds that
\begin{align*}
    \E(u) \leq \liminf_{n\to\infty} E_{n,\eps_n}(u_n).
\end{align*}
\end{proposition}
\begin{proof}
Without loss of generality we can assume that
\begin{align}\label{ineq:finite_energy}
    \liminf_{n\to\infty}E_{n,\eps}(u_n)<\infty
\end{align}
since otherwise there is nothing to prove.
Dropping cross-terms we get
\begin{align*}
    E_{n,\eps_n}(u_n) 
    \geq
    \sum_{i=1}^2
    \left(\frac{n\sind{i}}{n}\right)^2
    \frac{1}{\eps_n^{2+d\sind{i}}}
    \int_{\M\sind{i}}\int_{\M\sind{i}}\eta_{\eps_n} (|x-y|) \abs{v_n(x)-v_n(y)}^2\d\mu_n\sind{i}(x)\de\mu_n\sind{i}(y),
\end{align*}
where we define
\begin{align*}
    v_n(x) := 
    \frac{u_n(x)}{\sqrt{\eps_n^{-d\sind{i}}\deg_{n,\eps_n}(x)}},\quad x\in\M\sind{i}\cap V_n.
\end{align*}
Let us first argue that $v_n\to v$ in the $TL^2(\M)$-sense where $v := \frac{u}{\sqrt{\alpha\sind{i}\beta_\eta\sind{i}\rho\sind{i}}}$ on $\M\sind{i}$.
For this we let $T_n:\M\to\M$ be transport maps satisfying $(T_n)_\sharp\mu=\mu_n$ as well as \labelcref{eq:special_maps}.
Using $u_n\to u$ in the $TL^2$-sense, together with \cref{ass:densities,lem:weights_converge}, we obtain
\begin{align*}
     \lim_{n\to\infty} &
    \sum_{i=1}^2
    \int_{\M\sind{i}}
    \abs{v_n(T_n(x))-v(x)}^2\d\mu\sind{i}(x)
    \\
    &=
    \lim_{n\to\infty}
    \sum_{i=1}^2
    \int_{\M\sind{i}}
    \abs{u_n(T_n(x))-\sqrt{\frac{\eps_n^{-d\sind{i}}\deg_{n,\eps_n}(T_n(x))}{\alpha\sind{i}\beta_\eta\sind{i}\rho\sind{i}(x)}}u(x)}^2\frac{\d\mu\sind{i}(x)}{\eps_n^{-d\sind{i}}\deg_{n,\eps_n}(T_n(x))}
    \\
    &\leq 
    2
    \lim_{n\to\infty}
    \sum_{i=1}^2
    \int_{\M\sind{i}}
    \abs{u_n(T_n(x))-u(x)}^2\frac{\d\vol\sind{i}(x)}{\eps_n^{-d\sind{i}}\deg_{n,\eps_n}(T_n(x))}
    \\
    &\qquad
    +
    2
    \lim_{n\to\infty}
    \sum_{i=1}^2
    \int_{\M\sind{i}}
    \abs{\sqrt{\frac{\alpha\sind{i}\beta_\eta\sind{i}\rho\sind{i}(x)}{\eps_n^{-d\sind{i}}\deg_{n,\eps_n}(T_n(x))}}-
    1
    }^2
    \frac{\abs{u(x)}^2}{\alpha\sind{i}\beta_\eta\sind{i}\rho\sind{i}(x)}
    \d\vol\sind{i}(x) = 0
\end{align*}
and hence the $TL^2(\M)$-convergence of $v_n$ to $v$ holds.

Applying the $\liminf$ inequality from \cref{cor:Gamma-convergence_standard} to the manifolds $\M\sind{i}$ we obtain that
\begin{align}\label{ineq:liminf_energies}
        \liminf_{n\to\infty}
        E_{n,\eps_n}(u_n) 
        \geq
        \sum_{i=1}^2
        \frac{\alpha\sind{i}\sigma_\eta\sind{i}}{\beta_\eta\sind{i}}
        \int_{\M\sind{i}} \abs{\grad_{\M\sind{i}}\left(\frac{u}{\sqrt{\rho\sind{i}}}\right)\rho\sind{i}}^2\d\vol\sind{i}
\end{align}
holds almost surely.
In particular, this inequality shows that $\frac{u}{\sqrt{\rho\sind{i}}} \in H^1(\M\sind{i})$ for $i=1,2$.

\paragraph{Case 1, \texorpdfstring{$d\sind1\neq d\sind{2}$}{d1!=d2} or \texorpdfstring{$d\sind1-d\sind{12} > 1$}{d1-d12>1}}

In this case we are done since $u\in H^1_{\sqrt{\mu}}(\M)$ (defined in \cref{eq:def_H1mu}) and hence \labelcref{ineq:liminf_energies} implies $\liminf_{n\to\infty} E_{n,\eps_n}(u_n) \geq \E(u)$.

\paragraph{Case 2, \texorpdfstring{$d\sind1= d\sind{2}=:d$}{d1=d2=:d} and \texorpdfstring{$d\sind{12} = d-1$}{d12=d-1}.}

To show that $u\in H^1_{\sqrt{\mu}}(\M)$ we have to show the trace condition of 
\cref{eq:def_H1mu}.
For this, we take a sequence of transport maps $T_n : \M\to\M$ satisfying $(T_n)_\sharp\mu=\mu_n$ and \labelcref{eq:special_maps}.
Using their push-forward property it holds
\begin{align*}
    E_{n,\eps_n}(u_n) 
    =
    \frac{1}{\eps_n^{d+2}}
    \int_\M 
    \int_\M 
    \eta_{\eps_n}(\abs{T_n(x)-T_n(y)})
    \abs{
    v_n(T_n(x))
    -
    v_n(T_n(y))
    }^2
    \d\mu(x)
    \d\mu(y).
\end{align*}
By \cref{ass:eta} there exist $t_0\in(0,1]$ and $a_0>0$ such that $\eta(t)\geq a_0$ for all $0\leq t \leq t_0$.
Hence, as in \cite[Section 5]{garcia2016continuum} we can assume without loss of generality that $\eta(t) = 1_{[0,1]}(t)$. 
Using the estimate
\begin{align*}
    \abs{T_n(x)-T_n(y)} \geq \abs{x-y} - 2\norm{T_n - \operatorname{id}}_{L^\infty(\M)}
\end{align*}
and defining $\tilde\eps_n := \eps_n - 2\norm{T_n-\operatorname{id}}_{L^\infty(\M)}$ we see that
\begin{align*}
    E_{n,\eps_n}(u_n) 
    \geq
    \frac{\tilde\eps_n^{d+2}}{\eps_n^{d+2}}
    \frac{1}{\tilde\eps_n^{d+2}}
    \int_\M 
    \int_\M 
    \eta_{\tilde\eps_n}(\abs{x-y})
    \abs{
    v_n(T_n(x))-v_n(T_n(y))
    }^2
    \d\mu(x)
    \d\mu(y).
\end{align*}
By \cref{ass:epsilon,rem:scaling_condition} it holds $\lim_{n\to\infty}\frac{\tilde\eps_n}{\eps_n}=1$ and $\lim_{n\to\infty}\tilde\eps_n=0$.
Using also \labelcref{ineq:finite_energy} we have
\begin{align}\label{eq:finite_NL_energy}
    \liminf_{n\to\infty}
    \frac{1}{\tilde\eps_n^{d+2}}
    \int_\M 
    \int_\M 
    \eta_{\tilde\eps_n}(\abs{x-y})
    \abs{
    v_n(T_n(x))-v_n(T_n(y))
    }^2
    \d\mu(x)
    \d\mu(y)
    < \infty.
\end{align}
Above, we have already proved that $v_n\circ T_n\to v$ in $L^2(\M)$.
Using this together with \labelcref{eq:finite_NL_energy} and applying \cref{lem:trace} we obtain that $\trace\sind{1}\left({u}/{\sqrt{\alpha\sind{1}\rho\sind{1}}}\right)=\trace\sind{2}\left({u}/{\sqrt{\alpha\sind{2}\rho\sind{2}}}\right)$ on $\M\sind{12}$ which implies $u\in H^1_{\sqrt{\mu}}(\M)$.
Combining this with \labelcref{ineq:liminf_energies} we have that $\liminf_{n\to\infty}E_{n,\eps_n}(u_n)\geq\E(u)$, which concludes the proof.
\end{proof}

\subsection{\texorpdfstring{$\Gamma$}{Gamma}-limsup inequality}
\label{sec:limsup}

To conclude the proof of \cref{thm:gamma} we need to prove the $\limsup$ inequality, which requires distinguishing three cases. 
If the codimension of the intersection of the two manifolds within the one with the smaller dimension is bigger than two, the problem decouples completely and we show that recovery sequences on each of the manifolds do not interact in a significant way. 
Otherwise, one has to modify the recovery sequence on the manifold with larger dimension locally around the intersection so that the interaction with the other manifold does not contribute substantially to the energy. 
Lastly, if the two manifolds have equal dimension and their intersection has codimension one, we have to ensure the regularity of recovery sequence across the intersection to control the interaction across the intersection. 
This requires us to prove an approximation result for Sobolev functions on the two manifolds whose traces on the intersection coincide with smooth functions that enjoy the same property, see \cref{lem:density} in the appendix.

Before we state and prove the $\limsup$ inequality, we prove two useful lemmas which allow us to crudely bound the nonlocal energy from above by the local one.
We shall use these estimates for bounding certain error terms.
To prove them, we recall the Cheeger--Colding segment inequality which is satisfied for all Riemannian manifolds with a lower Ricci curvature bound and hence unconditionally for compact manifolds.
\begin{lemma}[Cheeger--Colding segment inequality \cite{cheeger1996lower}]\label{lem:Cheeger--Colding}
    Let $\M$ be a $d$-dimensional compact, connected Riemannian manifold equipped with the metric $d_\M$.
    For a function $f:\M\to[0,\infty)$ and $x,y \in \M$ let 
    \begin{align*}
        \mathcal F_f(x,y) = \inf_{\gamma_{xy}} \int_0^\ell f(\gamma_{xy}(t)) \d t,\qquad x,y\in\M,
    \end{align*}
    where the infimum is taken over all unit-speed \red minimizing  \nc geodesics $\gamma_{xy}$ from $x$ to $y$, where $\ell=d_\M(x,y)$.
     There exists $r_0=r_0(\M)>0$ and a constant $C=C(\M)>0$ such that for all $0<r<r_0$, and all open sets $A,B\subset B_\M(x_0,r)$  it holds that
    \begin{align*}
        \iint_{A\times B} \mathcal{F}_f \d\vol_{\M\times\M} \leq 
        C r \left(\vol(A)+\vol(B)\right)\int_{B_\M(x_0,2r)}f\d\vol.
    \end{align*}   
\end{lemma}

\begin{lemma}\label{lem:nonlocal_error}
    Let $\M\subset\R^N$ be a $d$-dimensional smooth, connected, compact manifold without boundary, $A\subset\M$ be an open subset, and $u\in \Lip(\M)$. 
    Under \cref{ass:eta} there exist $C,\eps_0>0$, independent of $u$, such that for all $0<\eps<\eps_0$ it holds
    \begin{align*}
        \int_{A}\int_{\M}
        \eta_\eps(\abs{x-y})\abs{u(x)-u(y)}^2\d\vol(y)\d\vol(x)
        \leq 
        C
        \eps^{2+d}
        \int_{A^\eps} 
        \abs{\nabla_\M
        u(y)}^2
        \d\vol(x),
    \end{align*}
    where $A^\eps = \{x\in\M\st d_\M(x,A)<\eps\}$ is the $\eps$-neighborhood of $A$ in $\M$.
\end{lemma}
                                                                                                                       \begin{proof}
    We use the notation $a\lesssim b$ to denote that $a\leq Cb$ for some constant $C>0$ independent of $a$ and $b$.
    Furthermore, when we write $\eps$ sufficiently small, this upper bound is independent of $u$.
    For $x,y\in\M$ we let $\gamma_{xy}:[0,d_\M(x,y)]\to\M$ denote a unit-speed \red minimizing \nc geodesic between $x$ and $y$.
    Note that for $\eps>0$ sufficiently small and $d_\M(x,y)<\eps$ there exists a unique unit-speed \red minimizing \nc geodesic connecting them.
    Using the fundamental theorem of calculus, Jensen's inequality and \cref{ass:eta} we get
    \begin{align*}
         \int_{A}\int_{\M} &
        \eta_\eps(\abs{x-y})\abs{u(x)-u(y)}^2\d\vol(y)\d\vol(x)
        \\
        &=
        \int_{A}\int_{\M}
        \eta_\eps(\abs{x-y})\abs{\int_0^{d_\M(x,y)}\frac{\d}{\d t}u(\gamma_{xy}(t))\d t}^2\d\vol(y)\d\vol(x)
        \\
        &\leq 
        \int_{A}\int_{\M}
        \eta_\eps(\abs{x-y})
        d_\M(x,y)
        \underbrace{\int_0^{d_\M(x,y)}\abs{\nabla_\M u(\gamma_{xy}(t))}^2\d t}_{=:\mathcal{F}(x,y)}\d\vol(y)\d\vol(x)
        \\
        &\lesssim
        \eps 
        \int_A 
        \int_{B_\M(x,2\eps)}
        \mathcal{F}(x,y)\d\vol(y)\d\vol(x).
    \end{align*}
    To apply the Cheeger--Colding segment inequality from \cref{lem:Cheeger--Colding} to this integral we need to localize it.
    Since $A$ is an open subset of a compact manifold we can cover it by finitely many geodesic balls of radius $\eps$, i.e., $A\subset\bigcup_{k=1}^K B_\M(x_k,\eps)$. 
    Hence, we get that
    \begin{align*}
        \int_A 
        \int_{B_\M(x,2\eps)}
        \mathcal{F}(x,y)\d\vol(y)\d\vol(x)
        &\leq 
        \sum_{k=1}^K
        \int_{B_\M(x_k,\eps)}
        \int_{B_\M(x,2\eps)}
        \mathcal{F}(x,y)\d\vol(y)\d\vol(x)
        \\
        &\leq 
        \sum_{k=1}^K
        \int_{B_\M(x_k,\eps)}
        \int_{B_\M(x_k,3\eps)}
        \mathcal{F}\d\vol_{\M\times\M}
        \\
        &\lesssim
        \eps^{1+d}
        \sum_{k=1}^K
        \int_{B_\M(x_k,3\eps)}
        \abs{\nabla_\M u}^2
        \d\vol,
    \end{align*}
    where we applied \cref{lem:Cheeger--Colding} in the last inequality.
    Note that we can pick a covering such that 
    \begin{align*}
        \sum_{k=1}^K
        \int_{B_\M(x_k,3\eps)}
        \abs{\nabla_\M u}^2
        \d\vol
        \lesssim
        \int_{A^\eps}
        \abs{\nabla_\M u}^2
        \d\vol.
    \end{align*}
Combing the three inequalities of this proof we obtain the assertion.
\end{proof}

\begin{lemma}\label{lem:nonlocal_error_T_map}
    Let $\M\subset\R^N$ be a $d$-dimensional smooth, connected, compact manifold without boundary, $A\subset\M$ be an open subset, $u\in \Lip(\M)$, and $T:\M\to\M$ a map which satisfies $\esssup_{x\in\M}\abs{T(x)-x}\leq\delta$ for some $\delta>0$. 
    Under \cref{ass:eta} there exist $C,\eps_0,\delta_0>0$, independent of $u$, such that for all $0<\eps<\eps_0$ and all $0<\delta<\delta_0$ it holds
    \begin{align*}
        \int_{A}\int_{\M}
        \eta_\eps(\abs{x-y})\abs{u(T(y))-u(y)}^2\d\vol(y)\d\vol(x)
        \leq 
        C
        \eps^d
        \delta^2
        \int_{A^{2\eps}}
        \esssup_{B_\M(x,2\delta)}
        \abs{\nabla_\M u}^2
        \d\vol(x).
    \end{align*}
\end{lemma}
\begin{proof}
Using Fubini's theorem and the fact that $\vol(B_\M(x,\eps))\lesssim\eps^d$ for $\eps>0$ sufficiently small, we get
\begin{align*}
     \int_{A}\int_{\M} &
    \eta_\eps(\abs{x-y})\abs{u(T(y))-u(y)}^2\d\vol(y)\d\vol(x)
    \\
    &=
    \int_{\M}
    \abs{u(T(y))-u(y)}^2
    \int_{A}
    \eta_\eps(\abs{x-y})\d\vol(x)
    \d\vol(y)
    \\
    &\lesssim 
    \eps^{d}
    \int_{\{y\in\M\st d(y,A)<\eps\}}
    \abs{u(T(y))-u(y)}^2\d\vol(y).
\end{align*}
For almost every $y\in\M$ we can pick a geodesic $\gamma_y:[0,1]\to\M$ with $\gamma_y(0)=y$ and $\gamma_y(1)=T(y)$.
Using the fundamental theorem of calculus and Jensen's inequality we get
\begin{align*}
     \int_{A}\int_{\M} &
    \eta_\eps(\abs{x-y})\abs{u(T(y))-u(y)}^2\d\vol(y)\d\vol(x)
    \\
    &\lesssim 
    \eps^{d}
    \int_{\{y\in\M\st d(y,A)<\eps\}}
    \int_0^1 
    \abs{
    \frac{\d}{\d s}
    u(\gamma_y(s))
    }^2
    \d s
    \d\vol(y)
    \\
    &\lesssim 
    \eps^d
    \delta^2
    \int_{\{y\in\M\st d(y,A)<\eps\}}
    \esssup_{B_\M(y,2\delta)}
    \abs{\nabla_\M u}^2
    \d\vol(y),
\end{align*}    
where we used that $d_\M(y,T(y))\leq 2\abs{y-T(y)}\leq 2\delta$ for $\delta>0$ sufficiently small.
We can conclude the proof by noting that $\{y\in\M\st d(y,A)<\eps\}\subset\{y\in\M\st d_\M(y,A)<2\eps\}=A^{2\eps}$ for $\eps>0$ sufficiently small.
\end{proof}

Now we continue with the proof of the limsup inequality for \cref{thm:gamma}. 
For notational convenience, we let $G_{n\sind{i},\eps_n}(\cdot;\M\sind{i})$ denote the standard graph Dirichlet energy from \cref{thm:Gamma-convergence_standard}, evaluated on the $n\sind{i}$ points in $V_n\cap\M\sind{i}$.

Furthermore, as before $\mathcal{G}(\cdot;\M\sind{i})$ denotes the standard local Dirichlet energy from the same theorem, evaluated on the manifold $\M\sind{i}$ using the probability measure $\mu\sind{i}=\rho\sind{i}\vol\sind{i}$.

\begin{proposition}[Limsup inequality]\label{prop:limsup_discrete}
Under \cref{ass:densities,ass:epsilon,ass:eta}  for all $u\in L^2(\M)$ there exists a sequence $(u_n)_{n\in\N}\subset L^2(\mu_n)$ such that $u_n\overset{TL^2}{\longrightarrow} u$ and
\begin{align*}
    \limsup_{n\to\infty} E_{n,\eps_n}(u_n)
    \leq 
    \E(u).
\end{align*}
\end{proposition}
\begin{proof}

The proof is divided into several cases based on the codimension of the intersection $\M\sind{12}$.
Recall that we assume that $d\sind{1}\leq d\sind{2}$.

\paragraph{Case 1, \texorpdfstring{$d\sind{1} \leq d\sind{2}$}{d1<=d2} and \texorpdfstring{$d\sind1-d\sind{12} > 2$}{d1-d12>2}.}
Since the energy $\E(u)$ is uncoupled in this case, one can assume without loss of generality that $u\vert_{\M\sind{i}}$ for $i\in\{1,2\}$ is smooth, as the general case follows by a standard density argument. 
                          Let us define
\[v\sind{i}
    = \frac{u}{\sqrt{\alpha\sind{i}\beta_\eta\sind{i}\rho\sind{i}}}
    \quad
    \text{on }\M\sind{i}\cap V_n.
\]
    We note
\begin{align}
    \label{eq:limsup_ineq_v}
    \lim_{n\to\infty}
    \left(
    \frac{n\sind{i}}{n}
    \right)^2
    G_{n\sind{i},\eps_n}(v\sind{i};\M\sind{i})
    =
    \left(\alpha\sind{i}\right)^2
    \mathcal{G}
    \left(
    v\sind{i};
    \M\sind{i}
    \right),
\end{align}
where these limits are known to exist almost surely by \labelcref{eq:alphas} as well as \cref{cor:Gamma-convergence_standard}, in the proof of which a constant recovery sequence is used.

We claim that the constant sequence $u\vert_{V_n}$ is a recovery sequence for $\E$.
To see this, note that for $A_{\eps} =  \M\sind{1}  \cup \{x \in \M\sind{2}\st  d(x, \M\sind{1})>\eps\}$ and $S_{\eps} = \M \setminus A_{\eps}$ we have 
\begin{align*}
   E_{n,\eps_n}(u) & = \frac{1}{\eps_n^2}
    \int_{\M}
    \int_{\M}
    \eta_{\eps_n}(\abs{x-y})
    \abs{
    \frac{v(x) \sqrt{\alpha\sind{i}\beta_\eta\sind{i}\rho\sind{i}(x)}}{\sqrt{\eta_{\eps_n}(\abs{x-\cdot})\star\mu_n}}
    -
    \frac{v(y) \sqrt{\alpha\sind{i}\beta_\eta\sind{i}\rho\sind{i}(y)}}{\sqrt{\eta_{\eps_n}(\abs{y-\cdot})\star\mu_n}}
    }^2
    \d\mu_n(x)\d\mu_n(y)
    \\
    &\leq 
    \frac{1}{\eps_n^2}
    \int_{A_{\eps_n}}
    \int_{A_{\eps_n}}
    \eta_{\eps_n}(\abs{x-y})
    \abs{
    \frac{v(x) \sqrt{\alpha\sind{i}\beta_\eta\sind{i}\rho\sind{i}(x)}}{\sqrt{\eta_{\eps_n}(\abs{x-\cdot})\star\mu_n}}
    -
    \frac{v(y) \sqrt{\alpha\sind{i}\beta_\eta\sind{i}\rho\sind{i}(y)}}{\sqrt{\eta_{\eps_n}(\abs{y-\cdot})\star\mu_n}}
    }^2
    \d\mu_n(x)\d\mu_n(y) \\
    & \phantom{\leq}  
      + \frac{2}{\eps_n^2}
    \int_{\M}
    \int_{S_{\eps_n}}
    \eta_{\eps_n}(\abs{x-y})
    \abs{
    \frac{v(x) \sqrt{\alpha\sind{i}\beta_\eta\sind{i}\rho\sind{i}(x)}}{\sqrt{\eta_{\eps_n}(\abs{x-\cdot})\star\mu_n}}
    -
    \frac{v(y) \sqrt{\alpha\sind{i}\beta_\eta\sind{i}\rho\sind{i}(y)}}{\sqrt{\eta_{\eps_n}(\abs{y-\cdot})\star\mu_n}}
    }^2
    \d\mu_n(x)\d\mu_n(y) \\
    & =: E^{far}_{n,\eps_n}(u) +E^{near}_{n,\eps_n}(u).
\end{align*}
Since, by \cref{lem:uniform_cvgc},
$ \chi_{A_{\eps_n}} \eps_n^{-d\sind{i}} \eta_{\eps_n}(\abs{x-\cdot})\star\mu_n -\chi_{A_{\eps_n}} \alpha\sind{i}\beta_\eta\sind{i}\rho\sind{i}(x)$ converges uniformly to zero, and also using \labelcref{eq:limsup_ineq_v}, we have that 
\[ \limsup_{n \to \infty} E^{far}_{n,\eps_n}(u) \leq \limsup_{n\to\infty} \sum_{i=1}^2 
\left(
\frac{n\sind{i}}{n}
\right)^2
G_{n\sind{i},\eps_n}(v\sind{i};\M\sind{i})\leq \sum_{i=1}^2  
\left(\alpha\sind{i}\right)^2
\mathcal{G}\left(v\sind{i};\M\sind{i}\right) = \E(u).
\]
So it remains to show that $\lim_{n\to \infty} E^{near}_{n,\eps_n}(u) = 0$.
By \cref{rem:degrees_lb}, on $\M\sind{i}$ it holds $\eps_n^{-d\sind{i}} \eta_{\eps_n}(\abs{x-\cdot})\star\mu \geq C>0$ for $n$ large enough. 

Note that by the first inequality of \cref{lem:degrees} we have $\eta_{\eps_n}(\abs{x-\cdot})\star\mu_n\restr S_{2\eps_n} \leq C \eps_n\sind{2}$ and, furthermore, that by a concentration inequality analogous to one of \cref{lem:degrees} we have that $\mu_n\sind{2}(S_{2\eps_n}) \leq C \eps_n^{d\sind{2} - d\sind{12}}$ and $\mu_n\sind{1}(\{x \in \M\sind{1} \st  d(x, \M\sind{2}) < \eps_n\}) \leq C \eps_n^{d\sind{2} - d\sind{12}}$. Thus, the reminder of the proof is practically identical to the estimate in the proof of limsup inequality in \cref{sec:NL_limsup}.

\paragraph{Case 2, \texorpdfstring{$d\sind{1} \leq d\sind{2}$}{d1<=d2} and \texorpdfstring{$d\sind{1}-d\sind{12} \leq 2 \leq d\sind{2}-d\sind{12}$}{d1-d12<=2<=d2-d12}.}

In this case the previous argument breaks down and the cross term does not vanish in the limit.  Thus we have to construct a non-trivial recovery sequence which makes sure that the interaction term between both manifolds converges to zero as $n\to\infty$.
As before, the energy $\E(u)$ is uncoupled and one can assume without loss of generality that the functions $u\sind{i}:=u\vert_{\M\sind{i}}$ for $i\in\{1,2\}$ are smooth and apply a diagonal argument to build a recovery sequence for general functions.

For $i\in\{1,2\}$, let us define the normalized Dirichlet energy $\E(\cdot;\M\sind{i})$ on $\M\sind{i}$ as
\begin{align*}
    \E(u;\M\sind{i}) := 
    \frac{\sigma_\eta\sind{i}}{\beta_\eta\sind{i}}
    \int_{\M\sind{i}}\abs{\nabla_{\M\sind{i}} \left( \frac{u\sind{i}}{\sqrt{\rho\sind{i}}} \right)\rho\sind{i}}^2 \d\vol\sind{i}.
\end{align*}
The recovery sequence which we shall construct is denoted by $u_n$ and for vertices lying in the lower-dimensional manifold $\M\sind{1}$ it is defined by
\begin{align} \label{eq:limsup_uM1}
    u_n(x) := u\sind{1}(x)
    \frac{\sqrt{\eps_n^{-d\sind{1}}\deg_{n,\eps_n}(x)}}{\sqrt{\alpha\sind{1}\beta_\eta\sind{1}\rho\sind{1}(x)}}
    ,\quad x\in V_n\cap \M\sind{1}.
\end{align}
  Note that from \cref{cor:uniform_cvgc2} it follows that $u_n$ converges uniformly to $u\sind{1}$ on $\M\sind{1}$.

On the higher-dimensional one, we need  a more subtle modification of the function $u\sind{2}$.
For this let us take $\eps_n>0$ so small that the geodesic projection $\pi : \M\sind{2} \to \M\sind{12}$ onto the intersection is uniquely defined in an $\eps_n$-neighborhood of $\M\sind{12}$. 
We abbreviate the distance of $x\in\M\sind{2}$ to $\M\sind{12}$ as $d(x):= \dist(x,\M\sind{12})$ and subdivide the higher-dimensional manifold $\M\sind{2}$ into a near, a mid, and a far region with respect to the distance to the intersection $\M\sind{12}$:
\begin{align*}
    \M\sind{2}_{\mathrm{near}} &:= \{x\in\M\sind{2}\st d(x)<C\eps_n\},
    \\
    \M\sind{2}_{\mathrm{mid}} &:= \{x\in\M\sind{2}\st C\eps_n\leq d(x)<C\sqrt{\eps_n}\},
    \\
    \M\sind{2}_{\mathrm{far}} &:= \{x\in\M\sind{2}\st d(x)>C\sqrt{\eps_n}\}.
\end{align*}
Here, the constant $C>0$ is chosen such that for $n\in\N$ large enough (and hence $\eps_n>0$ small enough) we have that $B(x,\eps)\cap\M\sind{1}=\emptyset$ for all $x\in\M\sind{2}\setminus\M\sind{2}_\mathrm{near}$ which is possible thanks to \cref{lem:angle}.

With the partitioning of $\M\sind{2}$ into a near, a middle, and a far region at hand, we can now define the recovery sequence as a logarithmic interpolation between the value of $u\sind{1}$ on the intersection $\M\sind{12}$ and $u\sind{2}$ in the middle region $\M\sind{2}_\mathrm{mid}$.
More precisely, it is defined as
\begin{align} \label{eq:limsup_uM2}
    u_n\sind{2}(x)
    :=
    \begin{dcases}
        u\sind{1}(\pi(x))\sqrt{\frac{\eps_n^{-d\sind{1}}\deg_{n,\eps_n}(x)}{\alpha\sind{1}\beta_\eta\sind{1}\rho\sind{1}(\pi(x))}},&\text{if }x\in\M\sind{2}_{\mathrm{near}},\\
        u\sind{2}(x)\sqrt{\frac{\eps_n^{-d\sind{2}}\deg_{n,\eps_n}(x)}{\alpha\sind{2}\beta_\eta\sind{2}\rho\sind{2}(x)}}
        +
        \frac{\log\sqrt{\eps_n} - \log d(x)}{\log\sqrt{\eps_n} - \log\eps_n}\times
        &
        \\
        \times\left(
        u\sind{1}(\pi(x))\sqrt{\frac{\eps_n^{-d\sind{1}}\deg_{n,\eps_n}(x)}{\alpha\sind{1}\beta_\eta\sind{1}\rho\sind{1}(\pi(x))}}
        -
        u\sind{2}(x)\sqrt{\frac{\eps_n^{-d\sind{2}}\deg_{n,\eps_n}(x)}{\alpha\sind{2}\beta_\eta\sind{2}\rho\sind{2}(x)}}
        \right),
        &\text{if }x\in\M\sind{2}_{\mathrm{mid}},\\
        u\sind{2}(x)
        \sqrt{\frac{\eps_n^{-d\sind{2}}\deg_{n,\eps_n}(x)}{\alpha\sind{2}\beta_\eta\sind{2}\rho\sind{2}(x)}},&\text{if } x\in\M\sind{2}_{\mathrm{far}},\\
    \end{dcases}
\end{align}
Our goal is to show $u_n\overset{TL^2}{\longrightarrow} u$ as well as the upper bound $\limsup_{n\to\infty}E_{n,\eps_n}(u_n)\leq\E(u)$ which would make $u_n$ a recovery sequence.
First note that $u_n\to u$ in the $TL^2$-sense follows from similar arguments as \cref{lem:weights_converge} using $d\sind{1}=d\sind{2}$, the Lipschitz continuity of $u\sind{i}$ for $i\in\{1,2\}$, and \cref{ass:densities}. 
           
For proving the desired upper bound, we note that we can split the energy $E_{n,\eps_n}(u_n)$ into a sum over vertices in $\M\sind{1}\times\M\sind{1}$, two over $\M\sind{1}\times\M\sind{2}$, and one over $\M\sind{2}\times\M\sind{2}$ which we now estimate. 
For the sum over $\M\sind{1}\times\M\sind{1}$ we observe that by the definition of $u_n$ it holds that
\begin{align}
    \nonumber
    &
    \limsup_{n\to\infty}
    \frac{1}{n^2\eps_n^2}
    \sum_{x,y\in V_n\cap\M\sind{1}}
    \eta_{\eps_n}(\abs{x-y})
    \abs{\frac{u_n(x)}{\sqrt{\deg_{n,\eps_n}(x)}}-\frac{u_n(y)}{\sqrt{\deg_{n,\eps_n}(y)}}}^2
    \\
    \nonumber
    &=
    \frac{1}{\alpha\sind{1}\beta_\eta\sind{1}}
    \limsup_{n\to\infty}
    \left(\frac{n\sind{1}}{n}\right)^2
    \frac{1}{(n\sind{1})^2\eps_n^2}
    \sum_{x,y\in V_n\cap\M\sind{1}}
    \eta_{\eps_n}(\abs{x-y})
    \abs{\frac{u\sind{1}(x)}{\sqrt{\rho\sind{1}(x)}}-\frac{u\sind{1}(y)}{\sqrt{\rho\sind{1}(y)}}}^2
    \\
    \label{eq:M1_M1}
    &=
    \alpha\sind{1}
    \E(u\sind{1};\M\sind{1})
\end{align}
since the Lipschitz function $\frac{u\sind{1}}{\sqrt{\rho\sind{1}}}$ is a recovery sequence for $\E(u\sind{1};\M\sind{1})$.

We continue with the sum over vertices $\M\sind{1}\times\M\sind{2}$. We note that $u_n\sind{2}$ is defined in such a way that these cross-interaction terms are small. 
\begin{align}
    \nonumber 
    &
    \frac{1}{n^2\eps_n^2}
    \sum_{\substack{y\in V_n\cap\M\sind{1}\\
    x\in V_n\cap\M\sind{2}}}
    \eta_{\eps_n}(\abs{x-y})
    \abs{\frac{u_n(x)}{\sqrt{\deg_{n,\eps_n}(x)}}-\frac{u_n(y)}{\sqrt{\deg_{n,\eps_n}(y)}}}^2
    \\
    \nonumber 
    &=
    \frac{1}{n^2\eps_n^{2+d\sind{1}}}
    \frac{1}{\alpha\sind{1}\beta_\eta\sind{1}}
    \sum_{\substack{y\in V_n\cap\M\sind{1}\\
    x\in V_n\cap\M\sind{2}}}
    \eta_{\eps_n}(\abs{x-y})
    \abs{\frac{u\sind{1}(\pi(x))}{\sqrt{\rho\sind{1}(\pi(x))}}-\frac{u\sind{1}(y)}{\sqrt{\rho\sind{1}(y)}}}^2
    \\
    \nonumber 
    &\leq 
    \frac{C}{\eps_n^{2+d\sind{1}}}
    \int_{\M\sind{1}}
    \int_{\M\sind{2}}    \eta_{\eps_n}\left(\abs{T_n\sind{2}(x)-T_n\sind{1}(y)}\right)
    \abs{\frac{u\sind{1}(\pi(T_n\sind{2}(x)))}{\sqrt{\rho\sind{1}(\pi(T_n\sind{2}(x)))}}-\frac{u\sind{1}(T_n\sind{1}(y))}{\sqrt{\rho\sind{1}(T_n\sind{1}(y))}}}^2 
    \\\nonumber
    & \phantom{C
    \frac{1}{\eps_n^{2+d\sind{1}}}
    \int_{\M\sind{1}}
    \int_{\M\sind{2}}
    \eta_{\eps_n}\left(\abs{T_n\sind{2}(x)-T_n\sind{1}(y)}\right)} 
    \d\mu\sind{2}(x)\d\mu\sind{1}(y)
    \\
    \nonumber 
    &\leq     
    C\frac{\tilde\eps_n^2}{\eps_n^{2+d\sind{1}}}
    \int_{\M\sind{1}}
    \int_{\M\sind{2}}
    \eta_{\tilde\eps_n}(\abs{x-y})
    \d\mu\sind{2}(x)\d\mu\sind{1}(y)
    \\
    &=
    \label{eq:M1_M2}
    O\left(\eps_n^{-d\sind{1}}\tilde\eps_n^{d\sind{1}+d\sind{2}-d\sind{12}}\right)
    =
    O(\eps_n^2)
    \quad
    \te{as }
    n\to\infty.
\end{align} 
where we used that ${u\sind{1}}/{\sqrt{\rho\sind{1}}}$ is Lipschitz continuous and that the codimension $d\sind{2}-d\sind{12}$ is at least two.
As in Case 1, we assumed without loss of generality that $\eta(t)=1_{[0,1]}(t)$, used the quantity $\tilde\eps_n$ which satisfies $\tilde\eps_n>\eps_n$ and $\lim_{n\to\infty}\frac{\eps_n}{\tilde\eps_n}=1$, and utilized the same transport maps $T_n\sind{i}:\M\sind{i}\to\M\sind{i}$.
Furthermore, the constant $C$ changes its value between the lines.

It remains to estimate the sum over $\M\sind{2}\times\M\sind{2}$ for which we introduce 
        the following abbreviation
\begin{align*}
    S(n,\eps_n) := 
    \left\{
    x \in V_n\cap\M\sind{2}\st x\in\M\sind{2}_\mathrm{near}\cup\M\sind{2}_\mathrm{mid}\text{ or }\dist(x,\M\sind{2}_\mathrm{mid})\leq\eps_n
    \right\}
\end{align*} 
for the set of graph points that fall into the near or mid region, or fall into the far region but have Euclidean distance at most $\eps$ to the mid region.
In particular we have that
\begin{align*}
    \left(V_n\cap\M\sind{2}\right)\setminus S(n,\eps_n)
    =
    \left\{x \in V_n\cap\M\sind{2}\st x\in\M\sind{2}_\mathrm{far}\text{ and }
    \dist(x,\M\sind{2}_\mathrm{mid})>\eps_n
     \right\},
\end{align*}
so points in that set do not interact with points in the near or the middle region.
Using this definition and the fact that $u_n=u\sind{2}\sqrt{\frac{\eps_n^{-d\sind{2}}\deg_{n,\eps_n}}{\alpha\sind{2}\beta_\eta\sind{2}\rho\sind{2}}}$ on $V_n\cap\M\sind{2}_\mathrm{far}$ we split the sum into  the terms that account for interaction only in the far region, that accounts for the bulk of the energy on $\M\sind{2}$, and a term that accounts for the transition layers and their interactions within $\M\sind{2}$, which we show vanish in the limit
\begin{align}
    \nonumber
     \frac{1}{n^2\eps_n^2} &
    \sum_{x,y\in V_n\cap\M\sind{2}}
    \eta_{\eps_n}(\abs{x-y})
    \abs{\frac{u_n(x)}{\sqrt{\deg_{n,\eps_n}(x)}}-\frac{u_n(y)}{\sqrt{\deg_{n,\eps_n}(y)}}}^2
    \\
    \nonumber
    &=
    \frac{1}{n^2\eps_n^2\alpha\sind{2}\beta_\eta\sind{2}}
    \sum_{\substack{x \in V_n\cap\M\sind{2}_\mathrm{far}\\
    x\notin S(n,\eps_n)}}
    \sum_{y\in V_n\cap\M\sind{2}}
    \eta_{\eps_n}(\abs{x-y})
    \abs{\frac{u\sind{2}(x)}{\sqrt{\rho\sind{2}(x)}}-\frac{u\sind{2}(y)}{\sqrt{\rho\sind{2}(y)}}}^2
    \\
    \nonumber
    &\qquad
    +
    \frac{1}{n^2\eps_n^2}
    \sum_{x\in S(n,\eps_n)}
    \sum_{y\in V_n\cap\M\sind{2}}
    \eta_{\eps_n}(\abs{x-y})
    \abs{\tilde u_n(x)-\tilde u_n(y)}^2
    \\
    \nonumber
    &\leq 
    \frac{1}{n^2\eps_n^2\alpha\sind{2}\beta_\eta\sind{2}}
    \sum_{x,y \in V_n\cap\M\sind{2}}
    \eta_{\eps_n}(\abs{x-y})
    \abs{\frac{u\sind{2}(x)}{\sqrt{\rho\sind{2}(x)}}-\frac{u\sind{2}(y)}{\sqrt{\rho\sind{2}(y)}}}^2
    \\
    \label{eq:M2_M2_part_1}
    &\qquad
    +
    \frac{1}{n^2\eps_n^2}
    \sum_{x\in S(n,\eps_n)}
    \sum_{y\in V_n\cap\M\sind{2}}
    \eta_{\eps_n}(\abs{x-y})
    \abs{\tilde u_n(x)-\tilde u_n(y)}^2
\end{align}
where the function $\tilde u_n$ is the recovery sequence divided by the square root of the degrees and equals
\begin{align}\label{eq:u_n_tilde}
    \tilde u_n(x) = 
    \begin{dcases}
        \frac{u\sind{1}(\pi(x))}{\sqrt{\eps_n^{d\sind{1}}\alpha\sind{1}\beta_\eta\sind{1}\rho\sind{1}(\pi(x))}},
        \quad
        &x\in\M\sind{2}_\mathrm{near},
        \\
        \frac{u\sind{2}(x)}{\sqrt{\eps_n^{d\sind{2}}\alpha\sind{2}\beta_\eta\sind{2}\rho\sind{2}(x)}}
        +
        \frac{\log\sqrt{\eps_n}-\log d(x)}{\log\sqrt{\eps_n}-\log\eps_n}\times
        &
        \\
        \times\left(\frac{u\sind{1}(\pi(x))}{\sqrt{\eps_n^{d\sind{1}}\alpha\sind{1}\beta_\eta\sind{1}\rho\sind{1}(\pi(x))}} - \frac{u\sind{2}(x)}{\sqrt{\eps_n^{d\sind{2}}\alpha\sind{2}\beta_\eta\sind{2}\rho\sind{2}(x)}}\right),
        \quad
        &x\in\M\sind{2}_{\mathrm{mid}},
        \\
        \frac{u\sind{2}(x)}{\sqrt{\eps_n^{d\sind{2}}\alpha\sind{2}\beta_\eta\sind{2}\rho\sind{2}(x)}},
        \quad
        &x\in\M\sind{2}_\mathrm{far}.
    \end{dcases}
\end{align}
Since the function $\frac{u\sind{2}}{\sqrt{\rho\sind{2}}}$ is Lipschitz, similar to before we have for the first term in \labelcref{eq:M2_M2_part_1} that
\begin{align}\label{eq:M2_M2_part_2}
    \limsup_{n\to\infty}
    \frac{1}{n^2\eps_n^2\beta_\eta\sind{2}}
    \sum_{x,y \in V_n\cap\M\sind{2}}
    \eta_{\eps_n}(\abs{x-y})
    \abs{\frac{u\sind{2}(x)}{\sqrt{\rho\sind{2}(x)}}-\frac{u\sind{2}(y)}{\sqrt{\rho\sind{2}(y)}}}^2
    \leq 
    \alpha\sind{2}
    \E(u\sind{2};\M\sind{2}).
\end{align}
We claim that the second term in \labelcref{eq:M2_M2_part_1} goes to zero as $n\to\infty$ and we estimate it as follows
\begin{align*}
    &
    \frac{1}{n^2\eps_n^2}
    \sum_{x\in S(n,\eps_n)}
    \sum_{y\in V_n\cap\M\sind{2}}
    \eta_{\eps_n}(\abs{x-y})
    \abs{\tilde u_n(x)-\tilde u_n(y)}^2
    \\
    &=
    \left(\frac{n\sind{2}}{n}\right)^{\!2}\!
    \frac{1}{\eps_n^2}
    \int_{D(n,\eps_n)}
    \int_{\M\sind{2}}
    \eta_{\eps_n}\left(\abs{T_n\sind{2}(x)-T_n\sind{2}(y)}\right)
    \abs{\tilde u_n(T_n\sind{2}(x))-\tilde u_n(T_n\sind{2}(y))}^2
    \d\mu\sind{2}(x)
    \d\mu\sind{2}(y)
    \\
    &\leq 
    \left(\frac{n\sind{2}}{n}\right)^{\!2}\!
    \frac{1}{\eps_n^2}
    \int_{D(n,\eps_n)}
    \int_{\M\sind{2}}
    \eta_{\tilde\eps_n}\left(\abs{x-y}\right)
    \abs{\tilde u_n\circ T_n\sind{2}(x)-\tilde u_n\circ T_n\sind{2}(y)}^2
    \d\mu\sind{2}(x)
    \d\mu\sind{2}(y),
\end{align*}
where $D(n,\eps_n)=\{x\in\M\sind{2}\st T_n\sind{2}(x)\in S(n,\eps_n)\}$.
We claim that in the limit $n\to\infty$ we can replace $\tilde u_n\circ T_n\sind{2}$ by $\tilde u_n$.
To show this, similarly as in \cite[p.267]{garcia2019variational}, we define
\begin{align*}
    A_n &:= \frac{1}{\eps_n^2}
    \int_{D(n,\eps_n)}
    \int_{\M\sind{2}}
    \eta_{\tilde\eps_n}\left(\abs{x-y}\right)
    \abs{\tilde u_n(x)-\tilde u_n(y)}^2
    \d\mu\sind{2}(x)
    \d\mu\sind{2}(y)
    \\
    B_n &:= \frac{1}{\eps_n^2}
    \int_{D(n,\eps_n)}
    \int_{\M\sind{2}}
    \eta_{\tilde\eps_n}\left(\abs{x-y}\right)
    \abs{\tilde u_n\circ T_n\sind{2}(x)-\tilde u_n\circ T_n\sind{2}(y)}^2
    \d\mu\sind{2}(x)
    \d\mu\sind{2}(y).
\end{align*}
Applying \cref{lem:nonlocal_error_T_map} we obtain that
\begin{align}
    \notag
    &
    (\sqrt{A_n}-\sqrt{B_n})^2
    \\
    \notag 
    &\leq 
    \frac{1}{\eps_n^2}
    \int_{D(n,\eps_n)}
    \int_{\M\sind{2}}
    \eta_{\tilde\eps_n}\left(\abs{x-y}\right)
    \abs{\tilde u_n\circ T_n\sind{2}(x)-\tilde u_n(x)+\tilde u_n(y)-\tilde u_n\circ T_n\sind{2}(y)}^2
    \d\mu\sind{2}(x)
    \d\mu\sind{2}(y)
    \\
    \notag 
    &\leq 
    \frac{2}{\eps_n^2}
    \int_{D(n,\eps_n)}
    \int_{\M\sind{2}}
    \eta_{\tilde\eps_n}\left(\abs{x-y}\right)
    \abs{\tilde u_n\circ T_n\sind{2}(x)-\tilde u_n(x)}^2
    +
    \abs{\tilde u_n(y)-\tilde u_n\circ T_n\sind{2}(y)}^2
    \d\mu\sind{2}(x)
    \d\mu\sind{2}(y)
    \\
    \label{eq:sA-sB}
    &\lesssim
    \left(\frac{\delta_n}{\eps_n}\right)^2
    \tilde\eps_n^{d\sind{2}}
    \int_{D(n,\eps_n)^{2\eps_n}}
    \sup_{B_\M(x,2\delta_n)}
    \abs{
    \nabla_{\M\sind{2}}\tilde u_n}^2
    \d\vol\sind{2}(x),
\end{align}
where $\delta_n := \max_{i\in\{1,2\}}\norm{T_n\sind{i}-\operatorname{id}}_{L^\infty(\M\sind{i})}$.
Applying \cref{lem:nonlocal_error} we also get that
\begin{align}\label{eq:A}
    A_n \lesssim 
    \eps_n^{d\sind{2}}
    \int_{D(n,\eps_n)^{\eps_n}}
    \abs{\nabla_{\M\sind{2}}\tilde u_n}^2
    \d\vol\sind{2}.
\end{align}
Let us assume for now that we already know that $(\sqrt{A_n}-\sqrt{B_n})^2$ in \labelcref{eq:sA-sB} goes to zero as $n\to\infty$ and that also $A_n$ goes to zero by \labelcref{eq:A}. 
This implies that $\abs{A_n-B_n}\to 0$ as $n\to\infty$ and, by applying \cref{lem:nonlocal_error} again, we can estimate the limies superior of the second term in \labelcref{eq:M2_M2_part_1} as follows:
\begin{align*}
     \limsup_{n\to\infty} &
    \frac{1}{n^2\eps_n^2}
    \sum_{x\in S(n,\eps_n)}
    \sum_{y\in V_n\cap\M\sind{2}}
    \eta_{\eps_n}(\abs{x-y})
    \abs{\tilde u_n(x)-\tilde u_n(y)}^2
    \\
    &\leq 
    \left(\alpha\sind{2}\right)^2
    \limsup_{n\to\infty}
    A_n 
    \leq 
    \left(\alpha\sind{2}\right)^2
    \limsup_{n\to\infty}
    \eps_n^{d\sind{2}}
    \int_{D(n,\eps_n)^{\eps_n}}
    \abs{\nabla_{\M\sind{2}}\tilde u_n}^2
    \d\vol\sind{2}
    = 0
\end{align*}
and therefore as a consequence of \labelcref{eq:M2_M2_part_1,eq:M2_M2_part_2} we get
\begin{align*}
    &\phantom{{}={}}
    \limsup_{n\to\infty}
    \frac{1}{n^2\eps_n^2}
    \sum_{x,y\in V_n\cap\M\sind{2}}
    \eta_{\eps_n}(\abs{x-y})
    \abs{\frac{u_n(x)}{\sqrt{\deg_{n,\eps_n}(x)}}-\frac{u_n(y)}{\sqrt{\deg_{n,\eps_n}(y)}}}^2
    \leq 
    \alpha\sind{2}
    \E(u\sind{2};\M\sind{2}).
\end{align*}
In combination with \labelcref{eq:M1_M1,eq:M1_M2} and the definition of $\E$ in \labelcref{eq:limit_fctl} this would conclude the proof.

To show that the right-hand sides in \labelcref{eq:sA-sB,eq:A} indeed converge to zero, and taking into account \cref{ass:epsilon,rem:scaling_condition}, it suffices to prove that
\begin{align*}
    \lim_{n\to\infty}\eps_n^{d\sind{2}}\int_{D(n,\eps_n)^{2\eps_n}}\sup_{B_\M(x,2\delta_n)}\abs{\nabla_{\M\sind{2}}\tilde u_n}^2\d\vol\sind{2}(x)
    =
    0.
\end{align*}
From the definition of $\tilde u_n$ in \labelcref{eq:u_n_tilde} it is immediately clear that
\begin{align*}
    \lim_{n\to\infty}\eps_n^{d\sind{2}}\int_{A}\sup_{B_\M(x,\delta_n)}\abs{\nabla_{\M\sind{2}}\tilde u_n}^2\d\vol\sind{2}(x)
    =
    0
\end{align*}
if $A\subset\{x\in\M_\mathrm{near}\sind{2}\st\dist(x,\M_\mathrm{mid}\sind{2})>2\delta_n\}$ or if $A\subset\{x\in\M_\mathrm{far}\sind{2}\st\dist(x,\M_\mathrm{mid}\sind{2})>2\delta_n\}$.
Therefore, it suffices to compute the integral over the following neighborhood of the middle region $A:=\{x\in\M\sind{2}\st \dist(x,\M_\mathrm{mid}\sind{2})<2\delta_n\}$.
Taking into account the definition of $\tilde u_n$ in \labelcref{eq:u_n_tilde} and using the product rule, we see that on $\M\sind{2}_\mathrm{mid}$ the gradient $\nabla_{\M\sind{2}}\tilde u_n$ consists of three terms: the first is the gradient of $\frac{u\sind{1}\circ\pi}{\sqrt{\eps_n^{d\sind{1}}\alpha\sind{1}\beta_\eta\sind{1}\rho\sind{1}\circ\pi}}$, the second is the gradient of the logarithmic interpolation times the parentheses, and the third is the logarithmic interpolation factor (which lies between $0$ and $1$) times the gradient of the parentheses.
We note that terms whose gradients are bounded can be neglected, since the size of the domain vanishes as $n \to \infty$ and thus their contribution vanishes in the limit.
Since the functions $u\sind{i}/\sqrt{\rho\sind{i}}$ for $i=1,2$ are Lipschitz continuous, it therefore suffices to consider the integral over the second of these terms.
Using that 
\begin{align*}
    \abs{\nabla_{\M\sind{2}}
    \frac{\log\sqrt{\eps_n}-\log d(x)}{\log\sqrt{\eps_n}-\log\eps_n}}
    =
    \frac{2}{\abs{\log\eps_n}d(x)}
\end{align*}
we have
\begin{align*}
    \sup_{B_\M(x,2\delta_n)}
    \abs{\nabla_{\M\sind{2}}
    \frac{\log\sqrt{\eps_n}-\log d(x)}{\log\sqrt{\eps_n}-\log\eps_n}}^2
    =
    \frac{4}{\abs{\log\eps_n}^2(d(x)-2\delta_n)^2}.
\end{align*}
Hence, using also the coarea formula we have
\begin{align*}
     & \eps_n^{d\sind{2}} 
    \left(\norm{\frac{u\sind{1}\circ\pi}{\sqrt{\eps^{d\sind{1}}\alpha\sind{1}\beta_\eta\sind{1}\rho\sind{1}\circ\pi}}}_\infty^2
    +
    \norm{\frac{u\sind{2}}{\sqrt{\eps^{d\sind{2}}\alpha\sind{2}\beta_\eta\sind{2}\rho\sind{2}}}}_\infty^2
    \right) \times \\
    & \times 
    \int_A
    \sup_{B_\M(x,2\delta_n)}
    \abs{\nabla_{\M\sind{2}}
    \frac{\log\sqrt{\eps_n}-\log d(x)}{\log\sqrt{\eps_n}-\log\eps_n}}^2
    \d\vol\sind{2}(x)
    \\
    & \qquad\lesssim
    \frac{1}{\abs{\log\eps_n}^2}
    \int_A
    \frac{1}{(d(x)-2\delta_n)^2}
    \d\vol\sind{2}(x)
    \\
    &\qquad\lesssim
    \frac{1}{\abs{\log\eps_n}^2}
    \int_{C\eps_n-2\delta_n}^{C\sqrt{\eps_n}+2\delta_n}
    \mathcal H^{d\sind{2}-1}\left(\{y\in\M\sind{2}\st d(y)=t-2\delta_n\}\right)
    \frac{1}{(t-2\delta_n)^2}
    \d t
    \\
    &\qquad\lesssim
    \frac{1}{\abs{\log\eps_n}^2}
    \int_{C\eps_n-4\delta_n}^{C\sqrt{\eps_n}}
    \mathcal H^{d\sind{2}-1}\left(\{y\in\M\sind{2}\st d(y)=s\}\right)
    \frac{1}{s^2}
    \d s
    \\
    &\qquad\lesssim
    \frac{1}{\abs{\log\eps_n}^2}
    \int_{C\eps_n-4\delta_n}^{C\sqrt{\eps_n}}
    s^{d\sind{2}-d\sind{12}-3}
    \d s
    =
    O\left(\frac{1}{\abs{\log\eps_n}}\right)
\end{align*}
where we used $d\sind{2}-d\sind{12}\geq 2$ and inferred $\lim_{n\to\infty}\frac{\delta_n}{\eps_n}=0$ from \cref{ass:epsilon,rem:scaling_condition}.

\paragraph{Case 3, \texorpdfstring{$d\sind{1} = d\sind{2} =:d$}{d1=d2=:d} and \texorpdfstring{$d-d\sind{12} =1$}{d-d12=1}}
                                   
Using \cref{lem:trace} as in the proof of the nonlocal limsup inequality for \cref{thm:Gamma_NL} in \cref{sec:NL_limsup}, we can assume that $v:=u/\sqrt{\rho}$ is Lipschitz continuous on $\M$.
Similar to the argument there, we consider the sequence
\begin{align*}
    u_n(x) := u(x)\sqrt{\frac{\eps_n^{-d}\deg_{n,\eps_n}(x)}{\alpha\sind{i}\beta_\eta\rho\sind{i}(x)}},
    \qquad
    x\in\M\sind{i}\cap V_n,
\end{align*}
where we dropped the index $i$ on $\beta_\eta$ since $d\sind{1}=d\sind{2}=d$.
Again, the idea is that the square-root term converges to one and its numerator cancels once plugged into the normalized graph Dirichlet energy $E_{n,\eps_n}$.

Using the usual definition $v:=\frac{u}{\sqrt{\alpha\sind{i}\beta_\eta\rho\sind{i}}}$ we have
\begin{align*}
    E_{n,\eps_n}(u_n)
    &=
    \sum_{i=1}^2
    \left(\frac{n\sind{i}}{n}\right)^2
    G_{n,\eps_n}
    \left(
    v   
    ;\M\sind{i}\right)
    +
    \frac{2}{n^2\eps_n^{2+d}}
    \sum_{\substack{x\in V_n\cap\M\sind{1}\\y\in V_n\cap\M\sind{2}}}
    \eta_{\eps_n}(\abs{x-y})
    \abs{v(x)-v(y)}^2.
\end{align*}
Note that $v$ is globally Lipschitz on $\M$. 
In particular, by \cref{cor:Gamma-convergence_standard} (in the proof of which a constant recovery is used for Lipschitz functions) and the definition of $\E$ in \labelcref{eq:limit_fctl} we have almost surely
\begin{align*}
    \limsup_{n\to\infty}\sum_{i=1}^2
    \left(\frac{n\sind{i}}{n}\right)^2
    G_{n\sind{i},\eps_n}
    \left(
    v
    ;\M\sind{i}\right)
    =
    \sum_{i=1}^2
    \left(\alpha\sind{i}\right)^2    
    \mathcal G
    \left(
    v   
    ;\M\sind{i}\right)
    =
    \E(u).
\end{align*}
It remains to show that the cross term in the above energy decomposition of $E_{n,\eps_n}(u_n)$ tends to zero for which we shall use Lipschitzness of $v$:
\begin{align*}
    &\phantom{{}={}}
    \limsup_{n\to\infty}
    \frac{2}{n^2\eps_n^{2+d}}
    \sum_{\substack{x\in V_n\cap\M\sind{1}\\y\in V_n\cap\M\sind{2}}}
    \eta_{\eps_n}(\abs{x-y})
    \abs{v(x)-v(y)}^2
    \\
    &\leq 
    \limsup_{n\to\infty}
    \frac{\Lip(v)^2}{n^2\eps_n^{d}}
    \sum_{\substack{x\in V_n\cap\M\sind{1}\\y\in V_n\cap\M\sind{2}}}
    \eta_{\eps_n}(\abs{x-y})
    \\
    &=
    \limsup_{n\to\infty}
    \frac{n\sind{1}n\sind{2}}{n^2}
    \frac{\Lip(v)^2}{\eps_n^{d}}
    \int_{\M\sind{1}}
    \int_{\M\sind{2}}
    \eta_{\eps_n}\left(\abs{T_n\sind{1}(x)-T_n\sind{2}(y)}\right)
    \d\mu\sind{2}(x)
    \d\mu\sind{1}(y)
    \\
    &\leq 
    \limsup_{n\to\infty}
    \alpha\sind{1}
    \alpha\sind{2}
    \frac{\Lip(v)^2}{\tilde\eps_n^{d}}
    \int_{\M\sind{1}}
    \int_{\M\sind{2}}
    \eta_{\tilde\eps_n}\left(\abs{x-y}\right)
    \d\mu\sind{2}(x)
    \d\mu\sind{1}(y)
    \\
    &\leq 
    \limsup_{n\to\infty}
    \frac{C\Lip(v)^2}{\tilde\eps_n^{d}}
    \tilde\eps_n^{d-d\sind{12}}\tilde\eps_n^d = 0,
\end{align*}
where $C>0$ is a constant and we used that $d-d\sind{12}=1$.

\paragraph{Case 4, \texorpdfstring{$d\sind{1}-d\sind{12}<2$}{d1-d12<2} and \texorpdfstring{$d\sind{2}-d\sind{12}<2$}{d2-d12<2}}

We have the standing assumption that $d\sind{12}<d\sind{1}$ which in this case is only possible if $d\sind{12}=0$ and $d\sind{1}=1$.
Then, however, we also have that $d\sind{2}=1=d\sind{1}$ which takes us back to Case 3.
    
\end{proof}

\section*{Acknowledgments}

The authors would like to thank Simone
Di Marino for enlightening discussions, Tim Laux for pointing them to the Cheeger--Colding segment inequality, and Anton Ullrich for careful reading of the manuscript and valuable comments. 
This work was initiated while the authors were visiting the Simons Institute for the Theory of Computing to participate in the program “Geometric Methods in Optimization and Sampling” during the
Fall of 2021, and they are very grateful for the hospitality of the institute.
Part of this work was done while LB was affiliated with the Hausdorff Center for Mathematics at the University of Bonn. 
LB acknowledges funding by the German Ministry of Science and Technology (BMBF) under grant agreement No. 16IS24072 (COMFORT) and by the Deutsche Forschungsgemeinschaft (DFG, German Research Foundation) – project number 544579844 (GeoMAR).  
DS is grateful to NSF for support via NSF RTG grant 2342349 and NSF grant DMS-2511684. The authors are grateful to CMU's Center for Nonlinear Analysis for its support. 

\section*{Declarations}

\textbf{Data Availability.} No datasets were generated during the current study.

\noindent \textbf{Conflicts of Interest.} The authors declare that they have no competing interests.
 
\bibliography{bibliography}

@article {jost94,
    AUTHOR = {Jost, J\"urgen},
     TITLE = {Equilibrium maps between metric spaces},
   JOURNAL = {Calc. Var. Partial Differential Equations},
  FJOURNAL = {Calculus of Variations and Partial Differential Equations},
    VOLUME = {2},
      YEAR = {1994},
    NUMBER = {2},
     PAGES = {173--204},
      ISSN = {0944-2669,1432-0835},
   MRCLASS = {58E20 (58E99)},
  MRNUMBER = {1385525},
MRREVIEWER = {Yuan-Jen\ Chiang},
       DOI = {10.1007/BF01191341},
       URL = {https://doi.org/10.1007/BF01191341},
}

@article {jost97,
    AUTHOR = {Jost, J\"urgen},
     TITLE = {Generalized {D}irichlet forms and harmonic maps},
   JOURNAL = {Calc. Var. Partial Differential Equations},
  FJOURNAL = {Calculus of Variations and Partial Differential Equations},
    VOLUME = {5},
      YEAR = {1997},
    NUMBER = {1},
     PAGES = {1--19},
      ISSN = {0944-2669,1432-0835},
   MRCLASS = {31C25 (58E20 60H07 31D05)},
  MRNUMBER = {1424345},
       DOI = {10.1007/s005260050046},
       URL = {https://doi.org/10.1007/s005260050046},
}

@incollection {jost98,
    AUTHOR = {Jost, J\"urgen},
     TITLE = {Nonlinear {D}irichlet forms},
 BOOKTITLE = {New directions in {D}irichlet forms},
    SERIES = {AMS/IP Stud. Adv. Math.},
    VOLUME = {8},
     PAGES = {1--47},
 PUBLISHER = {Amer. Math. Soc., Providence, RI; International Press, Cambridge, MA},
      YEAR = {1998},
   MRCLASS = {31C25 (58E20)},
  MRNUMBER = {1655206},
}

@book {jost_book,
    AUTHOR = {Jost, J\"urgen},
     TITLE = {Riemannian geometry and geometric analysis},
    SERIES = {Universitext},
   EDITION = {Seventh},
 PUBLISHER = {Springer, Cham},
      YEAR = {2017},
      ISBN = {978-3-319-61859-3; 978-3-319-61860-9},
       DOI = {10.1007/978-3-319-61860-9},
       URL = {https://doi.org/10.1007/978-3-319-61860-9},
}

@misc{aamari2024stratificationlearning,
      title={A theory of stratification learning}, 
      author={Eddie Aamari and Clément Berenfeld},
      year={2024},
      eprint={2405.20066},
      archivePrefix={arXiv},
      primaryClass={math.ST},
      url={https://arxiv.org/abs/2405.20066}, 
}

@book {boucheronLugosiMassart2013,
    AUTHOR = {Boucheron, St\'ephane and Lugosi, G\'abor and Massart, Pascal},
     TITLE = {Concentration inequalities: {A} nonasymptotic theory of
              independence},
 PUBLISHER = {Oxford University Press, Oxford},
      YEAR = {2013},
      ISBN = {978-0-19-953525-5},
   MRCLASS = {60-02 (60E15 60F10 62G05)},
  MRNUMBER = {3185193},
       DOI = {10.1093/acprof:oso/9780199535255.001.0001},
       URL = {https://doi.org/10.1093/acprof:oso/9780199535255.001.0001},
}

@article {beurling59,
    AUTHOR = {Beurling, A. and Deny, J.},
     TITLE = {Dirichlet spaces},
   JOURNAL = {Proc. Nat. Acad. Sci. U.S.A.},
  FJOURNAL = {Proceedings of the National Academy of Sciences of the United
              States of America},
    VOLUME = {45},
      YEAR = {1959},
     PAGES = {208--215},
      ISSN = {0027-8424},
   MRCLASS = {31.00},
  MRNUMBER = {106365},
MRREVIEWER = {G.\ A.\ Hunt},
       DOI = {10.1073/pnas.45.2.208},
       URL = {https://doi.org/10.1073/pnas.45.2.208},
}

@article {ambrosio13density,
    AUTHOR = {Ambrosio, Luigi and Gigli, Nicola and Savar\'e, Giuseppe},
     TITLE = {Density of {L}ipschitz functions and equivalence of weak
              gradients in metric measure spaces},
   JOURNAL = {Rev. Mat. Iberoam.},
  FJOURNAL = {Revista Matem\'atica Iberoamericana},
    VOLUME = {29},
      YEAR = {2013},
    NUMBER = {3},
     PAGES = {969--996},
      ISSN = {0213-2230,2235-0616},
   MRCLASS = {46E35},
  MRNUMBER = {3090143},
       DOI = {10.4171/RMI/746},
       URL = {https://doi.org/10.4171/RMI/746},
}

@article {ambrosio14heatflow,
    AUTHOR = {Ambrosio, Luigi and Gigli, Nicola and Savar\'e, Giuseppe},
     TITLE = {Calculus and heat flow in metric measure spaces and
              applications to spaces with {R}icci bounds from below},
   JOURNAL = {Invent. Math.},
  FJOURNAL = {Inventiones Mathematicae},
    VOLUME = {195},
      YEAR = {2014},
    NUMBER = {2},
     PAGES = {289--391},
      ISSN = {0020-9910,1432-1297},
   MRCLASS = {53C23 (31E05 35F21 58J35)},
  MRNUMBER = {3152751},
MRREVIEWER = {Nelia\ Charalambous},
       DOI = {10.1007/s00222-013-0456-1},
       URL = {https://doi.org/10.1007/s00222-013-0456-1},
}

@article {cheeger99,
    AUTHOR = {Cheeger, J.},
     TITLE = {Differentiability of {L}ipschitz functions on metric measure
              spaces},
   JOURNAL = {Geom. Funct. Anal.},
  FJOURNAL = {Geometric and Functional Analysis},
    VOLUME = {9},
      YEAR = {1999},
    NUMBER = {3},
     PAGES = {428--517},
      ISSN = {1016-443X,1420-8970},
   MRCLASS = {53C23 (49J52)},
  MRNUMBER = {1708448},
MRREVIEWER = {William\ P.\ Minicozzi, II},
       DOI = {10.1007/s000390050094},
       URL = {https://doi.org/10.1007/s000390050094},
}

@book {dalmaso93,
    AUTHOR = {Dal Maso, Gianni},
     TITLE = {An introduction to {$\Gamma$}-convergence},
    SERIES = {Progress in Nonlinear Differential Equations and their
              Applications},
    VOLUME = {8},
 PUBLISHER = {Birkh\"auser Boston, Inc., Boston, MA},
      YEAR = {1993},
     PAGES = {xiv+340},
      ISBN = {0-8176-3679-X},
   MRCLASS = {49-02 (46N10 47H99 47N10 49J45 73B27)},
  MRNUMBER = {1201152},
MRREVIEWER = {T.\ Zolezzi},
       DOI = {10.1007/978-1-4612-0327-8},
       URL = {https://doi.org/10.1007/978-1-4612-0327-8},
}

@misc{ullrich2024,
      title={The heat flow on glued spaces with varying dimension}, 
      author={Anton Ullrich},
      year={2024},
      eprint={2406.09996},
      archivePrefix={arXiv},
      primaryClass={math.AP},
      url={https://arxiv.org/abs/2406.09996}, 
}

@ARTICLE{carter10,
  author={Carter, Kevin M. and Raich, Raviv and Hero III, Alfred O.},
  journal={IEEE Transactions on Signal Processing}, 
  title={On Local Intrinsic Dimension Estimation and Its Applications}, 
  year={2010},
  volume={58},
  number={2},
  pages={650-663},
  doi={10.1109/TSP.2009.2031722},
  url={https://doi.org/10.1109/tsp.2009.2031722}
}

@article{allegra19,
	author = {Allegra, Michele and Facco, Elena and Denti, Francesco and Laio, Alessandro and Mira, Antonietta},
	title = {Data segmentation based on the local intrinsic dimension},
	da = {2020/10/05},
	doi = {10.1038/s41598-020-72222-0},
	journal = {Scientific Reports},
	number = {1},
	pages = {16449},
	url = {https://doi.org/10.1038/s41598-020-72222-0},
	volume = {10},
	year = {2020}
}

@article{GTS18,
   AUTHOR = {Garc\'ia Trillos, Nicol\'{a}s and Slep\v{c}ev, Dejan},
     TITLE = {A variational approach to the consistency of spectral
              clustering},
   JOURNAL = {Appl. Comput. Harmon. Anal.},
  FJOURNAL = {Applied and Computational Harmonic Analysis. Time-Frequency
              and Time-Scale Analysis, Wavelets, Numerical Algorithms, and
              Applications},
    VOLUME = {45},
      YEAR = {2018},
    NUMBER = {2},
     PAGES = {239--281},
      ISSN = {1063-5203},
   MRCLASS = {49J55 (49J45 60D05 62G20 68R10)},
  MRNUMBER = {3824112},
       DOI = {10.1016/j.acha.2016.09.003},
       URL = {https://doi.org/10.1016/j.acha.2016.09.003}
}

@article{garcia2020error,
  title={Error estimates for spectral convergence of the graph {L}aplacian on random geometric graphs toward the {L}aplace--{B}eltrami operator},
  author={Garc{\'i}a Trillos, Nicol{\'a}s and Gerlach, Moritz and Hein, Matthias and Slep{\v{c}}ev, Dejan},
  journal={Foundations of Computational Mathematics},
  volume={20},
  number={4},
  pages={827--887},
  year={2020},
  url={https://doi.org/10.1007/s10208-019-09436-w},
  publisher={Springer}
}

@article {laux2025large,
    AUTHOR = {Laux, Tim and Lelmi, Jona},
     TITLE = {Large data limit of the {MBO} scheme for data clustering:
              {$\Gamma $}-convergence of the thresholding energies},
   JOURNAL = {Appl. Comput. Harmon. Anal.},
  FJOURNAL = {Applied and Computational Harmonic Analysis. Time-Frequency
              and Time-Scale Analysis, Wavelets, Numerical Algorithms, and
              Applications},
    VOLUME = {79},
      YEAR = {2025},
     PAGES = {Paper No. 101800, 34},
      ISSN = {1063-5203,1096-603X},
   MRCLASS = {49J45 (35A15 35R01 35R02 49Q20 53Z50)},
  MRNUMBER = {4950988},
       DOI = {10.1016/j.acha.2025.101800},
       URL = {https://doi.org/10.1016/j.acha.2025.101800},
}

@article{laux2023large,
  title={Large data limit of the MBO scheme for data clustering: convergence of the dynamics},
  author={Laux, Tim and Lelmi, Jona},
  journal={Journal of Machine Learning Research},
  volume={24},
  number={344},
  pages={1--49},
  year={2023},
  url= {http://jmlr.org/papers/v24/22-1089.html}
}

@article{trillos2023large,
  title={Large sample spectral analysis of graph-based multi-manifold clustering},
  author={García Trillos, Nicolás and He, Pengfei and Li, Chenghui},
  journal={Journal of Machine Learning Research},
  volume={24},
  number={143},
  pages={1--71},
  year={2023},
  url = {http://jmlr.org/papers/v24/21-1254.html}
}

@article{calder2018game,
  title={The game theoretic $p$-{L}aplacian and semi-supervised learning with few labels},
  author={Calder, Jeff},
  journal={Nonlinearity},
  volume={32},
  number={1},
  pages={301},
  year={2018},
  url={https://doi.org/10.1088/1361-6544/aae949 },
  publisher={IOP Publishing}
}

@article{garcia2016continuum,
  title={Continuum limit of total variation on point clouds},
  author={Garc{\'i}a Trillos, Nicol{\'a}s and Slep{\v{c}}ev, Dejan},
  journal={Archive for Rational Mechanics and Analysis},
  volume={220},
  number={1},
  pages={193--241},
  year={2016},
  issn={0003-9527},
  doi={10.1007/s00205-015-0929-z},
  url={https://doi.org/10.1007/s00205-015-0929-z},
  publisher={Springer}
}

@inproceedings{brown2023verifying,
title={Verifying the Union of Manifolds Hypothesis for Image Data},
author={Bradley CA Brown and Anthony L. Caterini and Brendan Leigh Ross and Jesse C Cresswell and Gabriel Loaiza-Ganem},
booktitle={The Eleventh International Conference on Learning Representations },
year={2023},
url={https://openreview.net/forum?id=Rvee9CAX4fi}
}

@article{garcia2019variational,
  title={Variational limits of $k$-NN graph-based functionals on data clouds},
  author={{Garc\'ia Trillos}, Nicol\'as},
  journal={SIAM Journal on Mathematics of Data Sci.},
  volume={1},
  number={1},
  pages={93--120},
  year={2019},
  publisher={SIAM},
  url={https://doi.org/10.1137/18m1188999}
}

@article{cheeger1996lower,
  title={Lower bounds on Ricci curvature and the almost rigidity of warped products},
  author={Cheeger, Jeff and Colding, Tobias H},
  journal={Annals of Mathematics},
  volume={144},
  number={1},
  pages={189--237},
  year={1996},
  url={https://doi.org/10.2307/2118589 },
  publisher={JSTOR}
}

@book{villani2008optimal,
  title={Optimal transport: old and new},
  author={Villani, C{\'e}dric},
  volume={338},
  year={2008},
  publisher={Springer}
}

@article{trillos2015rate,
  title={On the rate of convergence of empirical measures in $\infty$-transportation distance},
  author={Garc\'ia Trillos, Nicol{\'a}s and Slep{\v{c}}ev, Dejan},
  journal={Canadian Journal of Mathematics},
  volume={67},
  number={6},
  pages={1358--1383},
  year={2015},
  publisher={Cambridge University Press},
  DOI={10.4153/CJM-2014-044-6},
  url={https://doi.org/10.4153/CJM-2014-044-6},
}

@article{calder2023rates,
  title={Rates of convergence for {L}aplacian semi-supervised learning with low labeling rates},
  author={Calder, Jeff and Slep{\v{c}}ev, Dejan and Thorpe, Matthew},
  journal={Research in the Mathematical Sciences},
  volume={10},
  number={1},
  pages={10},
  year={2023},
  url={https://doi.org/10.1007/s40687-022-00371-x},
  publisher={Springer}
}

@article {Caroccia20,
    AUTHOR = {Caroccia, Marco and Chambolle, Antonin and Slep\v cev, Dejan},
     TITLE = {Mumford-{S}hah functionals on graphs and their asymptotics},
   JOURNAL = {Nonlinearity},
  FJOURNAL = {Nonlinearity},
    VOLUME = {33},
      YEAR = {2020},
    NUMBER = {8},
     PAGES = {3846--3888},
      ISSN = {0951-7715,1361-6544},
   MRCLASS = {49J55 (05C90 49J45 49Q22 62G20 65N12)},
  MRNUMBER = {4115077},
MRREVIEWER = {Koji\ Kikuchi},
       DOI = {10.1088/1361-6544/ab81ee},
       URL = {https://doi.org/10.1088/1361-6544/ab81ee},
}

@article{slepcev2019analysis,
  title={Analysis of $p$-{L}aplacian regularization in semisupervised learning},
  author={Slep\v{c}ev, Dejan and Thorpe, Matthew},
  journal={SIAM Journal on Mathematical Analysis},
  volume={51},
  number={3},
  pages={2085--2120},
  year={2019},
  publisher={SIAM},
  url={https://doi.org/10.1137/17M115222X}
}

@misc{bungert2024poisson,
      title={Convergence rates for {P}oisson learning to a {P}oisson equation with measure data}, 
      author={Leon Bungert and Jeff Calder and Max Mihailescu and Kodjo Houssou and Amber Yuan},
      year={2024},
      eprint={2407.06783},
      archivePrefix={arXiv},
      primaryClass={math.AP},
      url={https://arxiv.org/abs/2407.06783}, 
}

@article{bungert2024ratio,
  title={Ratio convergence rates for Euclidean first-passage percolation: applications to the graph infinity {L}aplacian},
  author={Bungert, Leon and Calder, Jeff and Roith, Tim},
  journal={The Annals of Applied Probability},
  volume={34},
  number={4},
  pages={3870--3910},
  year={2024},
  publisher={Institute of Mathematical Statistics},
  url={https://doi.org/10.1214/24-AAP2052}
}

@article{wu2022strong,
  title={Strong uniform consistency with rates for kernel density estimators with general kernels on manifolds},
  author={Wu, Hau-Tieng and Wu, Nan},
  journal={Information and Inference: A Journal of the IMA},
  volume={11},
  number={2},
  pages={781--799},
  year={2022},
  url={https://doi.org/10.1093/imaiai/iaab014 },
  publisher={Oxford University Press}
}

@book {evans,
    AUTHOR = {Evans, Lawrence C.},
     TITLE = {Partial differential equations},
    SERIES = {Graduate Studies in Mathematics},
    VOLUME = {19},
   EDITION = {Second},
 PUBLISHER = {American Mathematical Society, Providence, RI},
      YEAR = {2010},
     PAGES = {xxii+749},
      ISBN = {978-0-8218-4974-3},
   MRCLASS = {35-01},
  MRNUMBER = {2597943},
MRREVIEWER = {Diego\ M.\ Maldonado},
       DOI = {10.1090/gsm/019},
       URL = {https://doi.org/10.1090/gsm/019},
}

@article {KorevaarSchoen93,
    AUTHOR = {Korevaar, Nicholas J. and Schoen, Richard M.},
     TITLE = {Sobolev spaces and harmonic maps for metric space targets},
   JOURNAL = {Comm. Anal. Geom.},
  FJOURNAL = {Communications in Analysis and Geometry},
    VOLUME = {1},
      YEAR = {1993},
    NUMBER = {3-4},
     PAGES = {561--659},
      ISSN = {1019-8385,1944-9992},
   MRCLASS = {58E20 (46E35)},
  MRNUMBER = {1266480},
MRREVIEWER = {Guojun\ Liao},
       DOI = {10.4310/CAG.1993.v1.n4.a4},
       URL = {https://doi.org/10.4310/CAG.1993.v1.n4.a4},
}

@book {Fukushima,
    AUTHOR = {Fukushima, Masatoshi and Oshima, Yoichi and Takeda, Masayoshi},
     TITLE = {Dirichlet forms and symmetric {M}arkov processes},
    SERIES = {De Gruyter Studies in Mathematics},
    VOLUME = {19},
   EDITION = {extended},
 PUBLISHER = {Walter de Gruyter \& Co., Berlin},
      YEAR = {2011},
     PAGES = {x+489},
      ISBN = {978-3-11-021808-4},
   MRCLASS = {60J25 (28A12 31C45 60F10 60J40 60J45 60J55)},
  MRNUMBER = {2778606},
}

@article {AlonsoBaudoin23,
    AUTHOR = {Alonso Ruiz, Patricia and Baudoin, Fabrice},
     TITLE = {Dirichlet forms on metric measure spaces as {M}osco limits of
              {K}orevaar-{S}choen energies},
   JOURNAL = {Ann. Sc. Norm. Super. Pisa Cl. Sci. (5)},
  FJOURNAL = {Annali della Scuola Normale Superiore di Pisa. Classe di
              Scienze. Serie V},
    VOLUME = {26},
      YEAR = {2025},
    NUMBER = {3},
     PAGES = {1539--1573},
      ISSN = {0391-173X,2036-2145},
   MRCLASS = {31C25 (28A80 31E05 60J45 60J60)},
  MRNUMBER = {4991578},
       DOI = {10.2422/2036-2145.202302\_013},
       URL = {https://doi.org/10.2422/2036-2145.202302_013},
}

@article {KumagaiSturm05,
    AUTHOR = {Kumagai, Takashi and Sturm, Karl-Theodor},
     TITLE = {Construction of diffusion processes on fractals, {$d$}-sets,
              and general metric measure spaces},
   JOURNAL = {J. Math. Kyoto Univ.},
  FJOURNAL = {Journal of Mathematics of Kyoto University},
    VOLUME = {45},
      YEAR = {2005},
    NUMBER = {2},
     PAGES = {307--327},
      ISSN = {0023-608X},
   MRCLASS = {60J60 (28A80 31C25 49Q20 60J35 60J45)},
  MRNUMBER = {2161694},
MRREVIEWER = {Ren\'e\ L.\ Schilling},
       DOI = {10.1215/kjm/1250281992},
       URL = {https://doi.org/10.1215/kjm/1250281992},
}

@article {Sturm98how,
    AUTHOR = {Sturm, Karl-Theodor},
     TITLE = {How to construct diffusion processes on metric spaces},
   JOURNAL = {Potential Anal.},
  FJOURNAL = {Potential Analysis. An International Journal Devoted to the
              Interactions between Potential Theory, Probability Theory,
              Geometry and Functional Analysis},
    VOLUME = {8},
      YEAR = {1998},
    NUMBER = {2},
     PAGES = {149--161},
      ISSN = {0926-2601,1572-929X},
   MRCLASS = {60J60 (31C25 58G32)},
  MRNUMBER = {1618442},
MRREVIEWER = {Zhongmin\ Qian},
       DOI = {10.1023/A:1008667129215},
       URL = {https://doi.org/10.1023/A:1008667129215},
}

@article {Sturm98AP,
    AUTHOR = {Sturm, Karl-Theodor},
     TITLE = {Diffusion processes and heat kernels on metric spaces},
   JOURNAL = {Ann. Probab.},
  FJOURNAL = {The Annals of Probability},
    VOLUME = {26},
      YEAR = {1998},
    NUMBER = {1},
     PAGES = {1--55},
      ISSN = {0091-1798,2168-894X},
   MRCLASS = {31C25 (58G32 60G07 60J60)},
  MRNUMBER = {1617040},
MRREVIEWER = {Zhen-Qing\ Chen},
       DOI = {10.1214/aop/1022855410},
       URL = {https://doi.org/10.1214/aop/1022855410},
}

@article {WormellReich21,
    AUTHOR = {Wormell, Caroline L. and Reich, Sebastian},
     TITLE = {Spectral convergence of diffusion maps: improved error bounds
              and an alternative normalization},
   JOURNAL = {SIAM J. Numer. Anal.},
  FJOURNAL = {SIAM Journal on Numerical Analysis},
    VOLUME = {59},
      YEAR = {2021},
    NUMBER = {3},
     PAGES = {1687--1734},
      ISSN = {0036-1429,1095-7170},
   MRCLASS = {60J60 (35P15 62M05 65D99)},
  MRNUMBER = {4273695},
       DOI = {10.1137/20M1344093},
       URL = {https://doi.org/10.1137/20M1344093},
}

@article {ChengWu22,
    AUTHOR = {Cheng, Xiuyuan and Wu, Nan},
     TITLE = {Eigen-convergence of {G}aussian kernelized graph {L}aplacian
              by manifold heat interpolation},
   JOURNAL = {Appl. Comput. Harmon. Anal.},
  FJOURNAL = {Applied and Computational Harmonic Analysis. Time-Frequency
              and Time-Scale Analysis, Wavelets, Numerical Algorithms, and
              Applications},
    VOLUME = {61},
      YEAR = {2022},
     PAGES = {132--190},
      ISSN = {1063-5203,1096-603X},
   MRCLASS = {62M45 (35P05 65C20)},
  MRNUMBER = {4452681},
       DOI = {10.1016/j.acha.2022.06.003},
       URL = {https://doi.org/10.1016/j.acha.2022.06.003},
}

@article {CalderGTLewicka22,
    AUTHOR = {Calder, Jeff and Garc\'ia Trillos, Nicol\'as and Lewicka,
              Marta},
     TITLE = {Lipschitz regularity of graph {L}aplacians on random data
              clouds},
   JOURNAL = {SIAM J. Math. Anal.},
  FJOURNAL = {SIAM Journal on Mathematical Analysis},
    VOLUME = {54},
      YEAR = {2022},
    NUMBER = {1},
     PAGES = {1169--1222},
      ISSN = {0036-1410,1095-7154},
   MRCLASS = {35J15 (35R02 65N06 68T05)},
  MRNUMBER = {4384039},
       DOI = {10.1137/20M1356610},
       URL = {https://doi.org/10.1137/20M1356610},
}

@article {CalderGT22,
    AUTHOR = {Calder, Jeff and Garc\'ia Trillos, Nicol\'as},
     TITLE = {Improved spectral convergence rates for graph {L}aplacians on
              {$\varepsilon$}-graphs and {$k$}-{NN} graphs},
   JOURNAL = {Appl. Comput. Harmon. Anal.},
  FJOURNAL = {Applied and Computational Harmonic Analysis. Time-Frequency
              and Time-Scale Analysis, Wavelets, Numerical Algorithms, and
              Applications},
    VOLUME = {60},
      YEAR = {2022},
     PAGES = {123--175},
      ISSN = {1063-5203,1096-603X},
   MRCLASS = {62R30 (05C50 60D05)},
  MRNUMBER = {4393800},
       DOI = {10.1016/j.acha.2022.02.004},
       URL = {https://doi.org/10.1016/j.acha.2022.02.004},
}

@article{ArmstrongVenkatraman25optimal,
  title={Optimal convergence rates for the spectrum of the graph {L}aplacian on {P}oisson point clouds},
  author={Armstrong, Scott and Venkatraman, Raghavendra},
  journal={Foundations of Computational Mathematics},
  pages={1--26},
  year={2025},
  url = {https://doi.org/10.1007/s10208-025-09696-9},
  publisher={Springer}
}

@misc{GTLi2025Venkatraman25minimax,
      title={Minimax Rates for the Estimation of Eigenpairs of Weighted {L}aplace--{B}eltrami Operators on Manifolds}, 
      author={Nicolás {García Trillos} and Chenghui Li and Raghavendra Venkatraman},
      year={2025},
      eprint={2506.00171},
      archivePrefix={arXiv},
      primaryClass={stat.ML},
      url={https://arxiv.org/abs/2506.00171}, 
}

@article{hein2007graph,
  title={Graph {L}aplacians and their convergence on random neighborhood graphs},
  author={Hein, Matthias and Audibert, Jean-Yves and Luxburg, Ulrike von},
  journal={Journal of Machine Learning Research},
  volume={8},
  number={6},
  year={2007},
  url={https://www.jmlr.org/papers/volume8/hein07a/hein07a.pdf},
}

@Article{vonLuxburg2007Tutorial,
author={von Luxburg, Ulrike},
title={A tutorial on spectral clustering},
journal={Statistics and Computing},
year={2007},
volume={17},
number={4},
pages={395-416},
doi={10.1007/s11222-007-9033-z},
url={https://doi.org/10.1007/s11222-007-9033-z}
}

@article{belkin2007convergence,
  title={Convergence of {L}aplacian eigenmaps},
  author={Belkin, Mikhail and Niyogi, Partha},
  journal={Advances in Neural Information Processing Systems},
  volume={19},
  pages={129},
  year={2007},
  publisher={MIT; 1998},
  url={https://doi.org/10.7551/mitpress/7503.003.0021}
}

@article{vLBeBo08,
	AUTHOR = {von Luxburg, Ulrike and Belkin, Mikhail and Bousquet, Olivier},
	TITLE = {Consistency of spectral clustering},
	JOURNAL = {Ann. Statist.},
	FJOURNAL = {The Annals of Statistics},
	VOLUME = {36},
	YEAR = {2008},
	NUMBER = {2},
	PAGES = {555--586},
	ISSN = {0090-5364},
	CODEN = {ASTSC7},
	MRCLASS = {62G20 (05C50)},
	MRNUMBER = {2396807 (2009d:62065)},
	MRREVIEWER = {Ricardo Fraiman},
	DOI = {10.1214/009053607000000640},
	URL = {https://doi.org/10.1214/009053607000000640},
}

@article{singer2006graph,
	title={From graph to manifold {L}aplacian: The convergence rate},
	author={Singer, Amit},
	journal={Applied and Computational Harmonic Analysis},
	volume={21},
	number={1},
	pages={128--134},
	year={2006},
    url={https://doi.org/10.1016/j.acha.2006.03.004 },
	publisher={Elsevier}
}

@article{SinWu13,
	author = {Singer, Amit and Wu, Hau-Tieng},
	title = {Spectral convergence of the connection {L}aplacian from random samples},
	journal = {Information and Inference: A Journal of the IMA},
	volume = {6},
	number = {1},
	pages = {58-123},
    url={https://doi.org/10.1093/imaiai/iaw016 },
	year = {2017}
}

@article{belkin2003laplacian,
  title={{L}aplacian eigenmaps for dimensionality reduction and data representation},
  author={Belkin, Mikhail and Niyogi, Partha},
  journal={Neural computation},
  volume={15},
  number={6},
  pages={1373--1396},
  year={2003},
  publisher={MIT Press}, 
  url= {https://dl.acm.org/doi/10.1162/089976603321780317}
}

@article{belkin2001laplacian,
  title={{L}aplacian eigenmaps and spectral techniques for embedding and clustering},
  author={Belkin, Mikhail and Niyogi, Partha},
  journal={Advances in Neural Information Processing Systems},
  volume={14},
  year={2001},
  url={https://doi.org/10.7551/mitpress/1120.003.0080}
}

@article{ng2001spectral,
  title={On spectral clustering: Analysis and an algorithm},
  author={Ng, Andrew and Jordan, Michael and Weiss, Yair},
  journal={Advances in Neural Information Processing Systems},
  volume={14},
  url ={https://papers.nips.cc/paper_files/paper/2001/file/801272ee79cfde7fa5960571fee36b9b-Paper.pdf},
  year={2001}
}

@article {CoifmanLafon05DM,
    AUTHOR = {Coifman, Ronald R. and Lafon, St\'ephane},
     TITLE = {Diffusion maps},
   JOURNAL = {Appl. Comput. Harmon. Anal.},
  FJOURNAL = {Applied and Computational Harmonic Analysis. Time-Frequency
              and Time-Scale Analysis, Wavelets, Numerical Algorithms, and
              Applications},
    VOLUME = {21},
      YEAR = {2006},
    NUMBER = {1},
     PAGES = {5--30},
      ISSN = {1063-5203,1096-603X},
   MRCLASS = {60J70 (60J60 68P30)},
  MRNUMBER = {2238665},
MRREVIEWER = {G\"otz\ Kersting},
       DOI = {10.1016/j.acha.2006.04.006},
       URL = {https://doi.org/10.1016/j.acha.2006.04.006},
}

@article{bruna2013spectral,
  title={Spectral networks and locally connected networks on graphs},
  author={Bruna, Joan and Zaremba, Wojciech and Szlam, Arthur and LeCun, Yann},
  journal={arXiv preprint arXiv:1312.6203},
    url={https://arxiv.org/abs/1312.6203}, 
  year={2013}
}

@article{defferrard2016convolutional,
  title={Convolutional neural networks on graphs with fast localized spectral filtering},
  author={Defferrard, Micha{\"e}l and Bresson, Xavier and Vandergheynst, Pierre},
  journal={Advances in Neural Information Processing Systems},
  volume={29},
  year={2016},
  url={https://doi.org/10.48550/arXiv.1606.09375}
}

@article{zhou20gnn,
title = {Graph neural networks: A review of methods and applications},
journal = {AI Open},
volume = {1},
pages = {57-81},
year = {2020},
issn = {2666-6510},
doi = {https://doi.org/10.1016/j.aiopen.2021.01.001},
url = {https://www.sciencedirect.com/science/article/pii/S2666651021000012},
author = {Jie Zhou and Ganqu Cui and Shengding Hu and Zhengyan Zhang and Cheng Yang and Zhiyuan Liu and Lifeng Wang and Changcheng Li and Maosong Sun}
}

@inproceedings{wang2010multi,
  title={Multi-manifold clustering},
  author={Wang, Yong and Jiang, Yuan and Wu, Yi and Zhou, Zhi-Hua},
  booktitle={Pacific Rim International Conference on Artificial Intelligence},
  pages={280--291},
  year={2010},
  organization={Springer},
  url={https://link.springer.com/chapter/10.1007/978-3-642-15246-7_27}
}

@InProceedings{goldberg2009multi,
  title = 	 {Multi-Manifold Semi-Supervised Learning},
  author = 	 {Goldberg, Andrew and Zhu, Xiaojin and Singh, Aarti and Xu, Zhiting and Nowak, Robert},
  booktitle = 	 {Proceedings of the Twelfth International Conference on Artificial Intelligence and Statistics},
  pages = 	 {169--176},
  year = 	 {2009},
  series = 	 {Proceedings of Machine Learning Research},
  url = 	 {https://proceedings.mlr.press/v5/goldberg09a.html}
}

@inproceedings{wang2015multi,
  title={Multi-manifold modeling in non-Euclidean spaces},
  author={Wang, Xu and Slavakis, Konstantinos and Lerman, Gilad},
  booktitle={Artificial Intelligence and Statistics},
  pages={1023--1032},
  year={2015},
  organization={PMLR},
  url={https://proceedings.mlr.press/v38/wang15b.html},
}

@article{wang2011spectral,
  title={Spectral clustering on multiple manifolds},
  author={Wang, Yong and Jiang, Yuan and Wu, Yi and Zhou, Zhi-Hua},
  journal={IEEE Transactions on Neural Networks},
  volume={22},
  number={7},
  pages={1149--1161},
  year={2011},
  url={https://doi.org/10.1109/tnn.2011.2147798 },
  publisher={IEEE}
}

@article{Arias-Castro_Chen_Lerman_2011, 
    title={Spectral clustering based on local linear approximations}, 
    volume={5}, 
    url={http://dx.doi.org/10.1214/11-EJS651}, 
    DOI={10.1214/11-ejs651}, 
    number={none}, 
    journal={Electronic Journal of Statistics}, 
    publisher={Institute of Mathematical Statistics}, 
    author={Arias-Castro, Ery and Chen, Guangliang and Lerman, Gilad}, 
    year={2011}, 
    month=jan 
}

@incollection{medina2019heuristic,
  title={Heuristic framework for multiscale testing of the multi-manifold hypothesis},
  author={Medina, F Patricia and Ness, Linda and Weber, Melanie and Djima, Karamatou Yacoubou},
  booktitle={Research in Data Science},
  pages={47--80},
  year={2019},
  publisher={Springer},
  url={https://link.springer.com/chapter/10.1007/978-3-030-11566-1_3}
}

@misc{chen2025robust,
      title={Robust Multi-Manifold Clustering via Simplex Paths}, 
      author={Haoyu Chen and Anna Little and Akin Narayan},
      year={2025},
      eprint={2507.10710},
      archivePrefix={arXiv},
      primaryClass={stat.ML},
      url={https://arxiv.org/abs/2507.10710}
}

@article{babaeian2015multiple,
  title={Multiple manifold clustering using curvature constrained path},
  author={Babaeian, Amir and Bayestehtashk, Alireza and Bandarabadi, Mojtaba},
  journal={PloS one},
  volume={10},
  number={9},
  pages={e0137986},
  year={2015},
  publisher={Public Library of Science San Francisco, CA USA},
  url={https://doi.org/10.1371/journal.pone.0137986}
}

@inproceedings{chen2023largest,
  title={Largest Angle Path Distance for Multi-Manifold Clustering},
  author={Chen, Haoyu and Little, Anna and Narayan, Akil},
  booktitle={2023 International Conference on Sampling Theory and Applications (SampTA)},
  pages={1--7},
  year={2023},
  url={https://doi.org/10.1109/sampta59647.2023.10301401 },
  organization={IEEE}
}

@inproceedings{gong2012robust,
  title={Robust multiple manifolds structure learning},
  author={Gong, Dian and Zhao, Xuemei and Medioni, G{\'e}rard},
  booktitle={Proceedings of the 29th International Coference on International Conference on Machine Learning},
  url={https://icml.cc/2012/papers/191.pdf},
  pages={25--32},
  year={2012}
}
 
 \begin{appendix}
\section{Appendix}
\subsection{The principal angle}
\label{sec:appendix_angle}

The positivity of the principal angle in \labelcref{eq:angle} implies the following lemma which we need for some of our proofs.
\begin{lemma}\label{lem:angle}
   If $\theta>0$, where $\theta$ is defined in \labelcref{eq:angle}, then there exist constants $C, \eps_0>0$ such that for all $0<\eps<\eps_0$ whenever $x\in\M\sind{i}$ with $\dist(x,\M\sind{12})>C\eps$, then $B(x,\eps)\cap\M\sind{j}=\emptyset$ for $i\neq j$.  
\end{lemma}
\begin{proof}
    Take $C = 2/\sin \theta$ and consider $\eps_0$ to be smaller than the injectivity radii of both manifolds.
    If the claim does not hold, then there exist sequences $x_n \in \M\sind{1}$ and $y_n \in \M\sind{2}$ such that $\dist(x_n, \M\sind{12}) \to 0$ and $\abs{x_n- y_n}<\tfrac{\sin\theta}{2}\dist(x_n, \M\sind{12}) $. 
    Let $z_n\in\M\sind{12}$ be the closest point to $x_n$ with respect to the Euclidean distance, and let $\log\sind{i}_{z_n}$ be the inverse of the exponential map from the ball of radius $\eps_0$ in $T_{\M\sind{i}}$ to $\M\sind{i}$ for $i=1,2$. 
    
    First, we note that combining the law of sines with our assumption that $\abs{x_n-y_n}\leq\tfrac{\sin\theta}{2}\abs{x_n-z_n}$, the angle $\beta_n := \angle \left(x_n - z_n, y_n - z_n\right)$ satisfies:
    \begin{align*}
        2 R_n \sin\beta_n = \abs{x_n-y_n}\leq\abs{x_n-z_n}\frac{\sin\theta}{2} \leq R_n\sin\theta,
    \end{align*}
    where $R_n$ is the radius of the circumcircle of the triangle with vertices $x_n$, $y_n$, and $z_n$.
    Hence, we obtain $\sin\beta_n \leq \tfrac{\sin\theta}{2}$ which implies $\beta_n\leq\tfrac{\theta}{2}$.
    Using this together with the definition of $\theta$ in \labelcref{eq:angle} and a simple triangle inequality for vectors we have
    \begin{align*}
        \theta 
        &\leq 
        \angle\left(\log\sind{1}_{z_n}(x_n), \log\sind{2}_{z_n}(y_n)\right) 
        \\
        &\leq  
        \angle\left(\log_{z_n}\sind{1}(x_n),x_n - z_n\right) + 
        \angle \left(x_n - z_n, y_n - z_n\right)
        + \angle\left(\log\sind{2}_{z_n}(x_n),y_n - z_n\right) 
        \\
        & \leq \angle \left(\log\sind{1}_{z_n}(x_n), x_n - z_n\right) + \frac{\theta}{2}
        + \angle\left(\log\sind{2}_{z_n}(x_n),y_n - z_n\right).
    \end{align*}
    Since for large $n$ we have $\angle\left(\log\sind{1}_{z_n}(x_n),x_n - z_n\right) < \frac{\theta}{8}$ and  $\angle\left(\log\sind{2}_{z_n}(y_n),y_n - z_n\right) < \frac{\theta}{8} $ we arrive at a contradiction.
\end{proof}

\subsection{The \texorpdfstring{$TL^p$}{TLp} convergence}
\label{sec:appendix_TLp}

To have a unified domain for functionals defined on graphs and continuum manifolds, we rely on the $TL^p$-topology, introduced in \cite{garcia2016continuum}.
There the $TL^p$-spaces for $p\in[1,\infty)$ where defined as spaces of pairs of probability measures and $L^p$-functions on a metric space $M$:
\begin{align*}
    TL^p(M) := 
    \left\lbrace
    (u,\mu) \st \mu \in \mathcal P(M),\;u\in L^p(M;\mu)
    \right\rbrace. 
\end{align*}
These spaces can be equipped with an optimal transport type metric by setting
\begin{align*}
    d_{TL^p}((u,\mu),(v,\nu)) 
    := 
     \inf_{\pi\in\Gamma(\mu,\nu)}
    \left(
    \iint_{M\times M}\big(d(x,y)^p
    +
    \abs{u(x)-v(y)}\big)^p
    \d\pi(x,y)
    \right)^\frac{1}{p},
         \end{align*}
where the infimum is taken over couplings $\pi\in\mathcal P(M\times M)$ of the two probability measures $\mu$ and $\nu$.
In \cite{garcia2016continuum} these definitions were used for $M$ being an open subset of Euclidean space.
In the case where $M=\M$ is a compact Riemannian manifold without boundary, equipped with a probability measure~$\mu$ that is absolutely continuous with respect to the volume measure on~$\M$, we refer to \cite{laux2025large,garcia2020error} for definitions and properties of $TL^p$-distances on such spaces. 
In our situation, where $\M=\M\sind{1}\cup\M\sind{2}$ is a union of manifolds embedded in $\R^N$ as defined in \cref{sec:unions}, we note that $M=\M$ is a metric space when equipped with the ambient Euclidean metric of $\R^N$.
                         
It turns out that one can characterize the convergence $(u_n,\mu_n)\to(u,\mu)$ in $TL^p(M)$ by means of so-called stagnating transport plans which are couplings $\pi_n\in\Pi(\mu,\mu_n):=\{\pi\in\mathcal P(M\times M)\st \pi(\cdot\times M)=\mu,\,\pi(M\times\cdot)=\mu_n\}$ of $\mu$ and $\mu_n$ that satisfy
\begin{align*}
    \lim_{n\to\infty}\iint_{M\times M}d(x,y)\d\pi_n(x,y) = 0.
\end{align*}
A special case of stagnating transport plans are of the form $\pi_n = (\operatorname{id}\times T_n)_\sharp\mu$ where $T_n:M\to M$ is a stagnating transport map, meaning that it satisfies $(T_n)_\sharp\mu = \mu_n$ and 
\begin{align*}
    \lim_{n\to\infty}\int_{M}d(x,T_n(x))\d\mu(x) = 0.
\end{align*}
Note, however, that such maps $T_n$ do not necessarily exist.
We have the following characterization of $TL^p$-convergence:
\begin{proposition}[{\cite[Proposition 3.12]{garcia2016continuum}}]\label{prop:char_TLp}
    Let $M$ be a Polish space, $\mu\in\mathcal P(M)$ be a Borel probability measure, and $p\in[1,\infty)$.
    Then the following statements are equivalent:
    \begin{listi}
        \item $(u_n,\mu_n)\to(u,\mu)$ in $TL^p(M)$ as $n\to\infty$;
        \item $\mu_n\wsto\mu$ weakly as measures and for every stagnating sequence of transport plans $\pi_n\in\Pi(\mu,\mu_n)$ it holds
        \begin{align}\label{eq:stagnating_cvgc}
            \lim_{n\to\infty}\iint_{M\times M}\abs{u(x)-u_n(y)}^p\d\pi_n(x,y) = 0;
        \end{align}
        \item $\mu_n\wsto\mu$ weakly as measures and there exists a stagnating sequence of transport plans $\pi_n\in\Pi(\mu,\mu_n)$ such that \labelcref{eq:stagnating_cvgc} holds.
    \end{listi}
    Moreover, if $M=\M$ is a union of Riemannian manifolds without boundary with volume form $\vol$ and $\mu\in\mathcal{P}(\M)$ is absolutely continuous with respect to $\vol$, then 1.-3. are also equivalent to 
    \begin{listi}
    \setcounter{broj}{3}
        \item $\mu_n\wsto\mu$ weakly as measures and for every stagnating sequence of transport maps $T_n:\M\to\M$ (with $(T_n)_\sharp\mu=\mu_n$) it holds
        \begin{align}\label{eq:stagnating_maps_cvgc}
            \lim_{n\to\infty}\int_{\M}\abs{u(x)-u_n(T_n(x))}^p\d\mu(x) = 0;
        \end{align}
        \item $\mu_n\wsto\mu$ weakly as measures and there exists a stagnating sequence of transport maps $T_n:\M\to\M$ (with $(T_n)_\sharp\mu=\mu_n$) such that \labelcref{eq:stagnating_maps_cvgc} holds.
    \end{listi}
\end{proposition}
\begin{remark}
    Note that (i)-(iii) in \cref{prop:char_TLp} are, in particular, applicable to the union of manifolds $\M$ equipped with the probability measure $\mu\in\mathcal P(\M)$.
\end{remark}
\begin{remark}
    Note that \cite[Proposition 3.12]{garcia2016continuum} is stated in a less general form where $M$ is an open and bounded domain in $\R^d$. 
    However, an inspection of the proof shows that \cref{prop:char_TLp} holds. 
    Indeed, the equivalence of (ii) and (iii) follows from the assumption that $\mu$ is a Borel probability measure and hence regular, which implies that bounded Lipschitz functions are dense in $L^1(M;\mu)$, cf. \cite[Lemma 3.10]{garcia2016continuum}.
    Furthermore, the equivalence of (i) and (iii) is proved as in \cite[Proposition 3.12]{garcia2016continuum}, using that convergence in the Wasserstein-$p$ distance for $p\in[1,\infty)$ metrizes weak convergence of measures in a Polish space \cite[Theorem 6.9]{villani2008optimal}.
\end{remark}
Finally, we use the same abuse of notation which was introduced in \cite[Definition 3.13]{garcia2016continuum}.
It accounts for the fact that in graph-based learning the measure part of a sequence in $TL^p$ is typically prescribed and therefore less important to deal with than the function part.
\begin{definition}[$TL^p$-convergence and compactness of functions]\label{def:TL_convergence}
    Let $\mu_n\wsto\mu$ weakly as measures as $n\to\infty$ and let $(u_n)_{n\in\N}$ with $u_n \in L^p(M,\mu_n)$ be a sequence. 
    We say that $u_n$ converges to $u$ in $TL^p(M)$ if and only if $(u_n,\mu_n)\to(u,\mu)$ in $TL^p(M)$.
    Similarly, we say that $(u_n)_{n\in\N}$ is relatively compact in $TL^p(M)$ if $(u_n,\mu_n)_{n\in\N}$ is relatively compact in $TL^p(M)$.
\end{definition}

\subsection{\texorpdfstring{$\Gamma$}{Gamma}-convergence of unnormalized Dirichlet energies}
\label{sec:appendix_Gamma}
\red We first recall the continuum $\Gamma$-convergence of nonlocal mean-value energies to Dirichlet energies in the recent form that we need. Analogous results, even in the more general setting of harmonic maps between Riemannian manifolds, are classical; see for example Jost \cite[Section~9.3]{jost_book}. \nc
\begin{theorem}[{\cite[Theorem 11]{laux2025large}}]\label{thm:Gamma-convergence_standard_NL}
Let $\M$ be a $d$-dimensional compact Riemannian manifold embedded in $\R^N$, let $\eta:[0,\infty)\to[0,\infty)$ satisfy \cref{ass:eta}, let $\rho\in C(\M)$ be a continuous function satisfying $\frac{1}{C}\leq\rho\leq C$ on $\M$ with some $C>0$, and let $(\eps_n)_{n\in\N}\subset(0,\infty)$ be a sequence converging to zero. 
Then as $n\to\infty$ the functionals
\begin{align*}
    \mathcal G_{\eps_n} &: L^2(\M) \to [0,\infty],\\
    u &\mapsto \frac{1}{\eps_n^{d+2}}\int_{\M}\int_{\M} \eta_{\eps_n}(\abs{x-y})\abs{u(x)-u(y)}^2\rho(x)\rho(y)\d\vol(x)\d\vol(y),
\end{align*}
$\Gamma$-converge in $L^2(\M)$ to
\begin{align*}
    \mathcal{G} &: L^2(\M) \to [0,\infty],\\
    u &\mapsto
    \begin{cases}
    \sigma_\eta\int_\M\abs{\grad_\M u}^2\rho^2\d\vol\quad&\text{if }u\in H^1(\M),\\
    \infty\quad&\text{else},
    \end{cases}
\end{align*}
where $\sigma_\eta := \int_{\R^d}\eta(\abs{x})\abs{x_1}^2\d x$.
Furthermore, any sequence $u_n\subset L^2(\M)$ such that 
\begin{align*}
    \sup_{n\in\N} \mathcal G_{\eps_n}(u_n)<\infty,\qquad \sup_{n\in\N} \norm{u_n}_{L^2(\M)}<\infty,
\end{align*}
is precompact in $L^2(\M)$.
\end{theorem}
\red 
We now recall the analogous result in the random discrete setting, in which the manifold is sampled and the energy is defined on the resulting graph. \nc
\begin{theorem}\label{thm:Gamma-convergence_standard}
Let $\M$ be a $d$-dimensional compact Riemannian manifold embedded in a Euclidean space. Let $\eta:[0,\infty)\to[0,\infty)$ satisfy \cref{ass:eta}, and let $\rho\in C(\M)$ be a continuous function satisfying $\frac{1}{C}\leq\rho\leq C$ on $\M$ with some $C>0$. 
Letting $\{x_1,\dots,x_n\}\subset\M$ for $n\in\N$ be an i.i.d. sample of the measure $\d\mu:=\rho\d\vol$, define $\mu_n := \frac{1}{n}\sum_{i=1}^n \delta_{x_i}$ and let 
$\eps_n$ satisfy \cref{ass:epsilon}.
  Then as $n\to\infty$ almost surely the functionals
\begin{align*}
    G_{n,\eps_n} &: TL^2(\M_n;\mu_n) \to [0,\infty],\\
    u &\mapsto \frac{1}{\eps_n^{d+2}}\int_{\M}\int_{\M} \eta_{\eps_n}(\abs{x-y})\abs{u(x)-u(y)}^2\d\mu_n(x)\d\mu_n(y),
\end{align*}
$\Gamma$-converge in $TL^2(\M)$ to
\begin{align*}
    \mathcal{G} &: L^2(\M) \to [0,\infty],\\
    u &\mapsto
    \begin{cases}
    \sigma_\eta\int_\M\abs{\grad_\M u}^2\rho^2\d\vol\quad&\text{if }u\in H^1(\M),\\
    \infty\quad&\text{else},
    \end{cases}
\end{align*}
where $\sigma_\eta := \int_{\R^d}\eta(\abs{x})\abs{x_1}^2\d x$.
Furthermore, any sequence $(u_n,\mu_n)\subset TL^2(\M)$ such that 
\begin{align*}
    \sup_{n\in\N} G_{n,\eps_n}(u_n)<\infty,\qquad \sup_{n\in\N} \norm{u_n}_{L^2(\mu_n)}<\infty,
\end{align*}
is precompact in $TL^2(\M)$.
\end{theorem}
\begin{proof}
    The theorem is essentially the one stated in \cite[Theorem 9]{laux2025large}, where it was assumed that $\rho$ is smooth and $d \geq 2$, although the proof allows for more general $\rho$ and for $d=1$.
    The proof relies on the corresponding $\Gamma$-convergence results in the Euclidean setting, see~\cite{GTS18} for $d\geq 2$ and \cite{slepcev2019analysis} for $d=1$, as well as the argument of \cite[Proposition 2.12]{CalderGT22} (or \cite[Lemmas 3.1, 3.2]{Caroccia20}) to get the sharp dependence of $\eps_n$ on $n$ for $d=2$.
  \end{proof}

We note that the theorem can be easily extended to the situation that the number of points $n$, is not deterministic, but is a random variable with controlled variance. We state and prove the corollary for $n$ which has a binomial distribution as this is the situation we face when dealing with samples from union of two manifolds. 
  \begin{corollary}\label{cor:Gamma-convergence_standard}
    In the setting of \cref{thm:Gamma-convergence_standard}, let $N\in\N$, $p\in(0,1)$, $n\sim\operatorname{Bin}(N,p)$, and $\{x_1,\dots,x_n\}\subset\M$ be an i.i.d. sample of the measure $\d\mu:=\rho\d\vol$.
    If \cref{ass:epsilon} holds,
     then as $N\to\infty$ the assertions of \cref{thm:Gamma-convergence_standard} remain true for the functional $G_{n,\eps_N}$. 
\end{corollary}
\begin{proof}
    A simple application of Hoeffding's inequality to the binomial random variable $n$ shows
    \begin{align*}
        \mathbb P\left[
        \abs{n-Np} \geq t
        \right]
        \leq 
        2\exp\left(-\frac{2t^2}{N}\right)
    \end{align*}
    and taking $t:=Np/2$ we obtain
    \begin{align*}
        \mathbb P\left[
        n \geq \frac{Np}{2}
        \right]
        \geq 
        1- 2\exp\left(-\frac{Np^2}{2}\right).
    \end{align*}
    At the same time, by \cite[Theorem 2]{garcia2020error} there exist constants $C_1,C_2>0$ such that for $m\in\N$ with high probability there exists a transport map $T_m$ pushing $\mu$ onto $\mu_m$ with infinity transport distance at most $C_1 \delta_m$
    \begin{align*}
        \mathbb P\left[
            \norm{\operatorname{id}-T_m}_{L^\infty(\M)}
            > C_1\delta_m 
        \right]
        \leq \frac{C_2}{m^2},
    \end{align*}
    where
        \begin{align*}
        \delta_n :=
        \begin{dcases}
        \sqrt{\frac{\log\log n}{n}} \quad&\text{if }d=1,
        \\
      \frac{(\log n)^{3/4}}{n^{1/2}} \quad&\text{if }d=2,
       \\
        \left(\frac{\log n}{n}\right)^\frac{1}{d} \quad&\text{if }d\geq 3.
        \end{dcases}
    \end{align*}
    With a union bound we obtain for any $m\in\N$:
    \begin{align*}
        \mathbb P\left[
            n\geq\frac{Np}{2}
            \text{ and }
            \norm{\operatorname{id}-T_m}_{L^\infty(\M)}
            \leq C_1\delta_m 
        \right]
        \geq 1- \frac{C_2}{m^2} - 2\exp\left(-\frac{Np^2}{2}\right).
    \end{align*}
    Now we can choose $m$ as the smallest integer larger or equal to $\frac{Np}{2}$ (implying that $n\geq m$ in the above probability) and use the law of total probability to obtain
    \begin{align*}
        \mathbb P\left[
            \norm{\operatorname{id}-T_n}_{L^\infty(\M)}
            \leq C_1\delta_{Np/2} 
        \right]
        &\geq
        \mathbb P\left[
            \norm{\operatorname{id}-T_n}_{L^\infty(\M)}
            \leq C_1\delta_{Np/2} 
            \,\Big\vert\,
            n\geq \frac{Np}{2}
        \right]
        \mathbb P\left[n\geq \frac{Np}{2}\right]
        \\
        &=
        \mathbb P\left[
            n\geq\frac{Np}{2}
            \text{ and }
            \norm{\operatorname{id}-T_n}_{L^\infty(\M)}
            \leq C_1\delta_{Np/2}  
        \right]
        \\
        &\geq 
        1- \frac{4C_2}{p^2N^2} - 2\exp\left(-\frac{Np^2}{2}\right).
    \end{align*}
Combining this estimate with the Borel--Cantelli lemma and using that $\limsup_{n\to\infty}\frac{\delta_N}{\delta_{Np/2}}\in(0,\infty)$ implies that almost surely we have
\begin{align*}
    \limsup_{n\to\infty}\frac{\norm{\operatorname{id}-T_n}_{L^\infty(\M)}}{\delta_N}<\infty.
\end{align*}
Since this is the only estimate on the distribution of points used in the proof of \cref{thm:Gamma-convergence_standard} and since $\left(\frac{\log N}{N}\right)^\frac{1}{d}\ll\eps_N\ll 1$ we assume is equivalent to $\left(\frac{\log (pN/2)}{pN/2}\right)^\frac{1}{d}\ll\eps_N\ll 1$, we conclude that 
the corollary holds.  
\end{proof}

\subsection{Establishing the trace condition}
\label{sec:appendix_trace}

The challenging part in proving the liminf inequality is to establish the trace condition encoded in \labelcref{eq:def_H1mu,eq:limit_fctl} in the case that the manifolds have equal dimension and the codimension of their intersection equals one.
Towards that goal, we first prove the following auxiliary lemma.
\begin{lemma}\label{lem:trace}
In the multi-manifold setting of \cref{sec:unions}, in particular, under \cref{ass:densities,ass:eta} let $d\sind{1}=d\sind{2}=:d$ and $d\sind{12}=d-1$.
Assume that $\eps_n\to 0$ and $v_n \to v$ in $L^2(\M)$ as $n\to\infty$ where $v\vert_{\M\sind{i}}\in H^1(\M\sind{i})$, and that 
\begin{align*}
    \liminf_{n\to\infty}
    \frac{1}{\eps_n^{d+2}} \int_{\M}\int_{\M} \eta_{\eps_n}(\abs{x-y}) \abs{v_n(x)-v_n(y)}^2\d\mu(x)\d\mu(y) < \infty.
\end{align*}
Then $\trace\sind{1}(v) = \trace\sind{2}(v)$ on $\M\sind{12}$.
\end{lemma}
\begin{proof}
Let us fix a point $z\in\M\sind{12}$ and a radius $r>0$.
Since $\M\sind{12}$ has codimension $1$ in $\M\sind{1}$ and $\M\sind{2}$, for $r$ small, $\M\sind{12}$ divides them locally into two pieces. Namely let $\M\sind{i}_\pm$ for $i=1,2$ be connected components of $B(z,r)  \cap \M\sind{i} \setminus \M\sind{12}$.
We define $U:=B(z,r) \cap (\M\sind{1}_+\cup\M\sind{2}_+ \cup \M\sind{12})$.
Since the intersection of the two manifolds is assumed nondegenerate and they have equal dimension, there exists open set $V \subset \R^d$ and a bi-Lipschitz map $\Psi:U\to V$, i.e., there exists $L>0$ such that
\begin{align*}
    \frac{1}{L}\abs{\Psi(x)-\Psi(y)}\leq\abs{x-y}
    \leq 
    L\abs{\Psi(x)-\Psi(y)}\quad\forall x,y\in U.
\end{align*}
Without loss of generality we can assume that $\Psi$ is such that
\begin{align*}
    \Psi(U\cap\M\sind{1}_+)\subset\{x\in\R^d\st x_1\geq 0\},
    \qquad
    \Psi(U\cap\M\sind{2}_+)\subset\{x\in\R^d\st x_1\leq 0\}
\end{align*}
and hence
\begin{align*}
    \Psi(B(z,r)\cap\M\sind{12}) \subset \{x\in\R^d \st x_1=0\}.
\end{align*}
Let $w_n := v_n \circ \Psi^{-1}$.  Using the regularity of $\Psi$ we obtain
\begin{align*}
     \frac{1}{\eps_n^{d+2}} \int_{\M}\int_{\M} & \eta_{\eps_n}(\abs{x-y}) \abs{v_n(x)-v_n(y)}^2\d\mu(x)\d\mu(y)
    \\
    &\geq 
    \frac{1}{\eps_n^{d+2}} \int_{U}\int_{U} \eta_{\eps_n}(\abs{x-y}) \abs{v_n(x)-v_n(y)}^2\d\mu(x)\d\mu(y)
    \\
    &=
    \frac{1}{\eps_n^{d+2}} \int_{V}\int_{V} \eta_{\eps_n}(\abs{\Psi^{-1}(\tilde x)-\Psi^{-1}(\tilde y)}) \abs{w_n(\tilde x)-w_n(\tilde y)}^2\d\Psi_\sharp\mu(\tilde x)\d\Psi_\sharp\mu(\tilde y)
    \\
    &\geq 
    \frac{1}{\eps_n^{d+2}} \int_{V}\int_{V} \eta_{\eps_n}(L\abs{\tilde x-\tilde y}) \abs{w_n(\tilde x)-w_n(\tilde y)}^2\d\Psi_\sharp\mu(\tilde x)\d\Psi_\sharp\mu(\tilde y)
    \\
    &\geq 
    \frac{1}{L^{d+2}}
    \frac{1}{\tilde\eps_n^{d+2}} \int_{V}\int_{V} \eta_{\tilde\eps_n}(\abs{\tilde x-\tilde y}) \abs{w_n(\tilde x)-w_n(\tilde y)}^2\d\Psi_\sharp\mu(\tilde x)\d\Psi_\sharp\mu(\tilde y)
    \\
    &\geq
    \frac{c}{L^{d+2}}
    \frac{1}{\tilde\eps_n^{d+2}} \int_{V}\int_{V} \eta_{\tilde\eps_n}(\abs{\tilde x-\tilde y}) \abs{w_n(\tilde x)-w_n(\tilde y)}^2\d\tilde x\d\tilde y
\end{align*}
where $\tilde\eps_n := \eps_n/L$ and $c>0$ depends on $L$ and the lower bound of $\rho$ (the density of $\mu$).
Using that $w_n$ (being a composition of a Lipschitz function and a convergent sequence of functions $v_n$) converges to $w:=v\circ\Psi^{-1}\in L^2(V)$ and the $\liminf$ inequality from \cref{thm:Gamma-convergence_standard} inequality in the Euclidean setting we get
\begin{align*}
    \int_V \abs{\grad w}^2\d x < \infty.
\end{align*}
Therefore, the trace of $w$ on the $d-1$-dimensional set $\mathcal S := V\cap\{x\in\R^d\st x_1=0\}$ is well-defined. Thus, the trace of $v$ on $\Psi^{-1}(\mathcal S)$ is well defined and hence
\begin{align*}
    \trace\sind{1}(v\vert_U) = \trace\sind{2}(v\vert_U).
\end{align*}
\end{proof}

\subsection{Smooth approximation with trace constraint} 
\label{sec:appendix_smooth_approx}
We show that smooth functions on a union of manifolds of equal dimension are dense in the space of $H^1$-functions which have identical traces on the intersection $\M\sind{12}$ in case the latter has codimension one.
\red
This is nontrivial since smoothing the functions separately on each of the manifolds is does not in general lead to traces which agree.\footnote{We would not be surprised if a result we need exists in the literature, but we were unable to find it.}
\nc

\begin{lemma} \label{lem:density}
Let $d\sind{1}=d\sind{2}=:d$, assume that $d\sind{12}=d-1$, and define
\begin{align*}
    \mathcal S 
    := 
    \left\{ u \in L^2(\M) \st 
    u\vert_{\M\sind{i}} \in C^\infty(\M\sind{i}) \te{ for } i=1,2 \te{ and }\trace\sind{1}(u)=\trace\sind{2}(u)\right\}.
\end{align*}  
Then $\mathcal S$ is dense in 
\begin{align*}
    H^1(\M) :=
    \left\lbrace
    u \in L^2(\M)
    \st 
    u\vert_{\M\sind{i}} \in H^1(\M\sind{i}) \text{ for }i\in\{1,2\}\text{ and }
    \trace\sind{1}(u)=\trace\sind{2}(u)
    \right\rbrace.
\end{align*}
\end{lemma}
\begin{proof}
    We let $u \in H^1(\M)$ and abbreviate $u\sind{i} := u\vert_{\M\sind{i}}$ for $i\in\{1,2\}$.

    \paragraph{Step 1, local tubular neighborhood theorem.}
    For any point $x \in \M\sind{12}$ let us consider an open neighborhood $V\subset\R^N$ of $x$ in $\R^N$.
    Let us furthermore consider an open neighborhood $O$ of $x$ in $\M\sind{12}$, compactly contained in $V$.
    We choose the neighborhood $O$ sufficiently small such that we can apply tubular neighborhood theorem to obtain $\eps>0$ and open neighborhoods $U\sind{i}$ of $x$ in $\M\sind{i}$ as well as diffeomorphisms $\Psi\sind{i}:U\sind{i}\to O \times(-\eps,\eps)$ with $\Psi\sind{i}(p,0)=p$ for all $p \in O$.
    Consequently, we can define the diffeomorphism $\Psi := (\Psi\sind{1})^{-1} \circ \Psi\sind{2} : U\sind{2}\to U\sind{1}$.
     Note that by construction the intersection $\M\sind{12}$ divides $U\sind{i}$ into two parts, $U\sind{i} \setminus \M\sind{12} =U\sind{i}_+ \cup U\sind{i}_-$ where $U\sind{i}_+ := (\Psi\sind{i})^{-1}(O\times(0,\eps))$ and $U\sind{i}_- := (\Psi\sind{i})^{-1}(O\times(-\eps,0))$. 
    Let $W \subset \R^N$ be an open neighborhood of $x$ such that $W \cap \M\sind{i} \subset U\sind{i}$ for $i=1,2$.

    \paragraph{Step 2, partition of unity.}
    By compactness of $\M\sind{12}$ we can cover $\M\sind{12}$ by finitely many such open sets $(O_k)_{k\in K}$ corresponding to points on $\M\sind{12}$, as constructed above. 
    By  $(U\sind{i}_k)_{k\in K}$, $(W_k)_{k\in K}\subset\R^N$, and $(\Psi\sind{i}_k)_{k\in K}$ for $i\in\{1,2\}$ we denote the associated open covers and diffeomorphisms constructed in Step 1.
    We let $(\xi_k)_{k\in K\cup\{\infty\}}\subset C^\infty(\R^N)$ be a smooth partition of unity subordinate to the open cover $(W_k)_{k\in K\cup\{\infty\}}$ of $\R^N$ where we let $W_\infty := \R^N\setminus\M\sind{12}$ with $\supp\xi_k\subset W_k$ for all $k\in K\cup\{\infty\}$.

    \paragraph{Step 3, local reduction to zero trace.}
    Using the partition of unity we can now write $u\sind{i} = \sum_{k\in K}\xi_k u + \xi_\infty u$ and note that $\xi_\infty u\equiv 0$ in a neighborhood of the intersection $\M\sind{12}$. 
    Hence, the last term can be easily smoothed without changing its trace values.
    So the key step is to smooth each of the functions $\xi_k u$, which have compact support in $U_k\sind{i}$.

    We now fix the index $k\in K$ and omit it in our notation of the open sets $U\sind{i}$ and the diffeomorphism $\Psi : U\sind{2}\to U\sind{1}$.
    We also use the abbreviation $\bar u\sind{i}:=\xi_k u\sind{i}$.
    We define the functions $\hat{u}\sind{2},u_0:U\sind{1}\to\R$ via $\hat{u}\sind{2}:=\bar u\sind{2}\circ\Psi^{-1}$ and $u_0 := \bar u\sind{1}-\hat{u}\sind{2}$ and note that 
    \[\trace\sind{1}(u_0)=\trace\sind{1}(\bar u\sind{1})-\trace\sind{1}(\bar u\sind{2}\circ\Psi^{-1})=\trace\sind{1}(\bar u\sind{1})-\trace\sind{2}(\bar u\sind{2})=0,\]
    since $\Psi$ equals the identity map on $\M\sind{12}$ and $\trace\sind{i}(\bar u\sind{i})=\trace\sind{i}(\xi_k u\sind{i})$ is independent of $i\in\{1,2\}$.
    We also note that, by its definition, $u_0$ has compact support in $U\sind{1}$
    We can find a smooth approximation $v\sind{1}$ of $\bar u\sind{1}$ with compact support in $U\sind{1}$, not caring about the trace values. 
    Provided we manage to find a smooth approximation $v_0\in C^\infty(\M\sind{1})$ of $u_0$ with $\trace\sind{1}(v_0)=0$, we claim that the function $v\sind{2}:=(v\sind{1} - v_0)\circ\Psi : U\sind{2}\to \R$ is a smooth approximation of $\bar u\sind{2}$ with the same trace as $v\sind{1}$.
    Indeed, since $\Psi$ is the identity on $\M\sind{12}$ and $\trace\sind{1}(v_0)=0$ it holds $\trace\sind{2}(v\sind{2})=\trace\sind{1}(v\sind{1})$.
    Furthermore, we have
    \begin{align*}
        \norm{v\sind{2}-\bar u\sind{2}}_{H^1(U\sind{2})}
        &=
        \norm{(v\sind{1} - v_0)\circ\Psi
        -
        \bar u\sind{2}}_{H^1(U\sind{2})}
        \\
        &=
        \norm{(v\sind{1} - \bar u \sind{1} + \bar u\sind{1} - u_0 + u_0 - v_0)\circ\Psi
        -
        \bar u\sind{2}}_{H^1(U\sind{2})}
        \\
        &=
        \norm{(v\sind{1} - \bar u \sind{1} + \hat u\sind{2} + u_0 - v_0)\circ\Psi
        -
        \bar u\sind{2}}_{H^1(U\sind{2})}
        \\
        &=
        \norm{(v\sind{1} - \bar u \sind{1}+ u_0 - v_0)\circ\Psi}_{H^1(U\sind{2})}
        \\
        &\leq 
        \norm{(v\sind{1}-\bar u\sind{1})\circ\Psi}_{H^1(U\sind{2})}+
        \norm{(u_0-v_0)\circ\Psi}_{H^1(U\sind{2})}
        \\
        &\lesssim
        \norm{\nabla \Psi^{-1}}_{L^\infty(U\sind{1})}
        \left[
        \norm{v\sind{1}-\bar u\sind{1}}_{H^1(U\sind{1})}+
        \norm{u_0-v_0}_{H^1(U\sind{1})}\right]
    \end{align*}
    using that $\Psi$ is a diffeomorphism.
    Hence, $v\sind{2}$ is a smooth approximation of $\bar u\sind{2}$ as claimed.
   
    \paragraph{Step 4, smoothing with zero trace.}
    To conclude the proof we have to prove that we can approximate the function $u_0 \in H^1(\M\sind{1})$, defined above, which is compactly supported in $U\sind{1}$ and has zero trace on the intersection, by a smooth function $v_0$ which too is  supported in $U\sind{1}$ and has zero trace on the intersection.    
       Furthermore, as noted in Step 1 the intersection $\M\sind{12}$ divides $U_k\sind{1}\setminus \M\sind{12}$ into two components $U\sind{1}_\pm$. 
    Since $u_0$ has zero trace in each of the parts and it can be approximated by smooth functions, $v_0^+$ and $v_0^-$ compactly supported in open sets $U\sind{1}_+$ and $U\sind{1}_-$ respectively. The function $v$ defined to be equal to $v_0^+$ on
    $U\sind{1}_+$, $v_0^-$ on
    $U\sind{1}_-$, and zero on $\M\sind{12}\cap U_k\sind{1}$ is the desired smooth approximation of $u_0$ on $U_k\sind{1}$.
 \end{proof}
    
\subsection{Density of Lipschitz functions in \texorpdfstring{$H^1(\M)$}{H1(M)} when \texorpdfstring{$d\sind{2}-d\sind{12}\geq 2$}{d2-d12 >= 2}} 
\label{sec:appendix_Lip_dens}
  
When $d\sind{2}-d\sind{12} \geq 2$ the Sobolev space $H^1(\M)$ is the space of functions whose restriction to $\M\sind{i}$ belong to  $H^1(\M\sind{i})$ for $i=1,2$. Here we show that the Lipschitz functions on the union, $\M$, are still dense in $H^1(\M)$. This helps bridge the gap to the literature on metric measure spaces where the Dirichlet energy is typically defined on the closure of the space of Lipschitz functions.

\begin{lemma} \label{lem:Lip-dens}
   In the setting of \cref{sec:unions}, assume $d\sind{2}-d\sind{12} \geq 2$.
   We claim that Lipschitz continuous functions on $\M$ are dense in $H^1(\M)$ which we previously characterized in \cref{eq:def_H1mv} as $H^1(\M\sind{1}) \times H^1(\M\sind{2})$.
\end{lemma}
\begin{proof}
Since functions $u \in H^1(\M)$ whose restrictions to $\M\sind{1}$ and to $\M\sind{2}$ are smooth are dense in $H^1(\M)$ it suffices to show that such function $u$ can be approximated in $H^1(\M)$ by Lipschitz continuous functions. Let $u_i$ be the smooth representative of the restriction of $u$ to $\M\sind{i}$ for $i=1,2$. The construction is similar to the one in the proof of limsup results in \cref{sec:NL_limsup,sec:limsup}.

For $n>1$  let $\eps_n=\frac{1}{n}$. 
Let us consider $n$ large enough so that  the orthogonal geodesic projection $\pi : \M\sind{2} \to \M\sind{12}$ onto the intersection is uniquely defined in an $\eps_n$-neighborhood of $\M\sind{12}$. 
We denote the distance of $x\in\M\sind{2}$ to $\M\sind{12}$ by $d(x):= \dist(x,\M\sind{12})$. Let
\begin{align}
    u_n(x) =
    \begin{cases}
        u\sind{1}(x) \quad & \te{if } x \in \M\sind{1}, \\
        u\sind{1}(\pi(x)) \quad & \te{if } x\in \M\sind{2}      \te{ and }        d(x) \leq \eps_n, \\
        u\sind{2}(x) + \frac{\ln \sqrt{\eps_n} - \ln d(x)}{\ln \sqrt{\eps_n} - \ln \eps_n } \left( u\sind{1}(\pi(x)) - u\sind{2}(x) \right) & \te{if } x\in \M\sind{2}      \te{ and }     \eps_n  < d(x) <  \sqrt{\eps_n} \\
        u\sind{2}(x) \quad & \te{if } x\in \M\sind{2} \te{ and }        d(x) \geq \sqrt{\eps_n}.
    \end{cases}
\end{align}
We note that $u_n$ is Lipschitz by construction, since both $\pi$ and $d$ are Lipschitz functions on the regions where they are used. 
It remains to show that $u_n - u\sind{2} \to 0$ in $H^1(\M\sind{2})$. 
Let $A_n = \{ x \in \M\sind{2} \::\: d(x) < \sqrt{\eps_n} \}$. Note that volume of $A_n$ vanishes as $n \to \infty$ while $u_n$ and $u\sind{2}$ are uniformly bounded on $\M\sind{2}$. Thus $u_n - u\sind{2} \to 0$ in $L^2(\M\sind{2})$. 
Note that for $x \in \M\sind{2}$
\begin{align}
    \nabla u_n(x) =
    \begin{cases}
        D\pi(x)^T \nabla u\sind{1}(\pi(x)) \quad & \te{if }       d(x) \leq \eps_n, \\
        \nabla u\sind{2}(x) 
        + \frac{2\nabla  d(x)}{d(x) \ln \eps_n}  \left( u\sind{1}(\pi(x)) - u\sind{2}(x) \right) & \\
        + \frac{\ln \sqrt{\eps_n} - \ln d(x)}{\ln \sqrt{\eps_n} - \ln \eps_n } \left( D\pi(x)^T \nabla u\sind{1}(\pi(x)) - \nabla u\sind{2}(x) \right) 
        & \te{if }    \eps_n  < d(x) <  \sqrt{\eps_n} \\
        \nabla u\sind{2}(x) \quad & \te{if }       d(x) \geq \sqrt{\eps_n}.
    \end{cases}
\end{align}
Showing that $\nabla u_n -\nabla u\sind{2}$ converges to zero in $L^2(\M\sind{2})$ boils down, due to uniform in $x$ boundedness of $D\pi$, $\nabla u\sind{1}$, and $\nabla d$, to showing that 
\[ \frac{1}{(\ln \eps_n)^2 }\int_{\{\eps_n  < d(x) <  \sqrt{\eps_n} \}}  \frac{1}{d(x)^2} \d x \to 0 \qquad \te{as } n \to \infty. \]
This follows by disintegrating over perpendicular slices, and arguing  analogously to \labelcref{eq:aux_log} in each slice, to obtain
\[ \int_{\{\eps_n  < d(x) <  \sqrt{\eps_n} \}}  \frac{1}{d(x)^2} \d x = O(|\ln \eps_n|). \]
\end{proof}

\end{appendix}

\end{document}